\documentclass[11pt]{amsart}
\usepackage[T1]{fontenc}

\usepackage{lmodern, amsfonts,amsmath,amstext,amsbsy,amssymb,
amsopn,amsthm,upref,eucal,mathptmx}

\usepackage{mathrsfs}

\RequirePackage{xcolor} % [dvipsnames]
\definecolor{halfgray}{gray}{0.55}%chapter numbers will be semi
\definecolor{webgreen}{rgb}{0,0.4,0}
\definecolor{webbrown}{rgb}{.8,0.1,0.1} 
\definecolor{red}{rgb}{1,0,0} 
\usepackage{hyperref}
\hypersetup{%
   bookmarks=true,        % show bookmarks bar?
   unicode=false,         % non-Latin characters in Acrobat’s bookmarks
   pdftoolbar=true,       % show Acrobat’s toolbar?
   pdfmenubar=true,       % show Acrobat’s menu?
   pdffitwindow=false,    % window fit to page when opened
   pdfstartview={FitH},   % fits the width of the page to the window
   pdfnewwindow=true,     % links in new PDF window
   colorlinks=true,       % false: boxed links; true: colored links
   linkcolor=webbrown,    % color of internal links (change box color with linkbordercolor)
   citecolor=webgreen,    % color of links to bibliography
   filecolor=magenta,     % color of file links
   urlcolor=cyan,         % color of external links
   pdfcreator={pdfLaTeX}, % creator of the document
   pdfproducer={LaTeX with hyperref}, % producer of the document
   pdftitle={Effective equidistribution of twisted horocycle flows and horocycle maps},   % title
   pdfauthor={Livio Flaminio, Giovanni Forni and James Tanis},      % author
   pdfsubject={2000 MAthematical Subject Classification: Primary: 37-XX, 37C15, 37C40},   % subject of the document
   pdfkeywords={Cohomological Equations} {Homogeneous flows} {Effective Equidistribution} % list of keywords
 }  
\usepackage{microtype}

\newcommand \R {{ \mathbb R}}
\def\C{{\mathbb C}}
\newcommand \Z {{ \mathbb Z}}
\newcommand \N {{ \mathbb N}}
\newcommand \T {{ \mathbb T}}

\newcommand \re {{ \operatorname{Re} }}
\newcommand \im {{ \operatorname{Im} }}
\newcommand{\SL}{\operatorname{SL}}

\renewcommand{\sl}{\operatorname{\mathfrak s\mathfrak l}}
\newcommand{\vol} {{\operatorname{vol} }}

\newtheorem{theorem}{Theorem}[section]
\newtheorem {lemma} [theorem]{Lemma}
\newtheorem {proposition}[theorem]{Proposition}
\newtheorem{corollary}[theorem]{Corollary}
\newtheorem{remark}[theorem]{Remark}

\newtheorem{definition}[theorem]{Definition}

\title[Effective equidistribution of twisted horocycle flows]%
{Effective equidistribution of twisted horocycle flows and horocycle
  maps} 
\author{Livio Flaminio} 
\thanks{L.~Flaminio was supported in
 part by the Labex~CEMPI}
\address[Livio Flaminio]{Unit\'e Mixte de Recherche CNRS 8524 \\
 Unit\'e de Formation et Recherche de Math\'ematiques\\
 Universit\'e de Lille 1\\
 F59655 Villeneuve d'Asq Cedex, FRANCE}
\email{livio.flaminio@math.univ-lille1.fr} 
\author{Giovanni Forni}
\thanks{G.~Forni was supported by the NSF grant DMS~1201534 and by a
 Simons Fellowship.}
\address[Giovanni Forni]{
 \noindent  Department of Mathematics\\
 University of Maryland\\ \\College Park\\ MD~20742, U.S.A 
}
\email{gforni@math.umd.edu} 
\author{James Tanis} 
\address[James Tanis]{Coll\`ege de
 France\\ 3 Rue d'Ulm\\ 75005 Paris,~France}
\email{james.tanis@college-de-france.fr} 
\thanks{J.~Tanis was
 partially supported by the ANR grant 'GeoDyM' (ANR-11-BS01-0004).}


\begin{document}
        
\begin{abstract} 
  \begin{sloppypar} We prove bounds for twisted ergodic averages for
    horocycle flows of hyperbolic surfaces, both in the compact and in
    the non-compact finite area case. From these bounds we derive
    effective equidistribution results for horocycle maps. As an
    application of our main theorems in the compact case we further
    improve on a result of A.~Venkatesh, recently already improved by
    J.~Tanis and P.~Vishe, on a sparse equidistribution problem for
    classical horocycle flows proposed by N.~Shah and G.~Margulis, and
    in the general non-compact, finite area case we prove bounds on
    Fourier coefficients of cups forms which are off the best known
    bounds of A.~Good only by a logarithmic term. Our approach is
    based on Sobolev estimates for solutions of the cohomological
    equation and on scaling of invariant distributions for twisted
    horocycle flows.
\end{sloppypar}
\end{abstract}

\maketitle

\tableofcontents

\section{Introduction}

In this paper we prove bounds for \emph{twisted ergodic averages} for
horocycle flows of hyperbolic surfaces, both in the compact and in the
non-compact finite area case. From these bounds we derive effective
equidistribution results for horocycle maps.  We also improve upon a
result of A.~Venkatesh \cite{V} on a question of N.~Shah~\cite{Sh} on
a sparse equidistribution problem for the horocycle flow. Finally,
from our estimates in the non-compact finite area case we derive
bounds for the Fourier coefficients of automorphic forms which
coincide with the best known bounds due to A.~Good \cite{Go} up to
logarithmic terms.

Our main results can be stated as follows.  Let $\{h_t\}$ denote the
stable horocycle flow on the unit tangent bundle $M$ of a finite area
hyperbolic surface $S$ of constant curvature $-1$.  Let $\{a_t\}$ be
the geodesic flow on $M$ and let $\text{dist}$ be the distance
function on $M \times M$ determined by a metric on $M$ for which the
orbits of $\{a_t\}$ are geodesics (the choice of this metric will
appear in Section~2).
  
Let $x_0$ be a fixed point in $M$.  For any $x \in M$, let $d_M(x) :=
\text{dist}(x, x_0)$.  For $A\in [0, 1]$ and $Q > 0$, let us define
subsets of ``Diophantine points''
\begin{equation}
  \label{eq:M_AQ}
  M_{A, Q} := \left\{x \in M :  d_M(a_t(x)) \leq A t  + Q \text{ for all } t > 0\right\}\,.
\end{equation}
When $M$ is compact, there is some $Q > 0$ such that $M_{0, Q} = M$.
More generally, $\cup_{Q > 0} M_{1, Q} = M$, and by the logarithm law
of geodesics, for any $A \in (0, 1]$, the set $ \cup_{Q > 0} M_{A, Q}
$ has full Haar measure, see \cite{Su}.  In fact, for almost all $x\in
M$ we have
\begin{equation}\label{equa:log-law-geodesics}
  \limsup_{t \to +\infty}  \frac{ d_M(a_t(x))}{\log |t|} = \frac{1}{2}\,.
\end{equation}

For all $s\in \R$, let $W^s(M)$ denote the  Sobolev space of
square-integrable functions on the unit tangent bundle $M$ with respect
to the normalized volume measure.

\begin{theorem}
  \label{thm:Main_Twisted}
  For every $s > 7$ and for every $(A, Q) \in [0, 1) \times \R^+$
  there is a constant $C_{s, A, Q} := C_{s, A, Q}(M) > 0$ such that
  the following bounds hold: for every $\lambda\in \R\setminus\{0\}$,
  for every $(x,T) \in M\times \R^+$ such that $x$ and $h_T(x) \in
  M_{A,Q}$ and $\vert \lambda T \vert \geq e$, for every zero-average
  function $f\in W^s(M)$, we have
  \begin{equation}
    \label{eq:Main_Twisted}
    \vert \int_0^T  e^{i\lambda t} f \circ h_t (x) dt\vert  \leq  C_{s, A, Q} \Vert f \Vert_s  
    (1+ \frac{ \vert \lambda\vert^{\frac{2 A}{1 - A} }}{ \vert \lambda\vert^{1/6}})
    T^{5/6 + \frac{2 A}{1 - A}} \log^{1/2}(\vert \lambda T\vert ) \,.
  \end{equation}
  By the logarithm law of geodesics, for all $\epsilon >0$ there
  exists a measurable function $C_{s,\epsilon}: M \to \R^+$, finite
  almost everywhere, such that, if $\vert \lambda T \vert \geq e$,
  \begin{equation}
    \label{eq:Main_LogLaw}
    \begin{aligned}
      \vert \int_0^T e^{i\lambda t} f \circ h_t (x) dt\vert &\leq
      C_{s,\epsilon} (x) C_{s,\epsilon} (h_T(x)) \Vert f \Vert_s \\ &
      \times (1+ \frac{1}{\vert \lambda\vert^{1/6} }) T^{5/6}
      \log^{3/2+\epsilon}(\vert \lambda T\vert ) \,.
    \end{aligned}
  \end{equation}
  Moreover, when $M$ is compact or when $M$ is non-compact but $x$
  belongs to a (closed) cuspidal horocycle of length $T\geq 1$ such
  that $ \lambda T \in 2\pi \Z$, there exists a constant
  $C_s:=C_s(M)>0$ such that, if $\vert \lambda T \vert \geq e$,
  \begin{equation}
    \label{eq:Main_Twisted_Compact}
    \vert \int_0^T  e^{i\lambda t} f \circ h_t (x) dt\vert  \leq  C_s \Vert f \Vert_s  
    (1+ \frac{1}{ \vert \lambda\vert^{1/6}})  T^{5/6 } \log^{1/2}(\vert \lambda T\vert ) \,.
  \end{equation}
\end{theorem}
Bounds in the regime $\vert \lambda T\vert \leq e$ can immediately be
derived by integration by parts from bounds on ergodic integrals of
the horocycle flow, which are well-known (see for instance \cite{Bur},
\cite{FF1}, \cite{St2}). In fact,
\[
\int_0^T e^{i\lambda t} f \circ h_t (x) dt = e^{i\lambda T} \int_0^T f
\circ h_t (x) dt - i \lambda \int_0^T e^{i\lambda t} \int_0^t f \circ
h_\tau (x) d\tau \,,
\]
whenever $\vert \lambda T\vert \leq e$ we have the bound
\[
\vert \int_0^T e^{i\lambda t} f \circ h_t (x) dt\vert \leq \vert
\int_0^T f \circ h_t (x) dt \vert + \frac{e}{T} \int_0^T \vert
\int_0^t f\circ h_\tau(x) d\tau \vert dt \,.
\]
We remark that a different integration by parts, which this time
exploits the cancellations given by the (fast) oscillations of the
exponential function, implies that for $f\in C^1(M)$ in the regime
$\vert \lambda T\vert \geq T^{1+\beta}$ we have, for all $x\in M$ and
for all $T\geq 1$,
\[
\vert \int_0^T e^{i\lambda t} f \circ h_t (x) dt \vert \leq \Vert
f\Vert_{C^1(M)} \frac{2+T}{\lambda} \leq 3 \Vert f\Vert_{C^1(M)}
T^{1-\beta}.
\]
Theorem~1.5 of \cite{FF1} provides a precise asymptotics for large
$T>0$ of the ergodic integral $\int_0^{T} f \circ h_t(x) dt$ in terms
of invariant distributions for the horocycle flow. A~refinement of the
asymptotics, in terms of finitely additive measures on horocycles, is
developed in \cite{BuFo} with applications to limit probability
distributions for horocycle flows. In particular, in the compact case
the following result holds.  Let $\mu_0>0$ denote the smallest
non-negative eigenvalue of the Laplace-Beltrami operator $\Delta_S$ of
the hyperbolic surface $S$ and let
\[
\nu_0:= \begin{cases} \sqrt{1- \mu_0} \quad &\text{ if } \, \mu_0\leq
  1 \,; \\ 0 \quad &\text{ if } \, \mu_0> 1 \,.
\end{cases}
\]
For every $s>3$ there exists a constant $C_s>0$ such that for every
zero-average function $f\in W^s(M)$, for all $(x,T)\in M\times \R^+$,
we have (see also \cite{Bur})
\[
\vert \int_0^T f \circ h_t (x) dt\vert \leq C_s \Vert f \Vert_s (1+
T^{\frac{1+\nu_0}{2}} + T^{1/2} \log (e + T))\,.
\]
The equidistribution of horocycle flows on surfaces of constant
negative curvature was proved by H.~Furstenberg \cite{Fu} in the
compact case and by Dani \cite{Da} in the non-compact finite area
case. In the first case the horocycle flow is uniquely ergodic, while
in second case all orbits equidistribute except for finitely many
one-parameter families of closed (cuspidal) horocycles.

The effective equidistribution, that is, bounds on the speed of
convergence of ergodic averages, for horocycle flows of hyperbolic
surfaces (the case $\lambda=0$ in Theorem~\ref{thm:Main_Twisted}) has
been investigated thoroughly in the past decades, both for compact and
non-compact, finite area surfaces \cite{Za}, \cite{Sa}, \cite{Bur} (in
the general geometrically finite case), \cite{Hj}, \cite{FF1},
\cite{St1}, \cite{St2}, \cite{BuFo} (which proves results on limit
distributions of ergodic integrals in the compact case). All these
results indicate that in general the speed of convergence of ergodic
averages of sufficiently smooth functions depend on the \emph{spectral
  gap} of the Laplace-Beltrami operator of the surface. In fact, a
rather complete asymptotics for ergodic averages of smooth functions
was established in \cite{FF1} and later refined in \cite{BuFo}, where
results on limit distributions of probability distributions given by
ergodic integrals were derived.

\smallskip A striking feature of our effective equidistribution result
(Theorem~\ref{thm:Main_Twisted}) in the regime $\lambda T \geq e$ is
its independence from the spectral properties of the Laplace-Beltrami
operator (``spectral gap'').  To the best of our knowledge this
phenomenon was first conjectured by Venkatesh (the third author of
this paper learned of this conjecture directly from A.~Venkatesh in
Spring 2012, the second author from P.~Vishe in Spring 2014). Indeed,
the first effective bounds on twisted ergodic integrals of horocycle
flows proved by A.~Venkatesh~\cite{V} were not uniform with respect to
the spectral gap. Recently, in work developed in parallel with this
paper, the third author and P.~Vishe have refined Venkatesh method
thereby proving a bound independent of the spectral gap \cite{TV}. Our
method is completely different from Venkatesh's approach in \cite{V}
(refined by the third author and P.~Vishe in \cite{TV}), which is
based on effective equidistribution \cite{Bur}, \cite{FF1} and
estimates on decay of correlations (see for instance \cite{Ra}) for
horocycle flows, and our bounds are somewhat better (for instance for
$\lambda$ small the bound of the third author and P.~Vishe~\cite{TV}
for twisted integrals is of the form $\lambda^{-1/2} T^{8/9}$).

\smallskip Recent results on the effective equidistribution of
horocycle maps by Venkatesh \cite{V} in the compact case and by
P.~Sarnak and A.~Ubis~\cite{SU} for the modular surface have been
motivated by \emph{sparse equidistribution} problems for the horocycle
flow, that is, by the questions whether the horocycle flow still
equidistributes when sampled along a polynomial sequence of times (of
fractional degree larger than $1$) or along the prime numbers. The
first question, which appears in the work of N.~Shah \cite{Sh}, asks
whether the horocycle flow equidistributes along polynomial sequences
of any degree. We recall that by a general pointwise ergodic theorem
proved by J.~Bourgain~\cite{Bou} the answer is affirmative for almost
all points, along time sequences given by polynomials of any degree
with integer coefficients. The second question comes up in the work of
P.~Sarnak and A.~Ubis \cite{SU} on the independence of the M\"obius
functions with respect to all sequences generated by the horocycle
flows on the modular surface. Both questions have also been asked by
G.~Margulis for general unipotent flows~\cite{Ma}.

In Venkatesh's work \cite{V} the \emph{effective equidistribution of
  horocycle maps} is derived from the effective equidistribution by a
direct argument based on Fourier expansion of delta measures on the
line. We retain in this paper the approach of the third author's
thesis \cite{T} which consists in combining effective equidistribution
results for the twisted cohomological equation with the complete
description of invariant distributions and the solution to the
cohomological equation for horocycle maps.  Our main result can be
stated as follows.

\begin{theorem}
  \label{thm:Main_Maps}
  For every $s > 14$ and $\epsilon > 0$, and for every $(A, Q) \in [0,
  1) \times \R^+$, there is a constant $C_{s, \epsilon, A, Q} := C_{s,
    \epsilon, A, Q}(M) > 0$ such that the following holds.  For every
  $L >0$, for every $(x,N) \in M \times \N\setminus\{0\}$ such that
  $x$ and $h_{NL}(x) \in M_{A,Q}$ and for every $f\in W^s(M)$, we have
  \begin{equation}\label{equa:main_maps1}
    \begin{aligned}
      \vert \sum_{k=0}^{N-1} &f \circ h_{Lk}(x) -\frac{1}{L}
      \int_0^{NL} f\circ h_t (x) dt \vert
      \leq C_{s, \epsilon, A, Q}\Vert f\Vert_s \\
      & \times \left( (1+ L^{1/6+\epsilon}) (NL)^{5/6 +\frac{2A}{1 -
            A}} \log^{1/2} N + \frac{1+ L^{6+A+\epsilon}}{L}\right)
      \,.
    \end{aligned}
  \end{equation}
  By the logarithm law for geodesics, for all $\epsilon > 0$ there
  exists a measurable function $C_{s,\epsilon}: M \to \R^+$ that is
  finite almost everywhere and satisfies
  \begin{equation}\label{equa:main_maps2}
    \begin{aligned}
      \vert \sum_{k=0}^{N-1}& f \circ h_{Lk}(x) -\frac{1}{L}
      \int_0^{NL} f\circ h_t(x) dt \vert \leq
      C_{s, \epsilon}(x) C_{s, \epsilon}(h_{NL}(x)) \Vert f\Vert_s  \\
      & \times \left( (1+ L^{1/6+\epsilon}) (NL)^{5/6}
        \log^{3/2+\epsilon}N + \frac{1+ L^{6+\epsilon}}{L} \right).
    \end{aligned}
  \end{equation}
  In addition, there exists a constant $C_{s,\epsilon}>0$ such that
  for all $(x,N)\in M\times \N\setminus\{0\}$, whenever $f\in W^s(M)$
  is a coboundary for the time-$L$ horocycle map $h_L$ we have
  \begin{equation}\label{equa:main_maps3}
    \begin{aligned}
      \vert \sum_{k=0}^{N-1} f \circ h_{Lk}(x)\vert \leq
      C_{s,\epsilon} \Vert f\Vert_s \frac{1+ L^{2+\epsilon}}{L}
      (e^{d_M(h_{-\frac{L}{2}}(x))} +
      e^{d_M(h_{L(N-\frac{1}{2})}(x))}).
    \end{aligned}
  \end{equation}
  Finally, when $M$ is compact there exists a constant
  $C_{s,\epsilon}:=C_{s,\epsilon}(M)>0$ such that for every $(x,N) \in
  M \times \N\setminus\{0\}$ and for every $f\in W^s(M)$ we have
  \begin{equation}\label{equa:main_maps4}
    \begin{aligned}
      \vert \sum_{k=0}^{N-1} f \circ h_{Lk}(x) -&\frac{1}{L}
      \int_0^{NL} f\circ h_t (x) dt \vert \leq C_{s, \epsilon}\Vert
      f\Vert_s \\ & \times \left( (1+ L^{1/6+\epsilon}) (NL)^{5/6}
        \log^{1/2}N + \frac{1+ L^{6+\epsilon}}{L}\right) \,.
    \end{aligned}
  \end{equation}

\end{theorem}

\smallskip In our paper we also derive from the bounds of Theorem
\ref{thm:Main_Maps} on the ergodic sums of horocycle maps the
following result on Shah's question.
 
 \begin{theorem}\label{thm:Main_Shah}
   Let $M$ be compact.  For all
   $0 \leq \delta < \delta_0=1/13$, for all $f \in C(M)$ and for all $x \in M$, we have
   \[
   \lim_{N \to \infty} \frac{1}{N} \sum_{n= 0}^{N - 1} f \circ h_{n^{1+ \delta}}(x) = \int_M f d\vol\,.
   \]
 \end{theorem}

 The above result improves upon the corresponding results by Venkatesh
 in \cite{V} (who had an explicit $\delta_0\in (0, 1/48)$ but no
 uniform estimate with respect to the ``spectral gap'') and by the
 third author and Vishe \cite{TV} (who had $\delta_0 = 1/26$). Our
 argument essentially follows Venkatesh's, so our improvement in the exponent 
 derives from our better bounds for twisted ergodic integrals. The main idea is to
 approximate polynomial sequences with arithmetic progressions with a
 small error for sufficiently long times. By its very nature this
 approach cannot go beyond the threshold $\delta_0=1$, so in any event
 a complete answer to Shah's question is very far out of reach of
 current methods.

 \smallskip The work of P.~Sarnak and A. Ubis~\cite{SU} focuses on the
 case of horocyle maps on the modular surface, for which they prove an
 effective version of Dani's equidistribution theorem for unipotent
 maps. However, this is far from enough to control the distribution of
 the horocycle flow ``at prime times'', as the equidistribution of
 one-parameter unipotents on the product space (that is, of the
 joinings of the horocycle flow) plays a crucial role.  This work has
 provided a crucial motivation for us in developing our scaling
 approach to twisted ergodic integrals, an approach that we hope can
 be applied to more general problems of effective distributions for
 unipotent flows, including perhaps joinings of the horocycle flow.
 Their approach is based on approximation of orbits of horocycle flows
 by segments of orbits on cuspidal horocycles, whose equidistribution
 is derived from bounds on the Fourier coefficients of automorphic
 forms. This approach is therefore limited to non-compact surfaces of
 finite area and it seems to fall rather short of the optimal exponent
 in the effective Dani's theorem.

 \smallskip In our paper the case of non-compact, finite area surfaces
 can be handled by the same method as the compact case, by taking into
 account the speed of escape of geodesic orbits into the cups.  By
 writing Fourier coefficients of cusp forms as twisted ergodic
 integrals along cuspidal horocycles, we derive bounds which coincide
 up to a logarithmic factor with the best bounds available, due to
 A.~Good~\cite{Go}, for general, possibly non-arithmetic, non
 co-compact lattices of $\SL(2,\R)$. This coincidence seems to
 indicate that our results are presumably rather hard to improve upon,
 at least as far as the exponent of the polynomial bound is concerned.

 \begin{corollary}\label{coro:Cusp-form} Let $\{a_n\} \subset \C$ denote the sequence of the Fourier
   coefficients of a holomorphic cusp form $f$ of even integral
   weight-$k$ for any non co-compact lattice $\Gamma\subset SL(2,
   \R)$. There is a constant $C_f > 0$ such that for all $n\in
   \N\setminus\{0\}$ we have
   \[
   |a_n| \leq C_f n^{k/2 - 1/6} (1+\log n)^{1/2}\,.
   \]
 \end{corollary}
 We recall that for the modular lattice, and more generally for
 congruence lattices, the Ramanujan-Petersson conjecture, proved by
 Deligne, states that the sharp bound $|a_n| \leq C_{f, \epsilon}
 n^{k/2 - 1/2+\epsilon}$ holds. To the authors best knowledge it is an
 open question whether the optimal bound holds for general non
 co-compact lattices or even whether Good's bound can be improved.

 \smallskip We conclude the introduction with an outline of the
 methods of our paper. We recall that the main idea underlying all
 advances on the effective equidistribution of the horocycle flow is
 the fundamental fact that the orbit foliation of the horocycle flow
 is invariant under the geodesic flow, hence the horocycle flow is
 renormalized by the geodesic flow. In other terms, such results
 establish refined versions of the exponential decay of correlation
 for the geodesic flow.  In our paper, we follow the approach first
 developed by the first two authors to prove effective ergodicity
 results for higher step nilflows \cite{FF3}. We view the twisted
 ergodic integrals for the horocycle flow as special ergodic integrals
 for the product of the horocycle flow and of a linear flow on a
 circle. Our goal thus becomes to prove effective ergodicity results
 for the above product flow. Our proof of effective equidistribution
 is based on a scaling argument which is a generalization of the
 renormalization method developed in the work of the first two authors
 to prove effective equidistribution of horocycle flows~\cite{FF3}. It
 consists in an analysis, based on the theory of unitary
 representations for the group $\SL(2,\R)$ of a scaling operator on
 the space of \emph{invariant distributions} for the appropriate
 \emph{cohomological equation}. In the present case the scaling is not
 induced by a renormalization dynamics. The exponent in our effective
 equidistribution theorem is the optimal scaling exponent of invariant
 distributions for the twisted horocycle flow. The logarithmic factor
 arises from the control of the geometry of the rescaled metric
 structure or, equivalently, from estimates related to close return
 times of the horocycle flow. The relevant measure of the degeneration
 of the geometry is a notion of injectivity radius, called the
 \emph{average width} of a horocycle arc, already introduced in
 \cite{FF3} for nilflows.

 \section{Statement of results}\label{sect:results}

 Let $\Gamma < \SL(2, \R)$ be any lattice. The group $\SL(2, \R)$ acts
 by right multiplication on the quotient manifold $M:=\Gamma
 \backslash \SL(2, \R)$. The Haar measure on $\SL(2, \R)$ induces a
 right-invariant volume form $\operatorname{vol}$ on $M$ which we
 normalize so that $\operatorname{vol}(M) = 1$.  We recall that
 $(X,U,V)$ is the basis of $\sl_2(\R)$ given by
 \[
 X= \begin{pmatrix} {1}&0\\0& {-1}
 \end{pmatrix}, \quad U=\begin{pmatrix} 0&1\\0& 0
 \end{pmatrix}, \quad V=\begin{pmatrix} 0&0\\1& 0
 \end{pmatrix}.
 \]
 These matrices satisfy the commutation relations
 % \begin{equation}\label{equa:commutation}
 \[
 [X,U]= 2U , \quad [X,V]= -2V , \quad [U,V]= X \,;
 \]
 The center of the enveloping algebra of the Lie algebra $\sl_2(\R)$
 is one-dimensional and is generated by the \emph{Casimir operator}
 \begin{equation}
   \label{eq:Casimir}
   \Box = -X^2 -2 (UV + VU)\,.
 \end{equation}
 The Casimir operator commutes with the action of the full group
 $\SL(2,\R)$ on the enveloping algebra, hence the differential
 operator induced by the Casimir operator on any unitary
 representation is $\SL(2,\R)$-invariant.

 \smallskip

 The flows on $M$ defined, for all $x\in M$ and $t \in \R$, by the
 formulas
 \[
 a_t(x) = x \exp (tX/2),\qquad h_t(x) = x\exp (tU),\qquad \bar h_t(x)
 =x \exp (tV)\,,
 \]
 are, by definition, the geodesic, stable and unstable horocycle flow,
 respectively. By ``horocycle flow'' we shall mean the stable
 horocycle flow $\{h_t\}$ generated by the vector field $U$ on $M$.
 Our ``twisted horocycle flows'' will be product flows on $M\times \T$
 of the horocycle flow on $M$ with with a linear flow on the circle
 \[
 \T := \R / 2 \pi \mathbb Z.
 \]

 Let $L^2(M \times \T)$ be the space of complex-valued,
 square-integrable functions on $M \times \T$.  Let $K$ denote the
 vector field on $\T$ that is given by ordinary differentiation. The
 Laplace-Beltrami operator on $M \times \T$ is an elliptic second
 order operator.  It is a non-negative essentially self-adjoint
 operator on $L^2(M \times \T)$ and is denoted by
 \[
 \triangle := -K^2 - X^2-2(U^2+V^2)\,.
 \]

 For any $s > 0$, the operator $(I + \triangle)^{s/2}$ is defined by
 the spectral theorem.  The \textit{Sobolev space} $W^s(M \times \T)
 \subset L^2(M \times \T)$ is defined to be the maximal domain of $(I
 + \triangle)^{s/2}$ on $M \times \T$ and is endowed with the inner
 product
 \[%\begin{equation}\label{equa:inner-product_MT}
 \langle F, G\rangle_{W^s(M\times \T)} := \langle (I + \triangle)^s F, G \rangle_{L^2(M\times \T)}\,. 
 \]%\end{equation}
 We denote the corresponding norm by
 \[
 \| F \|_s := \langle F, F \rangle_{W^s(M\times \T)} ^{1/2}\,.
 \]

 The distributional dual space to $W^s(M \times \T)$ is denoted
 \[
 W^{-s}(M \times \T) := \left(W^s(M \times \T)\right)'
 \]
 with norm denoted $\|\cdot \|_{-s}$.  The space of smooth vectors in
 $L^2(M \times \T)$ and its distributional dual space are denoted
 \begin{equation}\label{equa:def-smooth}
   \begin{array}{lll}
     W^\infty(M \times \T) & := \cap_{s \geq 0} W^s(M \times \T), \text{ and }\\
     W^{-\infty}(M \times \T) & := \cup_{s \geq 0} W^{-s}(M \times \T)\,,
   \end{array}
 \end{equation}
 respectively.  We remark that when $M$ is compact, then $W^\infty(M
 \times \T) = C^\infty(M \times \T)$ and $W^{-\infty}(M \times \T) =
 \mathcal D'(M \times \T)$.

 Let $W^s(M)$ be the subspace of functions in $W^s(M \times \T)$ that
 are constant with respect to the natural circle action on
 $M\times\T$, which is endowed with the same inner product $\langle
 \cdot, \cdot \rangle_s$.  Let $C^\infty(M)$ and $\mathcal E'(M)$ be
 defined as in \eqref{equa:def-smooth}.

 Our results will often be proven using Sobolev norms involving only
 the $K$, $X$ and $V$ derivatives and the Casimir operator.  Such
 \emph{foliated Sobolev spaces} are defined as follows.  The Casimir
 operator $\Box$ is an essentially self-adjoint operator on $L^2(M)$,
 hence on $L^2(M\times \T)$.  Also notice that the foliated Laplacian
 \[
 \widehat \triangle := - K^2 - X^2 - V^2
 \]
 is a non-negative essentially self-adjoint differential operator on
 $L^2(M \times \T)$.  For any $r, s \geq 0$, let $\widehat
 W^{r,s}(M\times \T)$ be the Sobolev space that is the maximal domain
 of $ (I + \Box^2)^{r/2}(I + \Box^2 +\widehat \triangle^2)^{s/2}$ on
 $L^2(M \times \T)$ with inner product
 \[
 \langle F, G \rangle_{\widehat W^{r, s}(M\times \T)} : = \langle (I +
 \Box^2)^{r/2} (I + \Box^2 + \widehat \triangle^2)^{s/2} F,
 G\rangle_{L^2(M \times \T)}\,.
 \]
 Let us define the norm on $\widehat W^{r,s}(M\times \T)$ to be
 \[
 \vert F \vert_{r, s} := \langle F, F \rangle_{\widehat
   W^{r,s}(M\times \T)}^{1/2}\,.
 \]
 \vspace{.2cm} The dual space of $\widehat W^{r,s}(M\times \T)$ is
 denoted $\widehat W^{-r,-s}(M\times \T)$ with norm
 $\vert~\cdot~\vert_{-r, -s}$.

\subsection{Twisted horocycle flows}
For any $\lambda \in \R^*$, the \textit{twisted horocycle flow} is the
flow $\{\phi_t^\lambda\}_{t \in \R}$ on $M \times \T$ generated by the
vector field $\mathcal (U + \lambda K).$ The main theorems of this
paper concern the quantitative equidistribution of this flow, and its
applications.  As in \cite{FF1}, \cite{FF2}, \cite{FF3} and \cite{T},
our analysis is based on invariant distributions and on bounds for
solutions of a cohomological equation, and it is carried out in
irreducible, unitary representation spaces.

Because the rate equidistribution of the horocycle flow has been
completely understood in \cite{FF1}, we restrict our attention to $(U
+ \lambda K)$-invariant distributions that are not $U$-invariant.  We
denote these spaces of infinite-order distributions by
\[
\mathcal I_\lambda(\Gamma) := \left\{\mathcal D \in \mathcal E'(M
  \times \T) : (U + \lambda K) \mathcal D = 0, \text{ and } U \mathcal
  D \neq 0\right\}\,.\] The space $\mathcal I_\lambda(\Gamma)$ is
filtered by subspaces of invariant distributions of finite order. For
all $r, s \geq 0$ we denote
\[
\begin{array}{ll}
  \mathcal I_\lambda^s(\Gamma) & := \left\{\mathcal D \in W^{-s}(M \times \T) : 
    (U + \lambda K) \mathcal D = 0, \text{ and } U \mathcal D \neq 0\right\}\,; \\
  \mathcal I_\lambda^{r, s}(\Gamma) & := \left\{\mathcal D \in \widehat W^{-r,-s}(M \times \T)  : 
    (U + \lambda K) \mathcal D = 0, \text{ and } U \mathcal D \neq 0\right\}\,.
\end{array}
\]
If $\lambda \notin \mathbb Z$ and $\mathcal D \in \mathcal
I_\lambda(\Gamma)$, then $U D \neq 0$.

\begin{theorem}\label{theo:Invariant_dist}
  For all $\lambda \in \R^*$ the space $\mathcal I_\lambda(\Gamma)$ of
  invariant distribution for the twisted horocycle flow which are not
  horocycle flow invariant is an infinite dimensional subspace of the
  foliated Sobolev space $\widehat W^{0, -(1/2 +)}(M \times \T)$.
\end{theorem}
For $f \in C^\infty(M \times \T)$, invariant distributions appear in
the study of solutions of the cohomological equation for the twisted
horocycle flow:
\begin{equation}\label{equa:coeqn12}
  (U + \lambda K) g = f\,.
\end{equation}
Let
\[
\text{Ann}_\lambda(\Gamma) := \left\{f \in C^\infty(M \times \T) :
  \mathcal D(f) = 0 \text{ for all } \mathcal D \in \mathcal
  I_\lambda(\Gamma)\right\}\,.
\]

\begin{theorem}\label{theo:cohomological_eqn}
  For every function $f \in \text{Ann}_\lambda(\Gamma)\subset
  C^\infty(M \times \T)$, there is a solution $g \in C^\infty(M \times
  \T)$ of the twisted cohomological equation
  \[
  (U + \lambda K) g = f\,,
  \]
  satisfying the following Sobolev estimates. For all $r$, $s \geq 0$,
  there is a constant $C_{r,s}:= C_{r,s}(\Gamma) > 0$ such that with respect to the
  foliated Sobolev norms
  \[
  \vert g \vert_{r,s} \leq \frac{C_{r,s}}{|\lambda|} (1+\vert \lambda
  \vert^{-s}) \vert f \vert_{r +3s, s +1}\,,
  \]
  hence, for all $s \geq 0$ there is a constant $C_{s}:= C_{s}(\Gamma)> 0$ such that
  with respect to the full Sobolev norms
  \[
  \| g \|_{s} \leq \frac{C_s}{|\lambda|} (1+\vert \lambda \vert^{-s})
  \| f \|_{4s +1}\,.
  \]
\end{theorem}

From Theorem~\ref{theo:Invariant_dist} and
Theorem~\ref{theo:cohomological_eqn} together with a trace theorem and
a quantitative analysis of the returns of the horocycle flow we will
derive effective equidistribution results for the twisted horocycle
flow $\{\phi_t^\lambda\}_{t\in \R}$ on $M \times \T$.

Let $\overline{x} \in M \times \T$ and $T \geq 1$.  Let us consider
the ergodic integral
\begin{equation}\label{equa:Twisted integral1}
  \frac{1}{T} \int_0^{T} (\phi_t^\lambda (\bar x))^* dt\,,
\end{equation}
as a distribution in $\widehat W^{0,-(1+)}(M \times \T)$, whose
regularity follows from a trace theorem discussed in
Subsection~\ref{subs:trace-transfer}.

On smooth functions which are constant with respect to the natural
circle action on $M\times \T$ (on the second factor), the ergodic
integral of the twisted horocycle flow restricts to the ergodic
integral for the horocycle flow on $M$.  The rate of equidistribution
of the horocycle flow has been completely understood in \cite{FF1}, so
we consider functions such that $\int_{\T} F = 0$, that is, functions
which are in the orthogonal complement of the subspace of functions
invariant under the above circle action.

We prove an effective equidistribution theorem for the twisted
horocycle flow on finite volume manifolds $M$ under some Diophantine
conditions.  In Section~\ref{sect:avg-width} we introduce a function
$C_\Gamma: M\times \R^+ \to \R^+\cup\{+\infty\}$ whose growth reflects
the Diophantine properties of points. In particular, for every $A\in
[0,1)$ and every $Q >0$ let $M_{A,Q} \subset M$ denote the set of
Diophantine points introduced in formula~\eqref{eq:M_AQ}. The function
$C_\Gamma$ will satisfy the following properties.  There exists a
constant $C_{\Gamma,A,Q}>0$ such that for all $x\in M_{A,Q}$ and all
$T\geq 1$,
\[
C_{\Gamma}(x,T) \leq C_{\Gamma,A,Q} T^{\frac{2A}{1-A}} \,.
\]

By the logarithm law of geodesics, for almost all $x\in M$, for all
$\epsilon >0$ there exists a constant $C_\epsilon(x)>0$ such that for
all $T\geq 1$,
\begin{equation}\label{equa:Wlog-law}
  C_{\Gamma}(x,T) \leq C_{\epsilon}(x)  (1+ \log^{1+\epsilon} T)\,.
\end{equation}
In addition, if $M$ is compact then there is a constant $C_{\Gamma}
>1$, depending only on $\Gamma$, such that then
\[
C_\Gamma(x,T) \leq C_\Gamma\,.
\]
\begin{theorem}\label{theo:equidistribution}
  For every $s > 2$ and $r\geq 5s-3$, there is a constant $C_{r,s}:=C_{r,s}(\Gamma) >
  0$ such that the following holds. For any $\bar x=(x,\theta) \in M
  \times \T$, and for any $T \geq e $, there are distributions
  $\mathcal D_{\bar x, \lambda, T}^{r,s}\in \mathcal
  I_\lambda^{r,s}(\Gamma)$ and $\mathcal R_{\bar x, \lambda, T}^{r,s}
  \in \widehat W^{-r,-s}(M)^\bot \subset \widehat W^{-r,-s}(M \times
  \T)$, orthogonal to $\mathcal I_\lambda^{r, s}(\Gamma)$ in $\widehat
  W^{-r,-s}(M)^\bot$, such that for any $F \in \widehat W^{r,s}(M
  \times \T)$ satisfying $ \int_{\T} F = 0, $ we have
  \begin{equation}\label{equa:ergodic integral decomposition2}
    \int_0^{T} F \circ \phi_t^\lambda (\bar x) dt = 
    \mathcal D_{\bar x, \lambda, T}^{r,s}(F) T^{5/6}  + 
    \mathcal R_{\bar x, \lambda, T}^{r,s}(F) \,,
  \end{equation}
  where the following bounds hold:
  \begin{equation}\label{equa:Remainder-inv-dist-twist_flow}
    \begin{aligned}
      \vert R_{\bar x, \lambda, T}^{r,s}\vert_{-r,-s}^2 & \leq C_{r,s}
      \frac{1 + |\lambda|^{-2(s - 1)}}{|\lambda|^2} [C_\Gamma(x,T) +
      C_\Gamma(h_T(x),T)]^2 \,; % \log^{1/2} T
      \\
      \vert \mathcal D^{r,s}_{\bar x, \lambda, T} \vert_{-r,-s}^2 &
      \leq C_{r,s} (1 + |\lambda|^{-8s}) [C_\Gamma(x,T) +
      C_\Gamma(h_T(x),T)]^2 \log T \,.
    \end{aligned}
  \end{equation}
\end{theorem}

\begin{remark}
  The above estimates are independent of the spectral gap of the
  Laplace-Beltrami operator on the underlying hyperbolic surface.
\end{remark}

In the regime $\vert \lambda T\vert \geq e$ and for $\vert
\lambda\vert \leq e$ the above result can be improved by a scaling
argument based on the action of the geodesic flow. In this case, for
all $x\in M$ let us defined the constant $C_{\Gamma,\lambda}(x,T)>0$
by the formula
\begin{equation}
  \label{eq:W_lambda}
  C_{\Gamma,\lambda}(x,T):= C_\Gamma(a^{-1}_{\log \vert\lambda\vert} (x), \vert \lambda T\vert) \,.
\end{equation}

\begin{corollary}
  \label{cor:equidistribution}
  For every $s > 7$, there is a constant $C_s := C_s(\Gamma) > 0$ such
  that the following holds.  For every $\lambda\in \R\setminus\{0\}$,
  the following bounds on the twisted ergodic integrals along the
  horocycle flow holds: for every $(x,T) \in M \times \R^+$ and for
  every zero-average function $f\in W^s(M)$, we have, if $\vert
  \lambda T \vert \geq e \geq \vert \lambda\vert$,
  \begin{equation}
    \begin{aligned}
      \vert \int_0^T e^{i \lambda t} & f \circ h_t (x) dt \vert \leq
      \frac{C_{s}}{\vert \lambda\vert } \Vert f \Vert_{s} \\ &\times
      [C_{\Gamma,\lambda}(x,T)+ C_{\Gamma,\lambda}(h_T(x),T)] \vert
      \lambda T\vert^{5/6} \log^{1/2} (\vert \lambda T\vert) \,.
    \end{aligned}
  \end{equation}
\end{corollary}
\begin{proof} For every $\lambda \in \R^*$, let
  $\{h^{U/\vert\lambda\vert}_t\}$ denote the horocycle flow with
  generator $U/\vert\lambda\vert$, which is a linear time-change of
  the stable horocycle flow $\{h_t\}= \{h^{U}_t\}$.  By the change of
  variable formula, assuming that $\vert \lambda\vert \leq e$,
  \[
  \int_0^T e^{i\lambda t} f \circ h_t (x) dt = \int_0^{\vert \lambda
    T\vert } e^{i\lambda t/\vert \lambda\vert } f \circ h^{U/\vert
    \lambda\vert }_t (x) \frac{dt}{\vert \lambda\vert}\,.
  \]
  Let $a_{\log \vert \lambda\vert}$ be the geodesic map such that
  $h^{U/\vert \lambda\vert}_t = a_{\log \vert \lambda\vert} \circ
  h^{U}_t \circ a^{-1}_{\log \vert \lambda\vert}$. It follows from
  this formula that
  \[
  \frac{d}{dt} f \circ h^{U/\vert \lambda\vert }_t \circ a_{\log
    \vert\lambda\vert}= \frac{d}{dt} f\circ a_{\log \vert\lambda\vert}
  \circ h^U_t \,.
  \]
  Hence, we get the formula
  \[
  \vert \lambda\vert ^{-1} Uf \circ a_{\log \vert\lambda\vert} = U
  (f\circ a_{\log \vert\lambda\vert}),
  \]
  which implies that
  \[
  \vert \lambda\vert Vf \circ a_{\log \vert\lambda\vert} = V (f\circ
  a_{\log \vert\lambda\vert}).
  \]
  Since our bounds are in terms of the foliated Sobolev norms, by the
  above choice of the geodesic map we have that
  \begin{equation}\label{equa:normbounds-geodesic-scaling}
    \vert  f \circ a_{\log \vert\lambda\vert}\vert_{r, s} \leq  
    \max\{ \vert \lambda\vert^j \vert 0\leq j \leq s\}  \vert f \vert_{r,s}\,.
  \end{equation}
  Now, since $\vert \lambda\vert \leq e$, the above theorem yields a
  bound
  \begin{equation}\label{equa:twist-scale1}
    \begin{aligned}
      \vert &\int_0^{\vert \lambda T\vert } e^{i \lambda t /\vert
        \lambda\vert} f \circ a_{\log \vert\lambda\vert} \circ h^U_t
      \circ a^{-1}_{\log \vert\lambda\vert} (x) \frac{dt}{\vert
        \lambda\vert } \vert \leq \frac{C_{r,s}}{\vert \lambda\vert }
      \vert f \vert_{r,s} \\ &\times [C_{\Gamma,\lambda}(x,T)+
      C_{\Gamma,\lambda}(h_T(x),T)] (\lambda T)^{5/6} \log^{1/2}
      (\vert \lambda T\vert) \,.
    \end{aligned}
  \end{equation}
  The argument is therefore complete.
\end{proof}

\begin{proof}[Proof of Theorem \ref{thm:Main_Twisted}] 
  Let $A\in [0,1)$ and $Q>0$.  By definition of $M_{A, Q}$, we have
  for all $x \in M_{A, Q}$, for all $\lambda \in \R^*$ and for all $t
  \geq 0$,
  \[
  d(a_t(a_{\log \vert \lambda\vert}^{-1}(x)) \leq A (t - \log
  |\lambda|) + Q\,.
  \]
  Then it follows from Remark~\ref{rmrk:C-Gamma} and
  Lemma~\ref{lemma:c(x,T)} that there exists a constant
  $C_{\Gamma,A,Q}>0$ such that whenever $x$, $h_T(x) \in M_{A,Q}$, if
  $\vert\lambda\vert \leq e$ we have
  \[
  C_{\Gamma,\lambda}(h_T(x),T)+C_{\Gamma,\lambda}(x,T) \leq
  2C_{\Gamma,A,Q} \vert \lambda T \vert^{\frac{2A}{1-A}} \,;
  \]
  by the logarithmic law of geodesics, for almost all $x\in M$ and for
  all $\epsilon>0$, there exists a constant $C_\epsilon(x)>0$ such
  that
  \[
  C_{\Gamma,\lambda}(x,T) \leq C_\epsilon(x) [1+ \log^{1+\epsilon}
  (\vert \lambda T\vert)]\,.
  \]
  By the above bounds on the constants $C_{\Gamma,\lambda}(x,T)$ under
  the relevant Diophantine conditions, the statement of
  Theorem~\ref{thm:Main_Twisted} is an immediate consequence of
  Theorem~\ref{theo:equidistribution} for $\vert \lambda\vert \geq e$
  and of Corollary~\ref{cor:equidistribution} for $\vert \lambda\vert
  \leq e$.

  Now assume that $x$ lies on the cuspidal horocycle $\gamma_T$ of
  length $T>0$. We remark that, under the assumption that $\lambda T
  \in 2\pi \Z$, for all $s>0$ and for all continuous functions $f$ on
  $M$ we have
  \[
  \begin{aligned}
    \vert \int_0^T e^{i\lambda t} &f\circ h_t(h_s(x)) dt \vert = \vert
    \int_s^{T+s} e^{i\lambda (t-s)} f\circ h_t(x) dt \vert \\ &= \vert
    \int_s^{T+s} e^{i\lambda t} f\circ h_t(x) dt \vert =\vert \int_0^T
    e^{i\lambda t} f\circ h_t(x) dt \vert \,.
  \end{aligned}
  \]
  In other terms in this case the modulus of a twisted horocycle
  integral along a cuspidal horocycle does not depend on the initial
  point.

  Let $K_\Gamma>0$ be a constant such that all cuspidal horocycles of
  unit length on $M$ are contained in the compact set $\{x\in M:
  d_M(x) \leq K_\Gamma\}$.  Let $C_\Gamma>0$ denote the constant
  introduced below in Lemma~\ref{lemma:fund_box}.

  By hyperbolic geometry for any given cusp and for any $x'\in M$,
  which does not belong to an unstable cuspidal horocycle for that
  cusp, there exists a (unique) point on every stable cuspidal
  horocycle such that the backward geodesic orbits of $x$ and $x'$ are
  asymptotic.  Because there exists a dense set of points in $M$ with
  relatively compact backward orbit, it follows that on every stable
  cuspidal horocycle there is a dense set of points with relatively
  compact \emph{backward} geodesic orbit.

  Thus, there exists a positive integer $n_\Gamma>0$ such that any
  stable cuspidal horocycle $\gamma_T$ of length $T>0$ has a partition
  $\{x_{1}, \dots, x_{n_\Gamma}\}$ such that for all $k\in \{1, \dots,
  n_\Gamma\}$ (modulo $n_\Gamma$), there exists
  \[T_k \in (0, \frac{C_\Gamma T} {10 K_\Gamma} )\] such that $x_{k+1}
  = h_{T_k} (x_k)$ and $x_k$ belongs to the relatively compact
  backward orbit of a point on the cuspidal horocycle of unit
  length. There exists therefore a constant $K'_\Gamma>0$ such that
  the following bound holds:
  \[
  \max_{0\leq y \leq \log T} d_M (a_y(x_k)) \leq K'_\Gamma \,.
  \]
  Since for all $T\geq 1$ the loop $a_{\log T}(\gamma_T)$ is a
  cuspidal horocycle of unit length, we have by the definition of
  $K_\Gamma$ that for all $x\in \gamma_T$,
  \[
  d_M(a_{\log T}(x)) \leq K_\Gamma\,.
  \]
  hence, for each $k$, it follows from Lemma~\ref{lemma:c(x,T)} and
  the condition $T_k < \frac{C_\Gamma T}{10 K_\Gamma}$ that
  \begin{equation}
    \label{eq:cusp_bound}
    C_\Gamma(x_k,T_k) = \max_{0\leq t\leq T_k} c_\Gamma(x_k,t) \leq 
    \left( \frac{10}{C_\Gamma}\right)^2 e^{2K'_\Gamma}\,.% \quad \text{ if } \,\, 
    % t \leq T_k < \frac{C_\Gamma}{10} \frac{T}{K_\Gamma} \,.
  \end{equation}
  Then we conclude from Theorem~\ref{theo:equidistribution} and
  Corollary~\ref{cor:equidistribution} that there exists a constant
  $C_s:=C_s(\Gamma)>0$ such that if $\vert \lambda T \vert \geq e$,
  \[
  \vert \int_0^{T_k} e^{i\lambda t} f\circ h_t(x_k) dt \vert \leq
  C_{s} \Vert f \Vert_{s} (1+ \frac{1}{\vert\lambda\vert^{1/6}})
  T^{5/6} \log^{1/2} (\vert \lambda T\vert) \,.
  \]
  Finally, the statement follows from finiteness of the partition.

\end{proof}

\subsection{Fourier coefficients of cusp forms}
\begin{proof}[Proof of Theorem~\ref{coro:Cusp-form}]
  Let $\Gamma \subset \SL(2, \R)$ be an arbitrary lattice containing
  the uni\-potent element $\left(\begin{smallmatrix}
      1 & 1 \\
      0 & 1
    \end{smallmatrix}
  \right)$.  Such lattices are not co-compact.  Let $k \in \N$ be
  even, let $f$ be a holomorphic cusp form of weight-$k$ for $\Gamma$.
  Let $f$ have the Fourier expansion
  \[
  f = \sum_{n > 0} a_n e^{2\pi i n z}\,,
  \]
  so the coefficients $\{a_n\}_{n > 0} \subset \C$ are given by
  \[
  a_n = e^{2\pi} \int_{\R / \Z} f(t + \frac{i}{n}) e^{-2\pi i n t}
  dx\,.
  \]
 
  Following Section~1.3.4 of \cite{V}, let $\tilde f$ be the lift of
  $f$ to $\Gamma \backslash \SL(2, \R)$ given by
  \[
  \tilde f : \Gamma \left(\begin{array}{cc} a & b \\ c &
      d \end{array}\right) \to f \left(\frac{a i + b}{ci + d}\right)
  (c i + d)^{-k}\,.
  \]
  Then
  \begin{equation}\label{equa:lift-smooth}
    \tilde f \in C^\infty(\Gamma \backslash \SL(2, \R))\,.
  \end{equation}  
  Let $ x_n = \Gamma \left(
    \begin{array}{cc}
      n^{-1/2} & 0 \\
      0 & n^{1/2}
    \end{array}
  \right) \in M\,.  $ Then
  \[
  \tilde f \circ h_t(x_n) = n^{-k/2} f(\frac{i + t}{n})\,.
  \]

  Consequently, we have as in formula (1.7) of \cite{V} that
  \begin{align}\label{equa:cusp-coeff-twist}
    a_n & = e^{2\pi} \int_{\R / \Z} f(t + \frac{i}{n}) e^{-2\pi i n t} dt \notag \\
    & = e^{2\pi} n^{-1} \int_{\R / n \Z} f(\frac{i +t}{n}) e^{-2\pi i t} dt \notag \\
    & = e^{2\pi} n^{k/2 - 1} \int_{0}^n \tilde{f} \circ h_t (x_n)
    e^{-2\pi i t} dt \,.
  \end{align}

  The twisted integral (\ref{equa:cusp-coeff-twist}) is over a closed
  horocycle of length $n$.
  % (see comment immediately below formula (1.7) of \cite{V}).
  Because $\tilde f$ is smooth, we may take any $s > 7$, and
  Theorem~\ref{thm:Main_Twisted} gives a constant $C_{r, s, f} > 0$
  such that
  \[
  \vert e^{2\pi} n^{k/2 - 1} \int_{0}^n \tilde{f} \circ h_t (x_n)
  e^{-2\pi i t} dt\vert \leq C_{r, s, f} n^{k/2 - 1/6} \log^{1/2}(e+n)
  \,.
  \]
  Theorem~\ref{coro:Cusp-form} is now immediate from
  \eqref{equa:cusp-coeff-twist}.
\end{proof}

\subsection{Horocycle maps}
A precise description of the space of invariant distributions and the
statement of our effective equidistribution theorem for horocycle maps
require that we recall of the theory of unitary representations of the
group $\SL(2, \R)$.

There are four classes of irreducible, unitary representations $H_\mu$
of $\SL(2, \R)$.  They are parameterized by the Casimir operator
$\Box$ and termed the principal series, the complementary series, the
discrete series and the mock discrete series.
We have the following cases: 
\begin{itemize}
\item  When $\mu \in (0, 1)$, then $H_\mu$
is in the complementary series.
\item  When $\mu > 1$,
then $H_\mu$ is in the principal series.
\item When $\mu = 1$, then $H_\mu$ is in
the mock discrete series or the principal series.
\item  When $\mu \leq 0$, then $H_\mu$ is in the discrete series.
\end{itemize}
The standard line and upper half-plane models for irreducible
representations of $\SL(2, \R)$ are discussed in
Appendix~\ref{appe:A}.

The irreducible, unitary representations of $\SL(2, \R) \times \T$ are
parameterized by tuples $(m, \mu) \in \Z \times \text{spec}(\Box)$,
and they are denoted
\[
H_{m, \mu} := H_\mu \otimes e_m\,,
\]
where $e_m \in L^2(\T)$ is given by
\[
e_m(t) := e^{i m t}\,.
\]
We will refer to a representation $H_{m, \mu}$ as a principal series
representation (resp.  complementary series, discrete series, or mock
discrete series) if $H_\mu$ is.

The regular representation $L^2(M \times \T)$ of $\SL(2, \R)\times \T$
decomposes as a direct sum or integral of irreducible, unitary
representation spaces $H_{m, \mu}$, which occur with at most finite
multiplicity.  By irreducibility, vector fields are decomposable in
the sense that, for any $s \in \R$, $W^s(M \times \T)$ decomposes as a
direct sum or integral of irreducible, unitary, Sobolev subspaces
$H_{m, \mu}^s$, where $H_{m, \mu}^s$ is the restriction of $W^s(M
\times \T)$ to $H_{m, \mu}$ and has inner product $\langle \cdot ,
\cdot \rangle_{W^s(M\times \T)}$ .

For any irreducible component $H_{m, \mu}$, the distributional dual
space to $ H_{m, \mu}^s$ is denoted $H_{m, \mu}^{-s} := \left(H_{m,
    \mu}^s\right)'$.  The subspace of smooth vectors in $H_{m, \mu}$
is denoted $H_{m, \mu}^\infty := \bigcap_{s \geq 0} H_{m, \mu}^s$, and
its distributional dual space is denoted
\[
H_{m, \mu}^{-\infty} := \left(H_{m, \mu}^\infty \right)' = \bigcup_{s
  \geq 0} H^{-s}\,.
\]

In completely analogous fashion, for any $r, s \geq 0$, the foliated
Sobolev space $W^{r, s}(M \times \T)$ decomposes into a direct sum or
integral of irreducible, unitary representation spaces denoted $H_{m,
  \mu}^{r, s}$.  The distributional dual space of $H_{m, \mu}^{r, s}$
is denoted $H_{m, \mu}^{-r, -s}$.

Note that $\triangle$ restricts on $L^2(M)$ to the essentially
self-adjoint elliptic operator
\[
-X^2 -2(U^2 + V^2)\,.
\]
For any $s \in \R$, $W^s(M)$ is defined to be $W^s(M\times \T) \cap
L^2(M)$, and it is endowed with an inner product that is obtained from
$W^s(M\times \T)$.  In completely analogous fashion, the Sobolev space
$W^s(M)$ decomposes into a direct sum or integral of irreducible,
unitary subspaces $H_\mu^s:= H_{0, \mu}^s$.

The space of smooth functions on $M$ is defined to be $C^\infty(M) =
C^\infty(M\times \T) \cap L^2(M)$, and its distributional dual space
is $\mathcal E'(M)$.  We let $H_\mu^\infty := H_{0, \mu}^\infty$ and
$H_{\mu}^{-\infty} := H_{0, \mu}^{-\infty}$.

Similarly, the foliated space $\widehat W^{r,s}(M)$ decomposes into a
direct sum or integral of irreducible, unitary subspaces $\widehat
H_\mu^{r, s} := \widehat H_{0, \mu}^{r,s}$.

For any $L > 0$ and $N \in \Z^+$, let
\[
\begin{aligned}
  \mathcal I^{s, L}(\Gamma) & := \left\{\mathcal D \in W^{-s}(M) :  h_L \mathcal D = \mathcal D\right\} \,, \\
  \mathcal I^{L}(\Gamma) & := \left\{\mathcal D \in \mathcal E'(M) :
    h_L \mathcal D = \mathcal D\right\}
\end{aligned}
\]
be the space of $h_L$-invariant distributions of order $s\geq 0$ and
infinity in $W^{-s}(M)$ and $\mathcal E'(M)$, respectively.  Let
\[
\mathcal I^{0}(\Gamma) := \left\{\mathcal D \in \mathcal E'(M) : U
  \mathcal D = \mathcal D\right\}
\]
be the space of invariant distributions for the horocycle flow in
$\mathcal E'(M)$, and let
\[
\begin{aligned}
  \text{Ann}^{s, L}(\Gamma) :&= \left\{f \in W^s(M) : \mathcal D(f) =
    0 \text{ for all }
    \mathcal D \in \mathcal I^{s, L}(\Gamma)\right\}\,, \\
  \text{Ann}^L(\Gamma) & := \left\{f \in C^\infty(M) : \mathcal D(f) =
    0 \text{ for all } \mathcal D \in \mathcal
    I^{L}(\Gamma)\right\}\,.
\end{aligned}
\]
We know from Theorem 1.2 of \cite{T} that these are the spaces of
coboundaries of Sobolev regularity $s\geq 0$ and $\infty$ for the
horocycle map $h_L$.
 
Let
\[
\mathcal I^{0}(\Gamma) := \left\{\mathcal D \in \mathcal E'(M) : U
  \mathcal D = \mathcal D\right\}
\]
and, by Theorem~1.2 of \cite{FF1}, the space of smooth coboundaries
for the horocycle flow is
\[
\text{Ann}^0(\Gamma) := \left\{f \in C^\infty(M) : \mathcal D(f) = 0
  \text{ for all } \mathcal D \in \mathcal I^{0}(\Gamma)\right\}\,.
\]

By Theorem 1.1 of \cite{T}, Theorem 1.1 of \cite{FF1} and
Theorem~\ref{theo:Invariant_dist} above, the space $\mathcal
I^{\infty, L}(M)$ is described as follows.
\begin{theorem}\label{theo:invariant-maps}
  Let $\sigma_{\text{pp}}$ be the spectrum of the Laplace-Beltrami
  operator $\triangle$ on $L^2(M)$.  Then in any Sobolev structure
  $W^{-s}(M)$, for $s > 0$, there is a splitting
  \[
  \mathcal I^{L}(\Gamma) = \mathcal I^{0}(M) \oplus \mathcal I^{L,
    \text{twist}}(\Gamma)\,,
  \]
  where we have $\mathcal I^{L, \text{twist}}(\Gamma) \subset
  W^{-(1/2+)}(M)$ and for each irreducible, unitary space $H$, the
  space $\mathcal I^{L, \text{twist}}(\Gamma) \cap H^{-(1/2 +)}$ has
  infinite, countable dimension.

  The space $\mathcal I^{0}(\Gamma)$ is described in Theorem 1.1 of
  \cite{FF1} as follows: It has infinite, countable dimension.  It is
  a direct sum of the trivial representation $\mathcal I_{\vol}$ and
  irreducible, unitary representations $\mathcal I_{\mu}$ belonging to
  the principal series, the complementary series, the discrete series
  and the mock discrete series.

  Specifically,
  \begin{itemize}
  \item The space $\mathcal I_{\vol}$ is spanned by the $\SL(2,
    \R)$-invariant volume;
  \item For $0 < \mu < 1$, there is a splitting $\mathcal I_\mu =
    \mathcal I_\mu^+ \oplus \mathcal I_\mu^-$, where $\mathcal
    I_\mu^{\pm} \subset W^{-s}(M)$ if and only if $s > \frac{1 \pm
      \sqrt{1 - \mu}}{2}$, and each subspace has dimension equal to
    the multiplicity of $\mu \in \text{spec}(\Box)$;
  \item If $\mu \geq 1$, then $\mathcal I_\mu \subset W^{-s}(M)$ if
    and only if $s > 1/2$, and it has dimension equal to twice the
    multiplicity of $\mu \in \text{spec}(\Box)$;
  \item If $\mu = -n^2 + 2n$ for $n \in \Z^+$, then $\mathcal I_\mu
    \subset W^{-s}(M)$ if and only if $s > n/2$ and it has dimension
    equal to twice the rank of the space of holomorphic sections of
    the $n_{\text{th}}$ power of the canonical line bundle over $M$.
  \end{itemize}
\end{theorem}

\begin{sloppypar}
  The description of $\mathcal I^{s, L}(\Gamma)$ is given by $\mathcal
I^{s, L}(\Gamma) := \mathcal I^{L}(\Gamma) \cap W^{-s}(M)$, where
$\mathcal I^{s, 0}(\Gamma) := \mathcal I^0(\Gamma) \cap W^{-s}(M)$ and
$\mathcal I^{s, L, \text{twist}}(\Gamma) := \mathcal I^{L,
  \text{twist}}(\Gamma) \cap W^{-s}(M)$.
\end{sloppypar}

By Theorem 1.4 of \cite{FF1}, the space $\mathcal I_\mu$ has a basis
of generalized eigendistributions for the geodesic flow.
Consequently, the projection of the ergodic sum for $h_L$ to $\mathcal
I^{s, 0}(M)$ is controlled by estimating the decay of invariant
distributions under the action of the geodesic flow, see Section 5 of
\cite{FF1} and Proposition 7.3 of \cite{T}.

In contrast, there is no subspace of $\mathcal I^{s, L,
  \text{twist}}(\Gamma)$ that is invariant under the geodesic flow.
The projection of the ergodic sum to $\mathcal I^{s, L,
  \text{twist}}(\Gamma)$ will be controlled by the ergodic average of
the twisted horocycle flow.
 
Let $\mathcal I_{\mu}^{s, 0} := \mathcal I_\mu \cap W^{-s}(M)$ be the
space of horocycle flow-invariant distributions in $\mathcal I_\mu$ of
order $s$.  For any $\mu \in \text{spec}(\Box)$, set
\[
\mathcal S_{\mu}^{\pm} := \left\{
  \begin{array}{ll}
    \frac{1 \pm \re\,\sqrt{1 - \mu}}{2} & \text{ if } \mu > 0\,; \\
    n/2 & \text{ if }  \mu = - n^2 + 2n \text{ and } n \in \Z^+ \,.
  \end{array}
\right.
\]

Theorem~\ref{theo:maps} on the rate of equidistribution of horocycle
maps is stated in terms of anisotropic Sobolev norms, which we
presently describe.  The operator $-U^2$ is non-negative and
essentially self-adjoint.  Then for any $a \geq 0$, $(I - U^2)^{a/2}$
is defined by the spectral theorem.  Then for any $a \geq 0$ and for
any $r, s \geq 0$, we define $W^{r, s, a}(M)$ to be the Sobolev
subspace that is the maximal domain of the operator
\[
((I + \Box^2)^{r/2} (I + \Box^2 + \widehat \triangle^2)^{s/2} (I -
U^2)^{a/2}
\]
on $L^2(M)$, which is endowed with the inner product
\[
\langle f, g \rangle_{W^{r, s, a}(M)} := \langle (I - U^2)^{a/4} f, (I
- U^2)^{a/4} g \rangle_{r, s} \,.
\]
We denote the corresponding norm by
\[
\Vert f \Vert_{r, s, a} := \langle f, f \rangle_{W^{r, s, a}(M)}^{1/2}
\,.
\]
The dual space of $W^{r, s, a}(M)$ is denoted $W^{-r, -s, -a}(M)$ and
has the corresponding dual norm denoted $\Vert \cdot \Vert_{-r, -s,
  -a}$\,.  Observe that for any $r, s \geq 0$ and $a > 0$, we have the
continuous embeddings
\begin{equation}\label{equa:sobolev-inclusions_map}
  W^{0, r+s+a, r+s+a}(M) \subset W^{r+s+a}(M) \subset W^{r, s, a}(M) \subset \widehat W^{r, s}(M)  \,.
\end{equation}
For all $(x,L, N) \in M\times \R^+\times \N$ we define the constants
\[
\begin{aligned}
  C_\Gamma(x,L,N)&:= C_\Gamma(h_{-L/2}(x), NL) + C_\Gamma(h_{L(N-1/2)}(x), NL) \,; \\
  D_\Gamma(x,L,N) &:= e^{d_M(h_{-L/2}(x))} +e^{ d_M (h_{L(N -
      1/2)}(x))}\,.
\end{aligned}
\]
\begin{theorem}\label{theo:maps}
  For any $s > 2$, $a>2$ and, $r \geq 5s - 3$, and for any $\epsilon >
  0$, there are constants $C_{r,s, a, \epsilon} := C_{r,s, a,
    \epsilon}(\Gamma) > 0$ and $C_{r,s, \epsilon} := C_{r,s,
    \epsilon}(\Gamma) > 0$ such that the following holds.  For any
  $(x, L, N)\in M\times \R^+ \times \Z^+ $, there is a decomposition
  of the ergodic sum of the time-$L$ horocycle map $h_L$ as a
  distribution in $W^{-r, -s, -a}(M)$ as follows.  We have
  \[
  \begin{aligned}
    \sum_{n = 0}^{N - 1} (h_{Ln} (x))^* = \mathcal D^0_{x, N,L, r, s,
      a} + \mathcal D^{\text{twist}}_{x, N,L, r, s, a} +
    \mathcal{R}_{x, N, L, r, s, a} \,.
  \end{aligned}
  \]
  The distribution $\mathcal D^0_{x, N,L, r, s, a}$ is invariant under
  the horocycle flow, the distribution $\mathcal D^{\text{twist}}_{x,
    N,L, r, s, a}$ is invariant under the time-$L$ horocycle map (but
  not under the horocyce flow), and the distribution $\mathcal{R}_{x,
    N, L, r, s, a}$ belongs to the orthogonal subspace $\mathcal I^{s,
    L}(\Gamma)^{\bot}$ of the space $\mathcal I^{s, L}(\Gamma)$ of
  invariant distributions for the time-$L$ horocycle map.  For all $f
  \in W^{2s + a + 1 + \epsilon}(M)$ and for all $(x,L,N) \in M \times
  \R^+\times \mathbb{Z}^+$ such that $NL \geq e$ the following
  estimates hold:
  \[
  \begin{aligned}
    \vert &\mathcal D^0_{x, N,L, r, s, a}(f) - \frac{1}{L} \int_0^{N L} f \circ h_t(x) dt \vert  \\
    &
    \,\,\,\,\,\,\,\,\,\,\,\,\,\,\,\,\,\,\,\,\,\,\,\,\,\,\,\,\,\,\,\,\,\,\,\,\,\,\,\,\,\,\,\,\,\,\,\,\,\,\,\,\,\,\,
    \leq C_{r,s, a, \epsilon} D_\Gamma(x,L,N) \frac{1+ L^{2s+2+\epsilon}}{L}  \Vert f \Vert_{r,s,s+a} \,; \\
    \vert &\mathcal D^{\text{twist}}_{x, N,L, r, s, a}(f) \vert \leq
    C_{r,s, \epsilon}(1+L^{8s+\epsilon})
    C_\Gamma(x,L,N) (NL)^{\frac{5}{6}}  \log^{\frac{1}{2}}(NL) \Vert f\Vert_{r,s,1 + \epsilon}  \\
    &\,\,\,\,\,\,\,\,\,\,\,\,\,\,\,\,\,\,\,\,\,\,\,\,\,\,\,\,\,\,\,\,\,\,\,\,\,\,\,\,\,\,\,\,\,\,\,\,\,\,\,\,\,\,\,\,\,
    +  C_{r,s, a, \epsilon} D_\Gamma(x,L,N) (1+ L^{2s+2+\epsilon}) \Vert f\Vert_{r,s,s+a+1+\epsilon}  \,;\\
    \vert &\mathcal{R}_{x, N, L, r, s, a}(f)\vert \leq C_{r,s,
      a,\epsilon} D_\Gamma(x,L,N) \frac{1+ L^{2+\epsilon}}{L} \Vert f
    \Vert_{r,s,a}\,.
  \end{aligned}
  \]
\end{theorem}

As in Corollary~\ref{cor:equidistribution}, the above theorem can be
improved for $L\geq 1$ by a geodesic scaling argument.  For all $(x,L,
N) \in M\times \R^+\times \N$ we define the constants
\[
\begin{aligned}
  \widetilde C_\Gamma(x,L,N)&:= C_\Gamma(a_{\log L}(x), 1, N)  \,; \\
  \widetilde D_\Gamma(x,L,N) &:= D_{\Gamma}(a_{\log L}(x), 1, N) \,.
\end{aligned}
\]
\begin{corollary}
  \label{cor:maps}
  For every $s > 14$ and every $\epsilon>0$, there is a constant
  $C'_{s,\epsilon}:= C'_{s,\epsilon}(\Gamma) > 0$ such that the
  following holds.  For every $L \geq 1$, the following bounds on the
  ergodic sums for horocycle maps holds: for every $(x,N) \in M \times
  \N$ and for every function $f\in W^s(M)$, we have,
  \begin{equation}
    \begin{aligned}
      \vert &\sum_{n = 0}^{N - 1} f \circ h_{Ln} (x) - \frac{1}{L}
      \int_0^{NL} f \circ h_t (x) dt \vert \leq C'_{s,\epsilon} \Vert
      f \Vert_{s} \\ &\times \left( \widetilde C_\Gamma(x,L,N) L^{1/6
          +\epsilon} (NL)^{\frac{5}{6}} \log^{\frac{1}{2}}N +
        \widetilde D_\Gamma(x,L,N) (1+L^{5+\epsilon})\right) \,.
    \end{aligned}
  \end{equation}
\end{corollary}
\begin{proof}
  For any $\epsilon >0$, let $s, a \in (2, 2+\epsilon/4)$, and $r \in
  (7, 7+\epsilon/4)$.  Then for all $f \in W^{14 + 2 \epsilon}(M)$,
  \begin{equation}\label{equa:Sob-exponents-main-maps}
    \Vert f\Vert_{r,s,a} \leq \Vert f \Vert_{r,s,s+a} 
    \leq  \Vert f\Vert_{r,s,s+a+1+\epsilon}  \leq 
    \Vert  f \Vert_{14+2\epsilon}\,.
  \end{equation}

  For any $t \in \R$, we have $a_{\log L}^{-1} \circ h_t \circ a_{\log
    L} = h_{L t}$. Then
  \[
  \begin{aligned}
    \sum_{n = 0}^{N - 1} f \circ h_{Ln} (x) & = \sum_{n = 0}^{N - 1} (f \circ a_{\log L}^{-1}) \circ h_{n} (a_{\log L}(x)) \,; \\
    \frac{1}{L} \int_0^{NL} f \circ h_t (x) dt & = \int_0^{N} f \circ
    a_{\log L}^{-1} \circ h_t(a_{\log L}(x)) dt \,.
  \end{aligned}
  \]
  Now by Theorem~\ref{theo:maps} in the case $L=1$ we have a
  decomposition
  \[
  \begin{aligned}
    \sum_{n = 0}^{N - 1} &(f \circ a_{\log L}^{-1}) \circ h_{n} (a_{\log L}(x))  = \mathcal D^0_{a_{\log L}(x), N, 1, r, s, a}(f \circ a_{\log L}^{-1}) \\
    &+ \mathcal D^{\text{twist}}_{a_{\log L}(x), N, 1, r, s, a}(f
    \circ a_{\log L}^{-1}) + \mathcal{R}_{a_{\log L}(x), N, 1, r, s,
      a}(f \circ a_{\log L}^{-1}) \,,
  \end{aligned}
  \]
  such that the bounds stated in Theorem~\ref{theo:maps} hold: for all
  $(\tilde x, N) \in M\times \Z^+$ and for all $\tilde f \in
  W^{r,s,a}(M)$ we have
  \[
  \begin{aligned}
    \vert & \mathcal D^0_{\tilde x, N,1, r, s, a}(\tilde f ) -
    \int_0^{N} \tilde f \circ h_t(x) dt
    \vert \leq C_{r,s, a, \epsilon} D_\Gamma(\tilde x,1,N) \Vert \tilde f \Vert_{r,s,s+a} ;  \\
    \vert &\mathcal D^{\text{twist}}_{\tilde x , N,1, r, s, a}(\tilde
    f) \vert \leq C_{r,s, \epsilon}
    C_\Gamma(\tilde x,1,N) N^{\frac{5}{6}}  \log^{\frac{1}{2}}(e+N) \Vert \tilde f\Vert_{r,s,1 + \epsilon};  \\
    &\qquad\qquad\qquad\qquad\qquad\qquad
    + C_{r,s, a, \epsilon} D_\Gamma(\tilde x,1,N) \Vert \tilde f\Vert_{r,s,s+a+1+\epsilon}; \\
    \vert &\mathcal{R}_{\tilde x, N, 1, r, s, a}(\tilde f )\vert \leq
    C_{r,s, a, \epsilon} D_\Gamma(\tilde x,1,N) \Vert \tilde f
    \Vert_{r,s,a}\,.
  \end{aligned}
  \]
  By the above bounds for $\tilde x = a_{\log L}(x)$ and $\tilde f= f
  \circ a_{\log L}^{-1}$ we derive the estimate
  \[
  \begin{aligned}
    \vert \sum_{n = 0}^{N - 1} & f \circ h_{Ln} (x)  - \frac{1}{L} \int_0^{NL} f \circ h_t (x) dt  \vert \\
    &
    \leq 3 C_{r,s, a, \epsilon} D_\Gamma(a_{\log L}(x),1,N) \Vert f \circ a_{\log L}^{-1}\Vert_{r,s,s+a+1+\epsilon}  \\
    &+C_{r,s, \epsilon} C_\Gamma(a_{\log L}(x),1,N) N^{\frac{5}{6}}
    \log^{\frac{1}{2}}(e+N) \Vert f \circ a_{\log L}^{-1} \Vert_{r,s,1
      + \epsilon} \,.
  \end{aligned}
  \]
  The proof of the main estimate in Theorem~\ref{thm:Main_Maps} then
  follows from straightforward scaling estimates for the Sobolev norms
  under the action of the geodesic flow: for all $(r,s,a)$, for all
  functions $f\in W^{r,s,a}(M)$ and for all $L>0$, we have
  \[
  \Vert f \circ a_{\log L}^{-1}\Vert_{r,s,a} \leq (1+ L^{-s}) (1+ L^a)
  \Vert f \Vert_{r,s,a} \,.
  \]
  Since we have chosen $s$, $a \in (2, 2+\epsilon/4)$ it follows that
  \[
  s+a+1 +\epsilon\leq 5+\epsilon
  \]
  so that, for all functions $f\in W^{r,s,a}(M)$ and for all $L\geq 1$
  we have
  \[
  \Vert f \circ a_{\log L}^{-1} \Vert_{r,s,1 + \epsilon} \leq 4\, L^{1
    +\epsilon} \quad \text{and} \quad \Vert f \circ a_{\log
    L}^{-1}\Vert_{r,s,s+a+1+\epsilon} \leq 4\, L^{5+\epsilon}\,.
  \]
  The argument is therefore completed.
\end{proof}

\begin{proof}[Proof of Theorem~\ref{thm:Main_Maps}]  
  For $NL \leq e$ the statement is immediate since for continuous
  functions Riemann sums approximate the integral. We can therefore
  assume that $NL\geq e$

  We first prove the bound in formula~\eqref{equa:main_maps1} when $L
  \leq e$.  The triangle inequality gives
  \[
  d_M(h_{-L/2}(x)) - d_M(x) \leq e\,.
  \]
  Hence, there is a constant $C > 0$ such that
  \begin{equation}\label{equa:D_Gamma-estimate}
    D_\Gamma(x, L, N) \leq C (e^{d_M(x)} + e^{d_M(h_{NL}(x))})\,.
  \end{equation}
  If $x, h_{NL}(x) \in M_{A, Q}$, we immediately have that
  $D_\Gamma(x, L, N)\leq 2 C e^Q$ and, by the definition of
  $C_\Gamma(x, L, N)$, there is a constant $C_{\Gamma, A, Q}' > 0$
  such that
  \[
  C_\Gamma(x, L, N) \leq C_{\Gamma,A,Q}' (NL)^{\frac{2A}{1-A}}\,.
  \]
  Then \eqref{equa:main_maps1} for $L\leq e$ follows from
  Theorem~\ref{theo:maps}.

  By the logarithm law for geodesics, there is a measurable, finite
  almost everywhere function $C_\epsilon:M \to \R^+$ such that
  \[
  C_{\Gamma}(x,N, L) \leq [C_{\epsilon}(x)+C_\epsilon(h_{NL}(x))] (1+
  \log^{1+\epsilon} (NL))\,.
  \]
  Now for any $s' > 2, a > 2$, $r \geq 5s' - 3$ and $\epsilon' > 0$
  such that $r + s' + a + \epsilon' < s$, there is a constant $C_s :=
  C_s(\Gamma) > 0$ such that for all constants $C_{r, s', a,
    \epsilon}(\Gamma) > 0$ given in Theorem~\ref{theo:maps}, we have
  \[
  C_{r, s', a, \epsilon}(\Gamma) \leq C_{s}.
  \]
  Then there is a measurable, finite almost everywhere function $C_{s,
    \epsilon}': M \to \R^+$ given by
  \[
  C_{s, \epsilon}'(x) \geq C_{s} + C_{\epsilon'}(x) + e^{d_M(x)}\,.
  \]
  The statement \eqref{equa:main_maps2} for $L\leq e$ follows from
  this.

  When $L \geq e$, the statements follow by similar arguments with
  Corollary~\ref{cor:maps} in place of Theorem~\ref{theo:maps}.  In
  fact we have
  \[
  \begin{aligned}
    \widetilde C_\Gamma (x,L,N) &= C_\Gamma (a_{\log L} (x),1,N)  \\
    & = C_\Gamma( h_{-1/2} \circ a_{\log L} (x), N)+ C_\Gamma( h_{N-1/2} \circ a_{\log L} (x), N) \; \\
    \widetilde D_\Gamma (x,L,N) &= D_\Gamma (a_{\log L} (x),1,N) \\
    &= e^{ d_M( h_{-1/2} \circ a_{\log L} (x))} + e^{d_M( h_{N-1/2}
      \circ a_{\log L} (x))} \,.
  \end{aligned}
  \]
  Let us assume that $x$ and $h_{NL}(x)\in M_{A,Q}$ then we have
  \[
  \begin{aligned}
    d_M ( a_t \circ h_{-1/2} \circ a_{\log L} (x)) &= d_M( h_{-e^{-t}/2} \circ a_{t + \log L} (x)) \\
    &\leq A t + A \log L + Q + 1/2\; \\
    d_M ( a_t \circ h_{N-1/2} \circ a_{\log L} (x)) &= d_M\left(
      h_{-e^{-t}/2} \circ a_{t + \log L} ( h_{NL}(x)) \right) \\ &\leq
    A t+ A \log L + Q + 1/2 \,.
  \end{aligned}
  \]
  It follows by Lemma~\ref{lemma:c(x,T)} that there exist constants
  $C_{\Gamma,Q}$, $D_Q >0$ such that
  \[
  \begin{aligned}
    \widetilde C_\Gamma (x,L,N) &\leq  C_{\Gamma,Q}  (N L)^{\frac{2A}{1-A}} \,; \\
    \widetilde D_\Gamma (x,L,N) & \leq D_Q L^A \,.
  \end{aligned}
  \]
  The bound in formula~\eqref{equa:main_maps1} for $L\geq e$ then
  follows immediately from Corollary~\ref{cor:maps} and from the above
  inequalities.

  Let us then assume that $x$ and $h_{NL}(x)\in \widetilde M_{A,Q}$,
  that is, that there exist constants $A>1/2$ and $Q>0$ such that
  $d_M(a_y(x))$, $d_M(a_y(h_{NL}(x))) \leq A \log y + Q$ for all
  $y\geq 1$. For all $L\geq e$ and all $t\geq 0$ we have
  \[
  \begin{aligned}
    d_M ( a_t \circ h_{-1/2} \circ a_{\log L} (x)) &= d_M( h_{-e^{-t}/2} \circ a_{t + \log L} (x)) \\
    &\leq A \log ( t +\log L)  + Q + 1/2\; \\
    d_M ( a_t \circ h_{N-1/2} \circ a_{\log L} (x)) &= d_M\left(
      h_{-e^{-t}/2} \circ a_{t + \log L} ( h_{NL}(x)) \right) \\ &\leq
    A \log (t +\log L) + Q + 1/2 \,,
  \end{aligned}
  \]
  thus by Lemma~\ref{lemma:c(x,T)} it follows that there exists a
  constant $C_{\Gamma,A}>0$ such that
  \[
  \begin{aligned}
    \widetilde C_\Gamma (x,L,N) &\leq C_{\Gamma,A} e^{2Q}
    \left(1+ Q +  \log(NL) \right)^{2A} \,; \\
    \widetilde D_\Gamma (x,L,N) & \leq 2A \log \log L + 2Q + 1 \,.
  \end{aligned}
  \]
  The bound in formula~\eqref{equa:main_maps2} for $L\geq e$ then
  follows immediately from Corollary~\ref{cor:maps} and from the above
  inequalities.

  Next, for the estimate \eqref{equa:main_maps3}, again $s' > 2, a >
  2$ and $r \geq 5s' - 3$ be %and $\epsilon' > 0$
  such that $r + s' + a < s$.  Let $f \in W^s(M)$ be a coboundary for
  $h_L$ of zero average.  Because $\mathcal D^{0}_{x, N,L, r, s', a}$
  and $\mathcal D^{\text{twist}}_{x, N,L, r, s', a}$ are invariant
  under $h_L$, we have
  \[
  \mathcal D^{0}_{x, N,L, r, s', a}(f) = \mathcal D^{\text{twist}}_{x,
    N,L, r, s', a}(f) = 0
  \]
  Then by the decomposition in Theorem~\ref{theo:maps}, it follows
  that
  \[
  \begin{aligned}
    |\sum_{k = 0}^{N - 1} f \circ h_{L k}(x)| & = |\mathcal R_{x, N,L, r, s', a}(f)| \\
    & \leq C_{s, \epsilon} (e^{d_M(h_{-L/2}(x))} + e^{d_M(h_{L(N-1/2)}(x))}) \frac{1 + L^{2+\epsilon}}{L} \Vert f\Vert_{r, s', a} \\
    & \leq C_{s, \epsilon} (e^{d_M(h_{-L/2}(x))} +
    e^{d_M(h_{L(N-1/2)}(x))}) \frac{1 + L^{2+\epsilon}}{L} \Vert
    f\Vert_{s} \,.
  \end{aligned}
  \]

  Finally, the statement for $M$ compact is immediate from
  Theorem~\ref{theo:maps}.
\end{proof}

\section{Twisted horocycle flows: cohomological equations}
\label{sect:inv-dist}

In this section we prove Theorem~\ref{theo:Invariant_dist} on
invariant distributions and Theorem~\ref{theo:cohomological_eqn} on
solutions of the cohomological equation for the twisted horocycle
flow.  The argument is carried out in irreducible unitary
representations of $SL(2,\R) \times \T$ of every unitary type.  Our
bounds are proved with respect to rescaled foliated Sobolev norms
introduced to prove the effective equidistribution theorem, Theorem
\ref{theo:equidistribution}, with the optimal exponent within reach of
our method.

It will often be convenient to use the representation parameter $\nu =
\sqrt{1 - \mu}$ in place of the spectral Casimir parameter $\mu\in
\R$.

\subsection{Rescaled Sobolev norms}

For each $\mathcal T \geq 1$, let $X_{\mathcal T}$ and $V_{\mathcal
  T}$ be the rescaled vector fields, defined as follows:
\begin{equation}
  \label{eq:XVscaling}
  X_{\mathcal T} = {\mathcal T}^{-1/3} X \quad \text{ and } \quad 
  V_{\mathcal T} = {\mathcal T}^{-2/3} V\,.
\end{equation}
Let $(\mathcal T, M \times \T)$ be the Riemannian manifold $M \times
\T$ endowed with the metric that makes the ordered basis of vector
fields $(K, \mathcal T(U + K), X_{\mathcal T}, V_{\mathcal T})$ of the
Lie algebra $\R \times \sl(2, \R)$ orthonormal.  Let
$\{\phi_t^{\lambda, \mathcal T}\}_{t\in \R}$ on $M \times \T$ be the
flow that is generated by $\mathcal T(U + \lambda K)$.  The rescaled
foliated Laplacian is the essentially self-adjoint non-negative
operator
\[
\widehat \triangle_{\mathcal T} := - K^2 - X_{\mathcal T} ^2 -
V_{\mathcal T} ^2\,.
\]

The \emph{rescaled foliated Sobolev spaces} $\widehat W_\mathcal
T^{r,s}(M \times \T)$ are defined as follows. Let $\Box_{\mathcal T}$
denote the rescaled Casimir operator
\[
\Box_{\mathcal T} := \mathcal T ^{-2/3} \Box\,.
\]
For $r\geq 0$, $s\geq 0$, let $\widehat W_\mathcal T^{r,s}(M \times
\T)$ be the the maximal domain of the operator $(I+ \Box^2)^{r/2}(
I+\Box_{\mathcal T}^2 +\widehat \triangle_{\mathcal T}^2)^{s/2}$ on
$L^2(M\times \T)$ with inner product
\[
\langle F, G \rangle_{\widehat W_\mathcal T^{r,s}(M \times \T)} :=
\langle (I+ \Box^2)^{r/2}( I+\Box_{\mathcal T}^2 +\widehat
\triangle_{\mathcal T}^2)^{s/2} F, G\rangle_{L^2(M \times \T)}\,.
\]
The norm on the space $\widehat W_\mathcal T^{r,s}(M \times \T)$ is
defined as
\[
\vert F\vert_{r,s; \mathcal T} := \langle F, F \rangle_{\widehat
  W_\mathcal T^{r,s}(M \times \T)}
\]
Let $\widehat W_\mathcal T^{-r,-s}(M \times \T) = \left(\widehat
  W_{\mathcal T}^{r,s}(M \times \T)\right)'$ be its distributional
dual.  Notice that by definition, for any $r, s \geq 0$, the Sobolev
space we have the continuous embeddings
\[
\begin{aligned}
  W^{r +s}(M \times \T) &\subset \widehat W_{\mathcal T}^{r,s}(M \times \T) \,, \\
  \widehat W_{\mathcal T}^{-r,-s}(M \times \T) & \subset W^{-(r +s)}(M
  \times \T) \,.
\end{aligned}
\]

In completely analogous fashion to the decomposition of $W^s(M \times
\T)$, we have that the space $\widehat W_\mathcal T^{r,s}(M \times
\T)$ decomposes into a direct sum or a direct integral of irreducible,
unitary, Sobolev subspaces $\widehat H_{m, \mu, \mathcal T}^{r,s}$,
where $\widehat H_{m, \mu, \mathcal T}^{r,s}$ is the intersection of
$\widehat W_\mathcal T^{r,s}(M \times \T)$ to $H_{m, \mu}$ and is
endowed with the inner product $\langle \cdot , \cdot \rangle_{r,s;
  \mathcal T}$.  By definition, since the Casimir operator acts as a
constant multiple of the identity on each irreducible subspace $H_{m,
  \mu}$, it follows that for all $r$, $r'$, $s\geq 0$,  
  the spaces $\widehat H_{m, \mu, \mathcal T}^{r, s}$ and $\widehat H^{r', s}_{m,
  \mu, \mathcal T}$ coincide as vector spaces, but are endowed with
different inner products and norms.

We simplify our notation.  For $H := H_{m, \mu}$ and for $r$, $s > 0$,
the distributional dual space to $\widehat H^{r, s} :=
\widehat H_{m, \mu,\mathcal T}^{r, s}$ is denoted $\widehat
H^{-r, -s} := \left(\widehat H^{r,
    s}\right)'$.  The space of smooth vectors in $H$ in this foliated
sense is denoted $\widehat H^\infty:= \bigcap_{s \geq 0}
\widehat H^{0, s}_{\mathcal T}$ , and its distributional dual space is
denotes $\widehat H^{-\infty} := \left(\widehat
  H^\infty\right)' = \bigcup_{s \geq 0} \widehat
H^{0, -s}_{\mathcal T}$. Our convention is justified since all the above spaces
do not depend, as topological vector spaces, on the scaling parameter 
$\mathcal T >0$ (but of course the rescaled foliated Sobolev norms do
depend on the scaling parameter).

\subsection{Principal and complementary
  series}\label{subs:principal-coeqn}

Let $H_\mu$ be the line model of an irreducible unitary representation
of the principal or complementary series (see Appendix \ref{appe:A}).
 By Plancherel's theorem, a computation left for Appendix \ref{appe:B}
shows
\begin{lemma}\label{lemma:Fourier:comp}
  Then there is a constant $C > 0$ such that, for all $f\in H_{\mu}$,
  \[
  \|f\|_{0}^2 = C \int_\R |\hat f(\xi)|^2 |\xi|^{-\re\,\nu} d\xi\,.
  \]
\end{lemma}
Notice that the Fourier transform of the $X$ and $V$ operators are
\begin{equation}\label{equa:vect_coeqn}
  \hat{X} := (1 - \nu) + 2 \xi \frac{\partial}{\partial \xi}, \qquad 
  \hat{V} := -i \left((1 - \nu) \frac{\partial}{\partial \xi} + \xi \frac{\partial^2}{\partial \xi^2}\right)\,.
\end{equation}
In fact, the above formulas hold by the standard property of the
Fourier transform on the Schwartz space $\mathscr{S}(\R)$ of rapidly
decaying smooth functions, hence they hold by duality on the space
$\mathscr{S}'(\R)$ of tempered distributions.

By Lemma~\ref{lemma:Fourier:comp}, it follows that there is also a
formula for the rescaled foliated Sobolev norms of functions in
Fourier transform.

Let us adopt the following notation.  For $\mu \in \text{spec}(\Box)$,
define $ \mu_\mathcal T := \mathcal T^{-2/3} \mu\, $ so that
\[
\text{spec} (\Box_{\mathcal T}) = \{\mu_{\mathcal T} \vert \mu \in
\text{spec}(\Box)\}\,.
\]
\begin{lemma}\label{lemma:principal_Sobolev_Fourier}
  Let $H := H_{m, \mu}$ and let $r$,$s \geq 0$.  For all $\mathcal T\geq 1$ and for all $f\in
  \widehat H^\infty$, we have
  \[
  % \begin{aligned}
  %   \vert f \vert_{0, s; \mathcal T}^2 &= C \int_{\R} |((1+m^2) I -
  %   \hat{X}_{\mathcal T}^2 -
  %   \hat{V}_{\mathcal T}^2)^{\frac{s}{2}} \hat f(\xi)|^2
  %   \frac{d\xi}{\vert \xi\vert^{\re \, \nu}}\,; \\
  \vert f \vert_{r,s; \mathcal T}^2 = C (1+\mu^2)^{\frac{r}{2}}
  \int_{\R} |[ I+ \mu_{\mathcal T}^2 + (m^2 I - \hat{X}_{\mathcal T}^2
  - \hat{V}_{\mathcal T}^2)^2] ^{\frac{s}{4}}\hat f(\xi)|^2
  \frac{d\xi}{\vert \xi\vert^{\re \, \nu}}\,.
  % \end{aligned}
  \]
\end{lemma}

\subsubsection{Invariant distributions}\label{subs:invdist-principal}
Let $m \in \mathbb Z$ and $\mu>0$.  Let $H_\mu$ be the line model of
an irreducible, unitary representation in the principal or
complementary series of $\SL(2, \R)$, and let $H_{m, \mu}$ the
corresponding model for an irreducible unitary representation of
$\SL(2, \R) \times \T$ of parameters $(m,\mu)$.  For any function
$f_m\in H_{m, \mu}$ there exists a function $f\in H_\mu$ such that
\[
f_m = f \otimes e_m\,.
\]
We formally define the functional $D^\lambda_{m,\mu}$ on $H_{m, \mu}$
by the formula
\begin{equation}
  \label{eq:D_lambda_m}
  D^\lambda_{m,\mu} (f_m) := \int_\R f (t) e^{-i t \lambda m} dt \,.
\end{equation}
Note that the functional $D^\lambda_{m,\mu}$ on $H_{m, \mu}$ induces a
functional $\bar D^\lambda_{m,\mu}$ on $H_\mu$ so that by definition
we formally have that
\[
\bar D^\lambda_{m,\mu} = \bar D^{\lambda m}_{1, \mu} \,.
\]
By the above identity we can reduce all statements about functionals
$D^\lambda_{m,\mu}$ on $H_{m,\mu}$ to statements about functionals
$\bar D^{\lambda m}_{1,\mu}$ on $H_\mu$.

We now show that $D^\lambda_{m,\mu} $ is densely defined on
$H_{m,\mu}$ .  We will use the horocycle flow invariant distribution
$D_\mu^+$ that is is (sharply) defined on $H_\mu^{(1 + \re\,\nu)/2 +}$
by
\[
D_\mu^+(f) := \lim_{x \to \infty} \frac{f(x)}{(1 + x^2)^{(1 + \nu)/2}}
\]
(see Section 3.2 of \cite{FF1}).  It is immediate that any $f \in
\text{Ann}(D_\mu^+)$ is also in $L^1(\R)$, and hence,
$D^\lambda_{m,\mu}(f \otimes e_m)$ is defined by the above formula.

By a standard construction of an orthogonal basis in $H_\mu$, the
element $u_0 \in \text{Ann}(U - V)$ given by
\[
u_0(x) := (1 + x^2)^{-(1 + \nu)/2}
\]
is in $H_\mu^\infty$.  Integration by parts shows that
$D^\lambda_{m,\mu}(u_0\otimes e_m) \in \C$.  Then for any $f_m= f
\otimes e_m \in H_{m, \mu}^\infty$, the distribution
$D^\lambda_{m,\mu}$ is defined by
\[
D^\lambda_{m,\mu} (f_m) := D_{\lambda}( f_m - D_\mu^+(f) (u_0\otimes
e_m)) + D_\mu^+(f) D^\lambda_{m,\mu}(u_0\otimes e_m)\,.
\]

It follows from Lemma 3.2 of \cite{T} that $D^\lambda_{m,\mu} \in
H_{m, \mu}^{-((1 + \re\,\nu)/2 +)}$.

In fact, we have the following stronger result.
\begin{lemma}\label{lemma:regularity-D_lambda-XV-principal}
  Let $H := H_{m, \mu}$, where $\mu > 0$ and $\lambda m \neq 0$.  Then
  \[
  D^\lambda_{m,\mu} \in \widehat H^{0, -(1/2+)}\,.
  \]
\end{lemma}
\begin{proof} The Fourier transform $\hat D^\lambda_{m,\mu}$ of the
  distribution $D^\lambda_{m,\mu}$ on $H^\infty_{m, \mu}$ is the Dirac
  mass at $-\lambda m$, that is,
  \[
  \hat D^\lambda_{m,\mu} (\hat f \otimes e_m) = \hat f (-\lambda m)
  \,, \quad \text{ for all } f \in \mathscr{S}(R)\,.
  \]
  The general case can be reduced to the case when $m=1$, therefore we
  prove the result in that case.  For simplicity of notation, let
  $D^\lambda:= D^\lambda_{1,\mu}$.  Let $I_\lambda$ be any open
  interval such that $-\lambda\in I_\lambda$ and $0\not\in I_\lambda$.
  By the Sobolev embedding theorem, it follows that the distribution
  $\hat D^\lambda \in W^{-s} (I_\lambda)$ for all $s>1/2$.  In fact,
  the distribution $\hat D^\lambda$ is a probability measure at
  $-\lambda$ and by Sobolev embedding theorem $W^s(I_\lambda) \subset
  C^0(I_\lambda)$ for all $s>1/2$.  By a direct calculation we can
  prove that for every $k \in \N$, there exist constants $C_{k, \nu,
    \lambda}, C_{k, \nu, \lambda}'>0$ such that, for any function
  $\hat f \in C^\infty(I_\lambda)$, we have
  \[%\begin{equation}\label{equa:Sobolev-invdist-principal}
  \Vert \frac{d^k \hat f}{d\xi^k} \Vert_{L^2(I_\lambda)} \leq  C_{k, \nu, \lambda} \sum_{i=0}^k 
  \Vert \hat X^i \hat f \Vert  _{L^2(I_\lambda)} \leq C_{k, \nu, \lambda}' \vert f\vert_{0, k}\,.
  \]%\end{equation}
  By interpolation it follows that for every $s\geq 0$, there exists a
  constant $C_s >0$ such that for all $f \in \widehat H^{0, s}$, we
  have
  \[
  \Vert \hat f \Vert_{W^s(I_\lambda)} \leq C_s \vert f \vert_{0, s}\,,
  \]
  hence $\hat f \in W^s(I_\lambda)$ whenever $\hat f \in \widehat H^s$
  for any $s>1/2$.  The statement then follows from the Sobolev
  embedding theorem, as explained above.
\end{proof}

\subsubsection{Twisted cohomological equations}\label{sect:coeqn}
 
Let $\mathcal T \geq 1$ and $\lambda \in \R^*$.  For $s \geq 0$, we
study the twisted cohomological equation
\begin{equation}\label{equa:coeqn}
  \mathcal T (U + \lambda K) g = f
\end{equation}  
in every irreducible subspace of the Sobolev space $\widehat
W_\mathcal T^s(M \times \T)$ of the principal and complementary
series.
 
The following a priori bounds for solutions of the cohomological
equation hold.

\begin{theorem}
  \label{theo:cohomology-principal}
  Let $r$, $s \geq 0$.  There is a constant $C_{r,s} > 0$ such that
  for any $H = H_{m, \mu}$ with $\mu > 0$ and $m \in \Z / \{0\}$, and
  for any function $f_m \in \widehat H^\infty \cap
  \text{Ann}(D^\lambda_{m,\mu})$, there is a unique solution $g_m \in
  H$ satisfying~\eqref{equa:coeqn}, and moreover, for all $\mathcal T
  \geq 1$,
  \[
  \vert g_m \vert_{r,s; \mathcal T} \leq \frac{ C_{r,s} }{{\mathcal
      T}^{1/3} } \frac{1 + \vert\lambda m\vert^{-s}}{\vert \lambda
    m\vert} \, \vert f_m \vert_{r+s, s+1; \mathcal T} \,.
  \]
\end{theorem}

By proceeding formally, we note that $f_m$ and $g_m$ are simple
tensors, so we write $f_m = f \otimes e_m$ and we consider a solution
$g\otimes e_m \in \widehat H^\infty$ of the cohomological equation
\begin{equation}\label{equa:coeqn_simple-tensor} {\mathcal T} (U +
  \lambda K) g \otimes e_m = f
  \otimes e_m\,.
\end{equation}
In a line model $H_{m,\mu}$ of an irreducible unitary representation
of the principal or complementary series, the cohomological
equation~\eqref{equa:coeqn_simple-tensor} takes the form
\[ {\mathcal T} (\frac{d}{dx} + i (\lambda m)) g(x) \otimes e_m(t) =
f(x) \otimes e_m(t) \,.
\]
Then it is enough to prove Sobolev a priori estimates for the solution
to the equation
\begin{equation}\label{equa:coeqn_M-P} {\mathcal T} (\frac{d}{dx} +
  i(\lambda m)) g = f \,.
\end{equation}
By taking the Fourier transform of both sides of
\eqref{equa:coeqn_M-P}, we get that for all $\xi \in \R$,
\begin{equation}\label{equa:hat g-principal}
  \hat g(\xi) = -i \frac{\hat f(\xi)}{{\mathcal T}(\xi + \lambda m)}.
\end{equation}
We observe that $D^\lambda_{m,\mu}(e_m \otimes f) = 0$ if and only if
$\hat f(-\lambda m) = 0$.  In what follows, we will simplify notation,
and let
\[
m := 1.
\]
The estimate for $m\not =1$ will be derived from this case for the
parameter equal to $\lambda m$.  In addition, the general estimate for
rescaled equation with respect to the rescaled Sobolev norms will be
derived from the non-rescaled case. We will therefore let
\[
\mathcal T:=1\,,
\]
and we are left to consider the formula for the solution:
\begin{equation}\label{eq:coeqn-twist-gen_lambda}
  \hat g(\xi) = -i \frac{\hat f(\xi)}{(\xi + \lambda)}\,.
\end{equation}

Now, for any $\hat f \in H_\mu$ and any $\xi \in \R$, let $ \hat
f_\lambda(\xi) = \hat f(\lambda \xi)\,$. Sobolev estimates for the
solution to \eqref{eq:coeqn-twist-gen_lambda} will be obtained in
Lemma~\ref{lemma:XVderivatives} from such estimates for the function
$\hat g_\lambda$ given by the equivalent equation
\begin{equation}\label{eq:twist_eqn-lambda=1}
  \hat g_\lambda(\xi) = -i \frac{\hat f_\lambda(\xi)}{\lambda(\xi + 1)}\,.
\end{equation}  

We will now prove estimates for the above equation, and to simplify
notation, we drop the subscript $\lambda$ from $\hat f_\lambda$ and
$\hat g_\lambda$.

\medskip To further simplify notation, set $D:= D^{1}_{1,\mu}$.  As a
first step, we have the following identity
\begin{lemma}\label{lemma:transfer-def}
  Under the condition that $D (f) =0$, that is, $\hat f(-1)= 0$, we
  have
  \[
  \hat g(\xi) = -\frac{i}{\lambda} \int_0^1 \hat f'( -1 + t( \xi + 1))
  dt \,.
  \]
\end{lemma}
\begin{proof}
  For $t\in \R$, let $F (t) = \hat f( -1 + t( \xi + 1))$.  By the
  fundamental theorem of calculus
  \[
  \hat f(\xi) = F(1)= F(0) + \int_0^1 \frac{dF}{dt} (t) dt = (\xi+1)
  \int_0^1 \hat f'( -1 + t( \xi + 1)) dt\,.
  \]
  The formula for the solution then follows immediately.
\end{proof}
We will split our estimates into different regions. Let
\begin{equation}\label{equa:def_I-princ}
  I = [-\frac{3}{2}, -\frac{1}{2}] \,.
\end{equation}
For every $\nu \in (0, 1) \cup i \R$ and for every subinterval
$J\subset \R$ we will adopt the following notation
\begin{equation}\label{equa:L^2_nu-def}
  L^2_\nu (J)=  L^2 ( J,   \frac{d\xi}{\vert \xi \vert^{\re \nu}} ) \,.
\end{equation}
Then the following holds.
\begin{lemma}\label{lemma:Xder-g-P}
  For every $\alpha \in \N$, there exists a constant $C'_\alpha>0$
  such that
  \[
  \| \hat X^\alpha \hat g \|_{L^2_\nu(\R\setminus I)} \leq C'_\alpha
  \vert\lambda\vert ^{-1} \sum_{k=0}^\alpha \| \hat X ^k \hat f \|_0
  \,.
  \]
\end{lemma}
\begin{proof} It is clear that $1/(\xi +1) \in L^\infty(\R\setminus
  I)$ and there exists a constant $C_1>0$ such that
  \[
  \Vert \frac{1}{\xi +1} \Vert_{L^\infty(\R\setminus I)}\leq C_1\,.
  \]
  It follows immediately from the formula for the solution
  \eqref{eq:twist_eqn-lambda=1} that
  \[
  \Vert \hat g \Vert_{L^2_\nu(\R\setminus I)} \leq C_1 \vert \lambda
  \vert^{-1} \Vert \hat f \Vert_{L^2_\nu(\R\setminus I)}\,.
  \]
  Let us now consider derivatives. We have
  \[
  \hat X \hat g(\xi)= -i \frac{\hat X \hat f(\xi)}{\lambda (\xi + 1 )}
  + i \frac{2 \xi \hat f(\xi)}{\lambda (\xi + 1 )^2} \,.
  \]
  Since the function $\xi/(\xi + 1 )^2 \in L^\infty(\R\setminus I)$,
  there exists a constant $C_2 > 0$ such that
  \[
  \Vert \frac{\xi}{(\xi + 1)^2} \Vert_{ L^\infty(\R\setminus I)} \leq
  C_2 \,.
  \]
  It follows that
  \[
  \| \hat X \hat g \|_{L^2_\nu(\R\setminus I}) \leq (C_1 + 2 C_2)
  \vert \lambda\vert^{-1} ( \| \hat X \hat f \|_0 + \| \hat f \| _0)
  \,.
  \]
  For higher order derivatives, by induction we prove the a
  Leibniz-type formula.  There exists universal constants
  $(a^{(\alpha)}_\ell)$ such that for all functions $f_1$, $f_2 \in
  C^\infty(\R)$, the following identity holds on $\R$:
  \begin{equation}
    \label{eq:X_Leibniz}
    \hat X^\alpha (f_1 f_2) (\xi) =  \sum_{\ell=0}^{\alpha} a^{(\alpha)}_\ell \hat X^\ell f_1 (\xi) 
    (\hat X -(1 - \nu))^{\alpha-\ell} f_2(\xi) )\,.
  \end{equation}
  In particular, for $f_1(\xi)= \hat f(\xi)$ and $f_2(\xi)= 1/(\xi+1)$
  on $I$, we have
  \[
  \hat X ^\alpha \hat g(\xi) = -\frac{i}{\lambda}
  \sum_{\ell=0}^{\alpha} a^{(\alpha)}_\ell \hat X ^\ell \hat f(\xi) (2
  \xi \frac{d}{d\xi})^{\alpha-\ell} (\frac{1}{\xi+ 1} )\,.
  \]
  Another induction argument leads to bounds of the following form.
  There exists a constant $C_{\alpha,\ell}>0$ such that
  \[
  \| (2 \xi \frac{d}{d\xi})^{\alpha-\ell} (\frac{1}{\xi+ 1} )
  \|_{L^\infty(\R\setminus I}) \leq C_{\alpha,\ell}\,.
  \]
  The stated bound therefore follows.
\end{proof}
For higher order derivatives of the form $X^\alpha V^\beta$, or
equivalently $V^\beta X^\alpha$, on the set $\R\setminus I$ we begin
by computing the following Leibniz-type formula.
\begin{lemma}
  \label{lemma:V_Leibniz}
  For any $\beta \in \N$ there exist universal coefficients
  $(b^{(\beta)}_{ijkm})$ such that for any pair of functions $f_1$,
  $f_2$ we have the formula
  \begin{equation}
    \label{eq:ind_1}
    \hat V^\beta(f_1 f_2)= \sum_{
      \substack
      {
        i+j+ m\leq \beta \\
        k \leq m
      }}
    b^{(\beta)}_{ijkm} [(\frac{d}{d\xi})^m \hat V^i f_1]
    [(\hat X- (1-\nu))^k \hat V^j f_2]\,.
  \end{equation}
\end{lemma}
\begin{proof}
  The proof is by induction.  For $\beta=1$ we have by a direct
  computation
  \[
  \hat V (f_1 f_2) = \hat V (f_1) f_2 + f_1 \hat V (f_2) - i [
  \frac{d}{d\xi} f_1] [ (\hat X-(1-\nu)) f_2]\,.
  \]
  The statement is therefore verified in this case.  The proof on the
  induction step is based on the above formula and on the following
  formulas for commutators:
  \[ [\hat V, \frac{d}{d\xi}] = i \frac{d^2}{d\xi^2} \quad \text{ and
  } \quad [\hat V, (\hat X-(1-\nu))]= [\hat V, \hat X]=2 \hat V\,.
  \]
  By the induction hypothesis and by formula~\eqref{eq:ind_1}, it
  follows that in the formula for $\hat V^{\beta +1}(f_1 f_2)$ we have
  terms of the following three types
  \begin{equation}
    \label{eq:3terms}
    \begin{aligned}
      & [\hat V (\frac{d}{d\xi})^m \hat V^i f_1]
      [(\hat X- (1-\nu))^k \hat V^j f_2] \,, \\
      & [(\frac{d}{d\xi})^m \hat V^i f_1] [\hat V
      (\hat X - (1-\nu))^k \hat V ^j f_2] \,, \\
      & [(\frac{d}{d\xi})^{m+1} \hat V^i f_1] [(\hat X-(1-\nu))^{k+1}
      \hat V^j f_2]\,.
    \end{aligned}
  \end{equation}
  In fact, by the first of the above commutation relation, by an
  induction argument for every $k\in \N\setminus\{0\}$we have
  \begin{equation}
    \label{eq:commutator_1}
    [\hat V, (\frac{d}{d\xi})^k] = -k (\frac{d}{d\xi})^{k+1}\,,
  \end{equation}
  hence the first term in the above formula~\eqref{eq:3terms} is of
  the required form.

  By the second of the above commutation relation, we derive by
  induction that for every $k\in \N\setminus\{0\}$ and every $i\in
  \{0, \dots, k-1\}$ there exists universal constants $C_{k,i}>0$ such
  that
  \begin{equation}
    \label{eq:commutator_2}
    [\hat V, (\hat X-(1-\nu))^k]   =  \sum_{i=0}^{k-1}  C_{k,i} (\hat X-(1-\nu))^i \hat V\,,
  \end{equation}
  hence the second term in formula~\eqref{eq:3terms} is of the
  required form.

  Finally the third term in formula~\eqref{eq:3terms} is already in
  the required form. Thus the induction step is proved and the
  argument is complete.
\end{proof}

\begin{lemma}\label{lemma:3.10}
  For every $\alpha$, $\beta \in \N$, there exists a constant
  $C'_{\alpha, \beta}>0$ such that
  \[
  \| \hat X ^\alpha \hat V ^\beta \hat g \|_{L^2_\nu(\R\setminus I)}
  \leq \frac{C'_{\alpha, \beta}}{\vert\lambda\vert} \sum_{i+j+k \leq
    \alpha +\beta} \vert 1- \nu\vert^i \Vert \hat X ^j \hat V ^k \hat
  f\Vert_0 \,.
  \]
\end{lemma}
\begin{proof}
  For $0\leq \ell \leq \alpha$, let
  \[
  \phi^{(\ell)}_{i,m}(\xi):= (\frac{d}{d\xi})^m \hat V^i (2 \xi
  \frac{d}{d\xi})^{\alpha-\ell} (\frac{1}{\xi+ 1})\,.
  \]
  By Lemma~\ref{lemma:V_Leibniz} and by
  formulas~\eqref{eq:twist_eqn-lambda=1} and \eqref{eq:X_Leibniz}, we
  derive the following
  \begin{equation}\label{equa:VXproduct}
    \begin{aligned}
      (\hat V^\beta \hat X^\alpha) \hat g (\xi) = \frac{- i}{\lambda}
      \sum_{\ell\leq \alpha} \, \sum_{ \substack {i+j+ m\leq \beta \\
          k \leq m }}
      &a^{(\alpha)}_\ell b^{(\beta)}_{ijkm} \phi^{(\ell)}_{i,m}(\xi) \\
      & \times [(\hat X-(1-\nu))^k \hat V^j \hat X^\ell \hat f](\xi)
      \,.
    \end{aligned}
  \end{equation}
  By an induction argument we can prove that for all $\alpha$,
  $\beta\in \N$ there exists a constant $K_{\alpha, \beta}>0$ such
  that, for all $0\leq \ell \leq \alpha$, all $i+m \leq \beta$, we
  have
  \[
  \Vert \phi^{(\ell)}_{i,m} \Vert_{L^\infty(\R\setminus I)} \leq
  K_{\alpha, \beta} (1+ \vert 1- \nu\vert^i)\,.
  \]
  By taking into account the commutation relation $[\hat X, \hat V]=-2
  \hat V$, it follows that for all $\alpha$, $\beta \in \N$ there
  exists a constant $K'_{\alpha, \beta}>0$ such that
  \[
  \Vert \hat V ^\beta \hat X ^\alpha \hat g \Vert_{L^2_\nu(\R\setminus
    I)} \leq \frac{K'_{\alpha, \beta}}{\vert \lambda\vert}
  \sum_{i+j+k\leq \alpha +\beta} \vert 1- \nu\vert^i \Vert \hat X ^j
  \hat V ^k \hat f\Vert_0 \,.
  \]
  The statement then follows, again by the above commutation relation.
\end{proof}

We then prove bounds on the interval $I$.  The estimates will be based
on the integral formula for the solution.
\begin{lemma}\label{lemma:X-g-estimate-P}
  For every $\alpha \in \N$, there exists a constant $C''_\alpha>0$
  such that
  \[
  \| \hat X^\alpha \hat g \|_{L^2_\nu(I)} \leq \frac{C''_\alpha}{\vert
    \lambda\vert} \sum_{k=0}^{\alpha+1} \vert 1- \nu\vert^{\alpha-k}
  \| \hat X^k \hat f \|_0 \,.
  \]
\end{lemma}
\begin{proof} By Lemma \ref{lemma:transfer-def} and the Minkowski
  integral inequality
  \[
  \| \hat g \|_{L^2_\nu( I)} \leq \frac{1}{\vert\lambda\vert} \int_0^1
  \Vert \hat f '( -1 + t( \cdot+ 1)) \Vert_{L^2_\nu( I)} \,dt \,.
  \]
  For all $t\in [0,1]$, let $I_{t}\subset \R$ denote the interval
  \[
  I_{t} = [ -1 -t/2,-1 + t/2]\,.
  \]
  Since for all $t\in [0,1]$ and all $\xi \in I_{t}$ we have
  \begin{equation}
    \label{eq:lambda_equiv}
    1/2 \leq \vert \xi \vert \leq 2 \,,
  \end{equation}
  by change of variable we have
  \begin{equation}\label{equa:0-norm-coordinates}
    \Vert \hat f '( -1 + t(\cdot+ 1)) \Vert_{L^2_\nu( I)} 
    \leq 2^{\re \nu}\, t^{-1/2} \left(\int_{I_{t}} |\hat f'(\xi)|^2 \frac{d\xi}{\vert \xi\vert^{\re \nu}} \right)^{1/2}\,.
  \end{equation}
  We recall that for the principal series $\nu \in i\R$, and for the
  complementary series $\nu \in (0,1)$.  It follows that there exists
  a constant $C_3>0$ such that
  \[
  \begin{aligned}
    \Vert \hat f '( -1 + t( \cdot+ 1)) \Vert_{L^2_\nu( I)} &\leq C_3
    t^{-1/2} \Vert \xi \hat f '( \xi) \Vert_0\\ &\leq C_3 t^{-1/2}
    (\Vert \hat X \hat f - (1-\nu) \hat f \Vert_0 \,.
  \end{aligned}
  \]
  Hence, we get by integration over $t\in [0,1]$ that
  \[
  \| \hat g \|_{L^2_\nu( I)} \leq \frac{2C_3}{3} \vert \lambda
  \vert^{-1} \Vert \hat X \hat f - (1-\nu) \hat f \Vert_0 \,.
  \]
  For higher order derivatives we compute as follows:
  \begin{equation}
    \label{eq:X_der}
    \begin{aligned}
      \hat X^\alpha\hat g(\xi) &= -\frac{i}{\lambda} \int_0^1 (2 \xi
      \frac{d}{d\xi} +(1-\nu))^\alpha
      [ \hat f'  ( -1 + t( \xi + 1)) ] dt \\
      &= -\frac{i}{\lambda}\int_0^1 [(\hat X + 2(1-t)
      \frac{d}{d\xi})^\alpha\hat f'] ( -1 + t( \xi + 1)) ] dt
    \end{aligned}
  \end{equation}
  By applying as above the Minkowski integral inequality followed by a
  change of coordinates and \eqref{equa:0-norm-coordinates}, we get
  \begin{equation}
    \label{eq:Xint_bound}
    \begin{aligned}
      \Vert \int_0^1 &[(\hat X + 2(1-t) \frac{d}{d\xi})^\alpha\hat f']
      ( -1 + t(\xi+1) dt \Vert_{L^2_\nu(I)} \\ &\leq \int_0^1 \Vert
      [(\hat X + 2(1-t) \frac{d}{d\xi})^\alpha\hat f']
      ( -1 + t(\cdot +1)  \Vert_{L^2_\nu(I)} dt \\
      &\leq 2^{\re\,\nu} \int_0^1 t^{-1/2} \Vert (\hat X + 2(1-t)
      \frac{d}{d\xi})^\alpha\hat f' \Vert_{L^2_\nu(I_{t})} dt\,.
    \end{aligned}
  \end{equation}
  We then observe that for all $\alpha\in \N$ and for $\xi \not =0$
  the following identity holds:
  \begin{equation}
    \label{eq:X_id}
    (\hat X + 2(1-t) \frac{d}{d\xi})^\alpha \frac{d}{d\xi} 
    =[\hat X + (1-t)\frac{1}{\xi}(\hat X -(1-\nu))]^\alpha  
    \frac{1}{2\xi}  (\hat X -(1-\nu)) \,.
  \end{equation}
  By induction for all $k\in \N$ and $j \in \Z^+$, we have
  \begin{equation}
    \label{eq:Xder_xi}
    \hat X^k (\frac{1}{\xi^j}) = (1 - 2 j - \nu)^k (\frac{1}{\xi^j}) \quad \text{ and } \quad
    [\hat X -(1-\nu)]^k (\frac{1}{\xi})= 2^k (\frac{1}{\xi})\,,
  \end{equation}
  hence by the upper bound in formula~\eqref{eq:lambda_equiv} it
  follows immediately that
  \[
  \Vert (\hat X -(1-\nu))^k (1/\xi) \Vert_{L^\infty(I_{t})} = 2^k
  \Vert 1/\xi \Vert_{L^\infty(I_{t})} \leq 2^{k+1} \,.
  \]
  Thus by the identity in formula~\eqref{eq:X_id} and by the
  Leibniz-type formula~\eqref{eq:X_Leibniz} it follows that there
  exists a constant $C_4(\alpha) >0$ such that
  \[
  \Vert (\hat X + 2(1-t) \frac{d}{d\xi})^\alpha\hat f']
  \Vert_{L^2_\nu(I_{t})} \leq C_4(\alpha) \sum_{i+j\leq \alpha} \Vert
  (\hat X -(1-\nu))^{i+1} \hat X^j \hat f \Vert_0\,.
  \]
  The statement then follows from the integral bound in
  formula~\eqref{eq:Xint_bound}.
\end{proof}

For higher order derivatives of the form $X^\alpha V^\beta$, or
equivalently $V^\beta X^\alpha$, on the interval $I$ we proceed as
above.

\begin{lemma}\label{lemma:V-X-g-f}
  For every $\alpha$, $\beta \in \N$, there exists a constant
  $C''_{\alpha, \beta}>0$ such that
  \[
  \begin{aligned}
    \| \hat V^\beta \hat X^\alpha g \|_{L^2_\nu(I)} &\leq
    \frac{C''_{\alpha, \beta}}{\vert \lambda\vert} (1+\vert \nu
    \vert)^{\beta} \sum_{ \substack{
        i+j+k\leq \alpha +\beta +1 \\
        i \leq \beta }} \| \hat V^i \hat X^j (\hat X-(1-\nu))^k \hat f
    \|_0 \,.
  \end{aligned}
  \]
\end{lemma}
\begin{proof}
  By formula~\eqref{eq:X_der} we have
  \[
  \hat X^\alpha\hat g(\xi) = -\frac{i}{\lambda} \int_0^1 [(\hat X +
  2(1-t) \frac{d}{d\xi})^\alpha\hat f'] ( -1 + t( \xi + 1)) ] dt \,.
  \]
  It follows by a short calculation that, for all $\alpha$, $\beta\in
  \N$, the derivatives $\hat V^\beta \hat X^\alpha\hat g$ of the
  solution $\hat g $ of the twisted cohomological equation are given
  by the formula
  \begin{equation}
    \label{eq:VXder}
    \frac{(-i)^{\beta+1}}{\lambda}\int_0^1 t^\beta [ \hat V +(1-t)  \frac{d^2}{d\xi^2}]^\beta
    [(\hat X + 2(1-t) \frac{d}{d\xi})^\alpha\hat f']  ( -1 + t( \xi + 1)) ] dt \,.
  \end{equation}
  By the above formula, by Minkowski integral inequality and by change
  of variables, the norm $\Vert \hat V^\beta \hat X^\alpha\hat g
  \Vert_{L^2_\nu(I_\lambda)}$ is bounded by the expression
  \begin{equation}
    \label{eq:VXder_bound}
    \frac{2^{\re \nu}}{\vert\lambda\vert}  \int_0^1 t^{\beta-1/2} \Vert ( \hat V +(1-t)  \frac{d^2}{d\xi^2})^\beta
    (\hat X + 2(1-t) \frac{d}{d\xi})^\alpha\hat f'  \Vert_{L^2_\nu(I_{t})}dt 
  \end{equation}
  We observe that for all $\alpha\in \N$ and for $\xi \not =0$ the
  following identity holds:
  \begin{equation}
    \label{eq:V_id}
    \begin{aligned}
      & [\hat V +(1-t) \frac{d^2}{d\xi^2}]^\beta
      [\hat X + 2(1-t) \frac{d}{d\xi}]^\alpha \frac{d}{d\xi} \\
      &= \{ \hat V +(1-t) \frac{1}{\xi}
      [i \hat V - (1-\nu) \frac{d}{d\xi})]\}^\beta \\
      & \times \{\hat X + (1-t)\frac{1}{\xi} [\hat X- (1-\nu)]
      \}^\alpha
      \frac{1}{2\xi}  [\hat X- (1-\nu)] \\
      &= \{ \hat V +(1-t) \frac{1}{\xi} [i \hat V -
      \frac{(1-\nu)}{2\xi}(\hat X-(1-\nu)]\}^\beta \\
      & \times \{\hat X + (1-t)\frac{1}{\xi} (\hat X- (1-\nu))
      \}^\alpha \frac{1}{2\xi} [\hat X- (1-\nu)]
    \end{aligned}
  \end{equation}
  By induction the following identity holds for all $k\in \N$ :
  \begin{equation}
    \label{eq:Vder_xi}
    \begin{aligned}
      \hat V^k (\frac{1}{\xi}) = (-i)^k k! \left( \prod_{j=1}^k
        (j+\nu) \right) \, \frac{1}{\xi^{k+1}} \,.
    \end{aligned}
  \end{equation}
  By the Leibniz-type formula \eqref{eq:X_Leibniz} and by that of
  Lemma~\ref{lemma:V_Leibniz}, from the identity~\eqref{eq:V_id}, from
  formulas~\eqref{eq:Xder_xi}, \eqref{eq:Vder_xi}, by the upper bound
  in formula~\eqref{eq:lambda_equiv}, it follows that there exists a
  constant $K_{\alpha, \beta}>0$ such that for all $t\in [0, 1]$ we
  have
  \begin{align}\label{equa:I_lambdat-I_lambda}
    \Vert &( \hat V +(1-t) \frac{d^2}{d\xi^2})^\beta
    (\hat X + 2(1-t) \frac{d}{d\xi})^\alpha\hat f'  \Vert_{L^2_\nu(I_{t})} \notag \\
    &\leq K_{\alpha, \beta} (1+ \vert \nu\vert)^{\beta} \sum_{
      \substack{
        i+j+k\leq \alpha +\beta +1 \\
        i \leq \beta }}
    \Vert \hat V^i \hat X^j (\hat X -(1-\nu))^k \hat f \Vert_{L_\nu^2(I)} \\
    &\leq K_{\alpha, \beta} (1+ \vert \nu\vert)^{\beta} \sum_{
      \substack{
        i+j+k\leq \alpha +\beta +1 \\
        i \leq \beta }} \Vert \hat V^i \hat X^j (\hat X -(1-\nu))^k
    \hat f \Vert_{0} \notag \,.
  \end{align}
  The statement follows from the bound given in
  formula~\eqref{eq:VXder_bound}.
\end{proof}

From the above lemmas we derive the following result.

\begin{lemma}
  \label{lemma:XVderivatives}
  For every $\alpha$, $\beta \geq 0$ there exists a constant
  $C^{(3)}_{\alpha, \beta}>0$ such that for all $\lambda \not =0$, the
  unique solution $g\in L^2_\nu(\R)$ of the equation $(U+i\lambda) g=f
  \in L^2_\nu(\R)$ satisfies the estimate
  \[
  \begin{aligned}
    \| V^\beta X^\alpha g \|_0 & \leq \frac{C^{(3)}_{\alpha,
        \beta}}{\vert \lambda\vert}
    (1+\vert \lambda\vert^{-\beta}) (1+\vert \nu \vert)^{\beta}  \\
    & \times \sum_{ \substack{ i+j+k \leq \alpha+\beta+1 \\ j \leq
        \beta} } \vert 1- \nu\vert^i \| V^j X^k f \|_0 \,.
  \end{aligned}
  \]
\end{lemma}
\begin{proof}
  For any $\hat f \in L^2_\nu(\R)$, define $\hat f_\lambda(\xi) :=
  \hat f(\lambda \xi).$ Notice that for any $\alpha, \beta \in \N$ and
  for any $\lambda \in \R^*$, we have the following identity
  \begin{equation}\label{equa:VX-lambda}
    \hat V^\beta \hat X^\alpha \hat f_\lambda = \lambda^\beta (\hat V^\beta \hat X^\alpha \hat f)_\lambda \,,
  \end{equation}
  whenever $\hat V^\beta \hat X^\alpha f$ is defined.
 
  By formula~\eqref{eq:twist_eqn-lambda=1} the solution to the
  cohomological equation $(U + i \lambda) \hat g = \hat f$ can be
  rewritten as
  \[
  \hat g_\lambda(\xi) = -i \frac{\hat f_\lambda(\xi)}{\lambda(\xi +
    1)}\,.
  \]
  Since by definition for any $\hat f \in L^2_\nu(\R)$ we have $\hat
  f=(\hat f_\lambda)_{1/\lambda}$ , from \eqref{equa:VX-lambda} and
  from Lemmas~\ref{lemma:3.10} and~\ref{lemma:V-X-g-f} it follows that

  \begin{equation}\label{equa:1->lambda}
    \begin{aligned}
      \Vert \hat V^\beta &\hat X^\alpha \hat g\Vert_0 = \Vert \hat V^\beta \hat X^\alpha (\hat g_\lambda)_{1/\lambda}\Vert_0 \leq \vert \lambda \vert^{-\beta}\Vert(\hat V^\beta \hat X^\alpha \hat g_\lambda)_{1/\lambda}\Vert_0 \\
      & \leq \frac{C^{(3)}_{\alpha, \beta}}{|\lambda|} \vert \lambda
      \vert^{-\beta} (1+\vert \nu \vert)^{\beta} \sum_{ \substack{
          i+j+k \leq \alpha+\beta+1 \\ j \leq \beta} }
      \vert 1- \nu\vert^i  \| (\hat V^j \hat X^k \hat f_\lambda)_{1/\lambda} \|_0 \\
      & \leq \frac{C^{(3)}_{\alpha, \beta}}{\vert\lambda\vert} \vert
      \lambda \vert^{-\beta} (1+\vert \nu \vert)^{\beta} \sum_{
        \substack{ i+j+k \leq \alpha+\beta+1 \\ j \leq \beta} } \vert
      \lambda\vert^j \vert 1- \nu\vert^i \| \hat V^j \hat X^k \hat f
      \|_0\,.
    \end{aligned}
  \end{equation}
\end{proof}

For rescaled Sobolev norms we have a similar statement which can be
immediately derived from Lemma~\ref{lemma:XVderivatives}.

\begin{lemma}
  \label{lemma:XVderivatives_scaled}
  For every $\alpha$, $\beta \geq 0$ there exists a constant
  $C^{(4)}_{\alpha, \beta}>0$ such that for all $\lambda \not =0$ and
  for all $\mathcal T\geq 1$, the solution $g\in L^2_\nu(\R)$ of the
  rescaled equation $\mathcal T(U+i\lambda) g=f \in L^2_\nu(\R) $
  satisfies the following estimate with respect to the rescaled
  Sobolev norms:
  \[
  \begin{aligned}
    \| V_{\mathcal T}^\beta X_{\mathcal T}^\alpha g \|_0 & \leq \frac{
      C^{(4)}_{\alpha, \beta} }{ \vert \lambda\vert {\mathcal T}^{1/3}
    } (1+\vert \nu \vert)^\beta (1+\vert \lambda\vert^{-\beta}) \\ &
    \times \sum_{i+j+k \leq \alpha+\beta+1 } [\mathcal T^{-1/3} \vert
    1- \nu\vert]^i \| V_{\mathcal T}^j X_{\mathcal T}^k f \|_0 \,.
  \end{aligned}
  \]
\end{lemma}

\begin{proof} [Proof of Theorem~\ref{theo:cohomology-principal}]
  For $r$, $s \in \N$ even integers, since the Casimir operator $\Box$
  takes the value $\mu= 1 - \nu^2$ on any irreducible, unitary
  representation $H_{m,\mu}$, by expanding the operator $(I
  +\Box^2)^{r/2} (I+ \Box_\mathcal T^2+ \widehat \triangle_{\mathcal
    T}^2)^{s/2}$ into a polynomial expression in $X_{\mathcal T}$ and
  $V_\mathcal T$, and by the commutation relations we derive that
  there exists a constant $C^{(0)}_{r,s}>0$ such that
  \begin{equation}\label{equa:coeqn-fullnorm0}
    \vert  g \vert_{r,s; \mathcal T}  \leq  C^{(0)}_{r,s}(1+ \mu^2)^{r/4} \sum_{k+m + \alpha +\beta \leq s} 
    [\mathcal T^{-1/3} \vert 1- \nu\vert]^k \| K^m  V_{\mathcal T}^\beta X_{\mathcal T}^\alpha g\|_0 \,.
  \end{equation}
  We conclude from Lemma~\ref{lemma:XVderivatives_scaled} that there
  exist constants $C^{(1)}_{r,s}>0$ and $C^{(2)}_{r,s}>0$ such that,
  if the functions $f$ and $g$ belong to a single irreducible
  component, then
  \[
  \begin{aligned}
    \vert g \vert_{r,s; \mathcal T} & \leq C^{(1)}_{r,s} \frac{1 +
      \vert\lambda\vert^{-s}} {\vert\lambda\vert \mathcal T^{1/3}} (1
    + |\nu|)^{r+s} \sum_{k = 0}^{s + 1} [\mathcal T^{-1/3}\vert 1 -
    \nu\vert]^{k}
    \vert f  \vert_{0, s + 1 - k; \mathcal T} \\
    & \leq C^{(2)}_{r,s} \frac{1 +
      \vert\lambda\vert^{-s}}{\vert\lambda\vert \mathcal T^{1/3}}
    \vert f \vert_{r+s, s + 1; \mathcal T}\,.
  \end{aligned}
  \]
  From the above estimate, the conclusion for $r$, $s$ even integers
  and $m=1$ follows.
 
  The statement for $r$, $s \geq 0$ follows by interpolation.  The
  estimate for $m \neq 1$ follows by setting $\lambda = \lambda m$.

  The uniqueness of the solution holds, because if $g$, $h \in H$ are
  solutions of the equation~\eqref{equa:coeqn_M-P}, then in Fourier
  transform the following identity holds in $L^2_\nu(\R)$, hence
  almost everywhere,
  \[
  -i(\xi + \lambda) (\hat g - \hat h)(\xi) =0
  \]
  Since $\xi + \lambda\not =0$ almost everywhere, it follows that
  $\hat g(\xi)= \hat h(\xi)$ almost everywhere, hence $g=h \in H$.
\end{proof}

\subsection{Discrete series}
\subsubsection{Invariant distributions}
Let $H_\mu$ be an irreducible unitary representation of the discrete
series with $\mu = 1 - \nu^2$ for $\nu \in \N$.  Let $\mathbb H$ be
the upper half-plane.

From Appendix \ref{appe:A},
\[
\| f \|_{H_\mu}^2 := \left\{
  \begin{array}{ll}
    \int_0^\infty \int_{-\infty}^\infty |f(x + i y)|^2 y^{\nu - 1} dx dy, & \nu \geq 1\\
    \sup_{y > 0} \int_{\R} |f(x + i y)|^2 dx, & \nu = 0\,. 
  \end{array}
\right.
\]
For $f_m = e_m \otimes f \in H_{m, \mu}$, define
\[
\begin{aligned}
  D_{m, \mu}^\lambda(f_m) := e^\lambda \int_{\R} f(t + i) e^{- i t
    \lambda m} dt \,,
\end{aligned}
\]
and observe that $D_{m, \mu}^{\lambda}$ induces a functional $\bar
D_{m, \mu}^\lambda$ on $H_\mu$ satisfying
\[
\bar D_{m, \mu}^\lambda = D_{1, \mu}^{\lambda m}\,.
\]

When $\nu = 0$, the lowest weight vector for the $H_\mu$ is $u_1(z) :=
(z + i)^{-1}$, and integration by parts shows $\bar D_{m,
  \mu}^\lambda(u_0) \in \C$.  Consider the horocycle flow invariant
functional $D_\mu^+$ defined by
\[
D_\mu^+(f) := \lim_{z \to \infty} f(z) (z + i) \,.
\]
The formulas from Section 2.4 of \cite{FF1} show the basis obtained
from $u_1$ by repeatedly applying the operator $1/2[X - i (U + V)]$ is
orthonormal.  Then formula (43) of \cite{FF1} shows $D_\mu^+$ is
(sharply) in $H_\mu^{-1/2}$.
% (see also the unit disc model described in Claim 2.2 of \cite{T}).
So for $f_m = f \otimes e_m$,
\begin{equation}\label{equa:D_lambda-defined}
  D_{m, \mu}^\lambda (f_m) = \bar D_{m, \mu}^\lambda 
  ( f - D_\mu^+(f) u_1) + D_\mu^+(f) \bar D_{m, \mu}^\lambda(u_1) 
\end{equation}
is defined via the above formula.  Moreover, it follows as in the
third case of Lemma~A.3 of \cite{T} that $D_{m, \mu}^\lambda \in H_{m,
  \mu}^{-1/2}$.

For $\nu \geq 1$, an elementary computation from Lemma~A.3 of \cite{T}
gives
\begin{lemma}\label{lemma:f_decay1}
  Let $\nu \geq 1$, and $f \in H_\mu^\infty$.  Then there is a
  constant $C > 0$ such that for all $z \in \mathbb H_\mu$,
  \[
  |f(z)| \leq C (1 + \im(z)^{\min\{2 - (\nu + 1)/2, -1/2\}}) \| f
  \|_{3} (1 + |z|)^{-2}\,.
  \]
\end{lemma}
Hence, for all $\mu$ such that $\nu \geq 0$, $D_{m, \mu}^\lambda \in
H_{m, \mu}^{-3}$\,.  The following stronger result holds.
\begin{lemma}\label{lemma:D^lambda-reg-discrete}
  Let $m \in \Z$ and $\nu \in \N$.  If $\lambda m \neq 0$, then
  \[
  D^\lambda_{m,\mu} \in \widehat H_{m, \mu}^{0, -(1/2+)}\,.
  \]
  Moreover, if $\lambda m < 0$, then $D^\lambda_{m,\mu} = 0$.  For
  $\lambda m = 0$, we have two cases:
  \[
  \left\{\begin{array}{ll}
      D_{m, \mu}^\lambda \text{ is undefined},& \text{ if } \nu = 0 \\
      D_{m, \mu}^\lambda = 0, & \text{ if } \nu > 0\,.
    \end{array}\right.
  \]
\end{lemma}

The proof will be as in
Lemma~\ref{lemma:regularity-D_lambda-XV-principal}, once we have a
description in Fourier transform of the upper half-plane model.  For
each $x + i y := z \in \mathbb H$, define
\[
\hat{f}^y(\xi) := \int_\R f(z) e^{-i \xi z } dx\,.
\]
Notice that $\bar D_{m, \mu}^\lambda(f) = \hat f^1(m\lambda)$.  By
Lemma~\ref{lemma:f_decay1} and a computation as in
\eqref{equa:D_lambda-defined}, the function $\hat{f}^y(\xi)$ is
defined for $\xi\in \R^*$.
% since it is the Fourier transform of an $L^1(\R)$ function along the
% line $\{\im(z) = y\}$.
By Cauchy's theorem, we get
% we have, as in Lemma~3.4 of \cite{T} that

\begin{lemma}\label{lemma:Fourier_Cauchy}
  Let $\xi \in \R$ and $y_1, y_2 > 0$.  Let $f \in H_\mu^\infty$.
  Then $\hat{f}^{y_1}(\xi) = \hat{f}^{y_2}(\xi)\,.$ If $\xi < 0$, then
  $\hat{f}^{y_1}(\xi) = 0$, and if $\nu \geq 1$, then $\hat f^{y_1}(0)
  = 0$.
\end{lemma}

With this in mind, we define the Fourier transform of $f$ to be
\[
\hat f := \hat f^1.
\]
\begin{lemma}\label{lemma:inversion}
  Let $\nu \in \Z^+$ and $f \in H_\mu^\infty$.  Then for all $z \in
  \mathbb H$,
  \[
  f(z) = \frac{1}{2\pi} \int_{\R^+} \hat f(\xi) e^{i \xi z } d\xi\,.
  \]
  % Moreover, when $n = 1$, we have
  % \[
  % \| f \|_0^2 = \frac{1}{2\pi} \int_{\R^+} |\hat f(\xi)|^2 d\xi.
  % \]
  Setting $(-1)! := 1,$ we get for any $\nu \in \N$,
  \[
  \| f\|_{0}^2 = \frac{(\nu - 1)!}{\pi 2^{\nu + 1}} \int_{\R^+} |\hat
  f(\xi)|^2 \frac{d\xi}{\xi^{\nu}}\,.
  \]
\end{lemma}
We leave the proof of Lemma~\ref{lemma:inversion} to
Appendix~\ref{appe:B}.  There is also a formula for Sobolev norms of
functions in Fourier transform.

\begin{lemma}\label{lemma:discrete_Sobolev_Fourier}
  Let $s \geq 0$.  Setting $(-1)! := 1$, we have for any $\nu \in
 \N$,
  \[
  \vert f \vert_{r, s; \mathcal T}^2
 %\left\{
 %  \begin{array}{ll}
 %    \int_{\R^+} |(I - \hat X_T^2 - \hat V_T^2)^{s/2} \hat f(\xi) d\xi, &\text{ if }n = 1 \\
  = \frac{(\nu - 1)!}{\pi 2^{\nu + 1}} (1 + \mu^2)^{\frac{r}{2}}
  \int_{\R^+} |[I + \mu_{\mathcal T}^2 + (m^2 I - \hat{X}^2 -
  \hat{V}^2)^2)]^{\frac{s}{4}} \hat f(\xi)|^2
  \frac{d\xi}{\xi^{\nu}}\,.
 %, & \text{ if }n \geq 2 .
 % \end{array}
 % \right.
  \]
\end{lemma}
\begin{proof}

  The usual formulas
  \begin{equation}\label{equa:X,V-discrete}
    \hat{X} := (1 - \nu) + 2 \xi \frac{\partial}{\partial \xi}, \ \ \ \ \ \ \ \hat{V} := -i \left((1 - \nu) \frac{\partial}{\partial \xi} + \xi \frac{\partial^2}{\partial \xi^2}\right)\,. 
  \end{equation}
  are verified on test functions $g\in H^\infty$ that satisfy $\hat g
  \in C_0^\infty(\R^+)$.  This set is dense in $H$ by
  Lemma~\ref{lemma:inversion}.  Thus, the identity holds.
\end{proof}

\begin{proof}[Proof of Lemma~\ref{lemma:D^lambda-reg-discrete}]
  If $\lambda m < 0$, then $\bar D_{m, \mu}^\lambda = 0$ by Lemma
  \ref{lemma:Fourier_Cauchy}, which implies $D_{m, \mu}^\lambda = 0$.
  Similarly, $D_{m, \mu}^\lambda = 0$ when $\lambda m = 0$ and $\nu
  \geq 1$.  If $\lambda m = 0$ and $\nu = 0$, then $D_{m,
    \mu}^\lambda$ is not defined on the vector $u_1(z) = (z +
  i)^{-1}$, so $D_{m, \mu}^\lambda$ is not defined.  The regularity
  statement follows as in
  Lemma~\ref{lemma:regularity-D_lambda-XV-principal}.
\end{proof}

\subsubsection{Twisted cohomological equations}

For every $\lambda\in \R^*$ we study the solution $g$ to the twisted
cohomological equation
\begin{equation}\label{equa:coeqn-discrete}
  \mathcal T(U + \lambda K) g= f 
\end{equation} 
in every irreducible, unitary representation subspace of the foliated
Sobolev space $\widehat W_\mathcal T^s(M \times \T)$ of the discrete
series or mock discrete series.

\begin{theorem}
  \label{theo:cohomology-discrete}
  For every $r$, $s \geq 0$, there is a constant $C_{r,s} > 0$ such
  that for any irreducible unitary representation $H := H_{m, \mu}$ in
  the discrete series with $m\not=0$ and for any function $f_m \in
  \widehat H^\infty \cap \text{Ann}(D^\lambda_{m,\mu})$, there is a
  unique solution $g_m \in H$ satisfying~\eqref{equa:coeqn}, and
  moreover, for all $\mathcal T \geq 1$,
  \[
  \vert g_m \vert_{r, s; \mathcal T} \leq \frac{C_{r,s}}{{\mathcal
      T}^{1/3}} \frac{1 + \vert\lambda m\vert^{-s}}{\vert \lambda
    m\vert} \, \vert f_m \vert_{r+3s, s+1; \mathcal T} \,.
  \]
  % In the full Sobolev norm, we have
  % \[
  % \| g_m \|_{s} \leq \frac{ C_s }{ \vert \lambda m\vert^{2s + 1}}
  % (1+\vert \lambda m\vert^{-s}) \|
  % f_m \|_{4s+1} \,.
  % \]
\end{theorem}

As in the proof of Theorem~\ref{theo:cohomology-principal}, we proceed
formally and note that $f_m = f \otimes e_m$ and $g_m = g \otimes e_m$
are simple tensors.  By Lemma~\ref{lemma:Fourier_Cauchy}, $\hat f$ and
$\hat g$ are functions supported on $\R^+$.  As in the derivation of
the formulas $\hat X$ and $\hat V$ in \eqref{equa:X,V-discrete}, we
use Lemma~\ref{lemma:inversion} again to see that $\hat U$ is
multiplication by $i \xi$.

Then we may restrict our considerations of the cohomological equation
\eqref{equa:coeqn-discrete} to
\begin{equation}\label{equa:discrete-coeqn}
  \hat g(\xi) := -i \frac{\hat f(\xi)}{\lambda(\xi + 1)}\,,
\end{equation}
by the same argument used in the proof of
Theorem~\ref{theo:cohomology-principal}\,.

\begin{lemma}
  Theorem~\ref{theo:cohomology-discrete} is true when the
  representation $H_{m, \mu}$ is a mock discrete representation.
\end{lemma}
\begin{proof}
  By Lemma~\ref{lemma:inversion}, $\widehat H_{m, \mu}$ consists of
  square integrable functions supported on $\R^+$, and the measure is
  Lebesgue.  Because the formulas for $\hat X$ and $\hat V$ are the
  same, the lemma follows identically as in the proof of
  Theorem~\ref{theo:cohomology-principal}.
\end{proof}

In what follows, we only consider discrete series representations
where $\nu \geq 1$.  As above, we separately estimate $\hat g$ near
$-1$ and away from $-1$.

Notice that the formulas for the vector fields $\hat X$ and $\hat V$
given in \eqref{equa:X,V-discrete} are identical to those given for
the principal and complementary series.  As in
Lemma~\ref{lemma:transfer-def},
\[
\hat g(\xi) = -\frac{i}{\lambda} \int_0^1 \hat f'( -1 + t( \xi + 1))
dt \,.
\]
Let $I = [-\frac{3}{2}, - \frac{1}{2}]$ and $L_\nu^2(I)$ be defined as
in formulas~\eqref{equa:def_I-princ} and \eqref{equa:L^2_nu-def}.

\begin{lemma}\label{lemma:XV-away-lambda-discrete}
  Let $\mu \leq 0$.  For every $\alpha$, $\beta \in \N$, there exists
  a constant $C'_{\alpha, \beta}>0$ such that
  \[
  \| \hat X ^\alpha \hat V ^\beta \hat g \|_{L^2_\nu(\R\setminus I)}
  \leq \frac{C'_{\alpha, \beta}}{\vert\lambda\vert}
  \sum_{\substack{i+j+k \leq \alpha +\beta\\ j\leq \beta}} \vert 1-
  \vert \nu\vert \vert^i \Vert X ^j V ^k f\Vert_0 \,.
  \]
\end{lemma}
\begin{proof}
  The proof is identical to that of Lemma~\ref{lemma:3.10}\,.
\end{proof}

Then it remains to prove
\begin{lemma}\label{lemma:discrete-I_lambda}
  Let $\mu \leq 0$.  Then for any $\alpha, \beta \geq 0$, there is a
  constant $C^{(4)}_{\alpha, \beta} > 0$ such that for all $\hat f,
  \hat g \in \widehat H_\mu^\infty$ satisfying
  \eqref{equa:discrete-coeqn},
  \[
  \begin{aligned}
    \| \hat V^\beta \hat X^\alpha \hat g \|_{L_\nu^2(I)} & \leq
    \frac{C^{(4)}_{\alpha, \beta}}{|\lambda|} (1 + |\nu|)^{3\beta}
    \sum_{ \substack {i + j+ k\leq \alpha + \beta + 1 \\ j \leq
        \beta}} \vert 1 -\nu \vert^i \|V^j X^k f \|_{0} \,.
  \end{aligned}
  \]
\end{lemma}

This is not immediate from the proof of Lemma~\ref{lemma:V-X-g-f},
because the factor $2^{\re\,\nu}$ in formula
\eqref{equa:0-norm-coordinates} can be arbitrarily large.  For $\nu
\leq 1+ 2\beta$, $2^{\re\,\nu}$ is bounded by a constant depending on
$\beta$, and the proof of Lemma~\ref{lemma:V-X-g-f} holds, so
Lemma~\ref{lemma:discrete-I_lambda} follows for this case.  For $\nu >
1+2\beta$, we will move the problem to the setting of the principal
series where $\re\,\nu = 0$.  Define
\begin{equation}\label{equa:eliminate-nu}
  \mathcal A : \widehat H_\mu^\infty \to L^2(\R^+) : \hat f \to \frac{\hat f(\xi)}{\xi^{\nu/2}}\,.
\end{equation}
Notice also that $\mathcal A$ is invertible, where $\mathcal A^{-1} :
\hat f \to \xi^{\nu/2} \hat f$ \,.

As a first step, we have
\begin{lemma}\label{lemma:discrete-principal}
  Let $\mu \leq 0$.  Then for any $\alpha, \beta \in \N$, there is a
  constant $C_{\alpha, \beta}^{(5)} > 0$ such that for all $\hat f,
  \hat g \in \widehat H_\mu^\infty$ satisfying
  \eqref{equa:discrete-coeqn},
  \[
  % \| \hat V^{\beta} \hat X^\alpha A \hat g\|_{L^2(\R^+)} \leq C_{s,
  % \nu, \lambda} \| (I - \hat X^2
  % - \hat V^2)^{(s + 1)/2} A \hat f \|_{L^2(\R^+)}\,,
  \begin{aligned}
    \| \hat V^\beta \hat X^\alpha \mathcal A \hat g \|_{L^2(I)} \leq &
    \frac{C^{(5)}_{\alpha, \beta}}{\vert \lambda\vert} (1+\vert \nu
    \vert)^{\beta} \sum_{ \substack{
        i+j+k \leq \alpha+\beta+1 \\
        j \leq \beta }} (1 + |\nu|)^i \| \hat V^j \hat X^k \mathcal A
    \hat f \|_{L^2(I)} \,.
  \end{aligned}
  \]

  % where
  % \[
  % C_{s, \nu, \lambda} := \frac{C_s}{|\lambda|} (1 + |\nu|)^s (1 +
  % |\lambda|^{-s}) > 0\,.
  % \]
\end{lemma}
\begin{proof}
  Formula \eqref{equa:I_lambdat-I_lambda} in the case $\re\,\nu = 0$
  gives a constant $K_{\alpha, \beta} > 0$ such that
  \[
  \begin{aligned}
    \| \hat V^\beta \hat X^\alpha \mathcal A \hat g \|_{L^2(I)} &\leq
    \frac{K_{\alpha, \beta}}{\vert \lambda\vert}
    (1+\vert \nu \vert)^{\beta} \notag \\
    & \times\sum_{ \substack{
        i+j+k \leq \alpha+\beta+1 \\
        i \leq \beta }}
    \| \hat V^i \hat X^j (\hat X-(1-\nu))^k \mathcal A \hat f \|_{L^2(I)} \notag \\
    & \leq \frac{C''_{\alpha, \beta}}{\vert \lambda\vert} (1+\vert \nu
    \vert)^{\beta} \sum_{ \substack{
        i+j+k \leq \alpha+\beta+1 \\
        j \leq \beta }} (1 + |\nu|)^i \| \hat V^j \hat X^k \mathcal A
    \hat f \|_{L^2(I)} \,.
  \end{aligned}
  \]
\end{proof}
  
Lemma~\ref{lemma:discrete-I_lambda} will then be obtained by
estimating the norm of the linear operators $\mathcal A$ and $\mathcal
A^{-1}$ on foliated Sobolev spaces.  For any $\alpha \in \Z^+$,
formula~\eqref{eq:X_Leibniz} gives universal coefficients
$(a_w^{(\alpha)})$ such that
% \left\{
\begin{align}
  \hat X^\alpha (\mathcal A \hat f) & = \xi^{-\nu/2} \sum_{w = 0}^\alpha a_w^{(\alpha)} (1 - 2 \nu)^w (\hat X - (1 - \nu))^{\alpha - w} \hat f\,. \label{equa:X-AF1}\\
  \hat X^\alpha (\mathcal A^{-1} \hat f) & = \xi^{\nu/2} \sum_{w =
    0}^\alpha a_w^{(\alpha)} (\hat X - (1 - \nu))^{\alpha - w} \hat
  f\,. \label{equa:X-AF2}
\end{align}
% \right.  By \eqref{eq:Vder_xi}), we have
% \[
% \begin{aligned}
%   \hat V^l(\frac{1}{\xi^{\nu}}) & = \hat V^l \left(\frac{\hat V^{\nu - 1}(\frac{1}{\xi})}{(-i)^{\nu - 1} (\nu - 1)! \prod_{j=1}^{n - 1} (j+\nu) } \right)  \\
%   & = (-i)^{l + 1} \frac{(\nu + l - 1)!}{(\nu - 1)!} \prod_{j =
%   \nu}^{l + \nu - 1} (j + \nu) \frac{1}{\xi^{l + \nu}}
% \end{aligned}
% \]
% As in Lemma~\ref{lemma:discrete_Sobolev_Fourier}, set $(-1) := 1$.
As in formula~\eqref{eq:Vder_xi}, for any integer $0 < k < \nu/2$ we
have
\[%\begin{equation}\label{equa:V-xi^nu}
\begin{aligned}
  \hat V^k (\frac{1}{\xi^{\nu/2}}) & = (-i)^k \prod_{j = 0}^{k - 1}
  (\nu/2 + j)(3\nu/2 + j) % \prod_{j = 0}^{k - 1} (3\nu/2 + j)
  \xi^{-(\nu/2 + k)} \notag \\
  \hat V^k (\xi^{\nu/2}) & = i^k \prod_{j = 0}^{k - 1} [(\nu/2)^2 -
  j^2] \xi^{\nu/2 - k} \,.
\end{aligned}
\]%\end{equation}

With this and Lemma~\ref{lemma:V_Leibniz}, we get, for any integer $0
\leq \beta < \nu/2$, universal coefficients $(b_{l, j, k, m}^{(\beta),
  '})$ such that
\begin{align}
  \hat V^\beta(\mathcal A \hat f) = \sum_{ \substack {
      l+j+ m\leq \beta \\
      k \leq m \\
      j \leq \beta }}
  b^{(\beta), '}_{ljkm}  \prod_{\tilde l = 0}^{l - 1} (\nu/2 + \tilde l)(3\nu/2 + \tilde l) \prod_{\tilde m = 0}^{m - 1}(\nu/2 + \tilde m) \  \notag  \\
  \times \xi^{-(\nu/2 + l + m)} (\hat X- (1-\nu))^k \hat V^j \hat f \,; \label{equa:V-Af1} \\ \notag \\
  \hat V^\beta(\mathcal A^{-1} \hat f) = \sum_{ \substack {
      l+j+ m\leq \beta \\
      k \leq m \\
      j \leq \beta }}
  b^{(\beta), '}_{ljkm} (-1)^l  \prod_{\tilde l = 0}^{l - 1} [(\nu/2)^2  - \tilde l^2] \prod_{\tilde m = 0}^{m - 1}(\nu/2 + \tilde m)  \notag  \\
  \times \xi^{\nu/2 - l - m} (\hat X- (1-\nu))^k \hat V^j \hat f
  \,.\label{equa:V-Af2}
\end{align}

\begin{lemma}\label{lemma:norm-A^-1}
  Let $f \in \widehat H_\mu^\infty$.  Then for any $\alpha\in \N$ and
  integer $0 \leq \beta <\frac{\nu}{2}$, there is a constant
  $C_{\alpha, \beta}^{(6)} > 0$ such that
  \[
  \| \hat V^\beta \hat X^\alpha (\mathcal A \hat f) \|_{L^2(I)} \leq
  C_{\alpha, \beta}^{(6)}\sum_{ \substack {
      i+j+ k\leq \alpha + 2 \beta \\
      j +k \leq \alpha +\beta \\ j\leq \beta }} (1 + |\nu|)^{i} \|\hat
  V^j \hat X^k \hat f \|_{L_\nu^2(I)} \,.
  \]
\end{lemma}
\begin{proof}
  By~\eqref{equa:X-AF1} and~\eqref{equa:V-Af1}, we get universal
  coefficients $(b_{w, l, j, k, m}^{(\alpha, \beta)})$ such that
  \begin{align}\label{equa:VX-discrete1}
    \hat V^\beta &\hat X^\alpha (\mathcal A \hat f) = \hat
    V^{\beta}\left(\sum_{w = 0}^\alpha a_w^{(\alpha)}
      (1 - 2 \nu)^w \xi^{-\nu/2} \cdot (\hat X - (1 + \nu))^{\alpha - w}(\hat f)\right) \notag \\
    & = \sum_{w = 0}^\alpha a_w^{(\alpha)} (1 - 2 \nu)^w \cdot \hat
    V^{\beta}
    \left(\xi^{-\nu/2} \cdot (\hat X - (1 + \nu))^{\alpha - w}(\hat f)\right) \notag \\
    & = \sum_{w = 0}^\alpha \sum_{ \substack {
        l+j+ m\leq \beta \\
        k \leq m }} b_{w, l, j, k, m}^{(\alpha, \beta)} (1 - 2 \nu)^w
    \prod_{\tilde l = 0}^{l - 1}
    (\nu/2 + \tilde l)(3\nu/2 + \tilde l) \prod_{\tilde m = 0}^{m - 1}(\nu/2 + \tilde m) \notag  \\
    & \ \ \ \ \ \ \ \ \ \ \ \ \ \ \ \ \ \ \ \ \ \ \ \ \ \ \ \times
    (\frac{1}{\xi^{l + m + \nu/2}}) (\hat X- (1-\nu))^k \hat V^j(\hat
    X - (1 - \nu))^{\alpha - w}(\hat f) \notag
  \end{align}

  By the commutation relation $[\hat X, \hat V] = - 2\hat V$, it
  follows that there are constants $C_{\alpha, \beta} > 0$ such that
  \[
  \begin{aligned}
    \| \hat V^\beta \hat X^\alpha &(\mathcal A \hat f) \|_{L^2(I)}
    \leq C_{\alpha, \beta} \sum_{w = 0}^\alpha \sum_{ \substack {
        l+j+ m\leq \beta \\
        k \leq m }}
    (1 + |\nu|)^{w + 2l + m}  \notag  \\
    & \qquad \times \|\frac{1}{\xi^{l + m + \nu/2}}
    (\hat X- (1-\nu))^k \hat V^j(\hat X - (1 - \nu))^{\alpha - w} \hat f \|_{L^2(I)} \notag \\
    & \leq C_{\alpha, \beta} \sum_{w = 0}^\alpha \sum_{ \substack {
        l+j+ m\leq \beta \\
        k \leq m }}
    (1 + |\nu|)^{w + 2l + m}  \notag  \\
    & \qquad \times \sum_{\tilde k \leq k}
    \|\hat V^j(\hat X - (1 - \nu))^{\alpha + \tilde k - w} \hat f \|_{L_\nu^2(I)} \notag \\
    & \leq C_{\alpha, \beta} \sum_{ \substack {
        i+j+ k\leq \alpha + 2 \beta \\
        j+k \leq \alpha +\beta \\ j\leq \beta }}
    (1 + |\nu|)^{i} \|\hat V^j(\hat X - (1 - \nu))^{k} \hat f \|_{L_\nu^2(I)} \notag \\
    & \leq C_{\alpha, \beta} \sum_{ \substack {
        i+j+ k\leq \alpha + 2 \beta \\
        j +k \leq \alpha +\beta \\ j\leq \beta }} (1 + |\nu|)^{i}
    \|\hat V^j \hat X^k \hat f \|_{L_\nu^2(I)} \,.
  \end{aligned}
  \]
\end{proof}

Similarly, we have
\begin{lemma}\label{lemma:norm-A}
  Let $\mathcal A^{-1} \hat h \in \widehat H_\mu^\infty$.  Then for any
  $\alpha\in \N$ and integer $0 \leq \beta <\frac{\nu}{2}$, there is a
  constant $C_{\alpha, \beta}^{(7)} > 0$ such that
  \[
  \| \hat V^\beta \hat X^\alpha (\mathcal A^{-1} \hat h)
  \|_{L_\nu^2(I_\lambda)} \leq C_{\alpha, \beta}^{(7)} \sum_{
    \substack {
      i+j+ k\leq \alpha + 2 \beta \\
      j +k \leq \alpha +\beta \\ j\leq \beta }} (1 + |\nu|)^{i} \|\hat
  V^j \hat X^k \hat h \|_{L^2(I_\lambda)} \,.
  \]
\end{lemma}
\begin{proof}
  By formulas~\eqref{equa:X-AF2} and~\eqref{equa:V-Af2} in place of
  \eqref{equa:X-AF1} and \eqref{equa:V-Af1}, we get
  \[
  \begin{aligned}
    \hat V^\beta \hat X^\alpha (\mathcal A^{-1} \hat h) & = \sum_{w =
      0}^\alpha \sum_{ \substack {
        j+ m\leq \beta \\
        k \leq m }} b^{(\alpha, \beta)}_{wljkm} (-1)^l \prod_{\tilde l
      = 0}^{l - 1}
    [(\nu/2)^2  - \tilde l^2] \prod_{\tilde m = 0}^{m - 1}(\nu/2 + \tilde m)  \notag  \\
    & \times \xi^{\nu/2 - m - l}  (\hat X- (1-\nu))^k \hat V^j(\hat X - (1 - \nu))^{\alpha - w}(\hat h) \,. \\
  \end{aligned}
  \]

  Then as in the proof of Lemma~\ref{lemma:norm-A^-1}, there are
  constants $C_{\alpha, \beta} > 0$ such that
  \[
  \begin{aligned}
    \|\hat V^\beta \hat X^\alpha &(\mathcal A^{-1} \hat h)
    \|_{L_\nu^2(I_\lambda)} \leq C_{\alpha, \beta} \sum_{w = 0}^\alpha
    \sum_{ \substack {
        l+j+ m\leq \beta \\
        k \leq m }}
    (1 + |\nu|)^{w + 2l + m}  \notag  \\
    & \quad \times \|\xi^{\nu/2 - l - m} (\hat X- (1-\nu))^k
    \hat V^j(\hat X - (1 - \nu))^{\alpha - w} \hat h \|_{L_\nu^2(I_\lambda)} \notag \\
    & \leq C_{\alpha, \beta} \sum_{ \substack {
        i+j+ k\leq \alpha + 2 \beta \\
        j +k \leq \alpha +\beta \\ j\leq \beta }} (1 + |\nu|)^{i}
    \|\hat V^j(\hat X - (1 - \nu))^{k} \hat f \|_{L^2(I_\lambda)} \,.
  \end{aligned}
  \]
\end{proof}

\begin{proof}[Proof of Lemma \ref{lemma:discrete-I_lambda}] 
  Let $h := \mathcal A g$, and observe from
  \eqref{equa:discrete-coeqn} that
  \[
  \hat h(\xi) = -i \frac{\mathcal A \hat f(\xi)}{\lambda(\xi + 1)}
  \]
  Then by Lemma~\ref{lemma:norm-A},
  Lemma~\ref{lemma:discrete-principal} and
  Lemma~\ref{lemma:norm-A^-1}, we get
  \[
  \begin{aligned}
    \| \hat V^\beta& \hat X^\alpha  \hat g \|_{L_\nu^2(I_\lambda)}  =  \| \hat V^\beta \hat X^\alpha \mathcal A^{-1} \hat h \|_{L_\nu^2(I_\lambda)}  \leq C_{\alpha, \beta}^{(7)}  \notag \\
    & \qquad\qquad\qquad\qquad \times \sum_{ \substack{i+j+ k\leq
        \alpha + 2 \beta \\ j + k\leq \alpha +\beta \\ j\leq \beta }}
    (1 + |\nu|)^{i} \|\hat V^j \hat X^k \hat h \|_{L^2(I_\lambda)} \notag \\
    & \leq \frac{C_{\alpha, \beta}^{(7)} C_{\alpha,
        \beta}^{(5)}}{|\lambda|} \sum_{ \substack {
        i + j+ k\leq \alpha + 2 \beta \\
        j +k\leq \alpha + \beta \\ j\leq \beta }} (1 + |\nu|)^i
    \\
    & \qquad\qquad\qquad\qquad \times \sum_{ \substack {
        i' + j'+ k'\leq j + k + 1 \\
        j' \leq j }}
    (1 + |\nu|)^{i'} \|\hat V^{j'} \hat X^{k'}  \mathcal A\hat f \|_{L_\nu^2(I_\lambda)} \notag \\
    & \leq \frac{C_{\alpha, \beta}^{(7)}C_{\alpha,
        \beta}^{(5)}C_{\alpha, \beta}^{(6)}}{|\lambda|} (1 +
    |\nu|)^\beta \sum_{ \substack {
        i + j+ k\leq \alpha + 2 \beta \\
        j +k \leq \alpha + \beta \\ j\leq \beta }} (1 + |\nu|)^i
    \\
    & \ \times \sum_{ \substack {
        i' + j'+ k'\leq j + k + 1 \\
        j' \leq j }} (1 + |\nu|)^{i'} \sum_{ \substack {
        i'' + j''+ k''\leq k' + 2 j'  \\
        j'' +k'' \leq j' +k' \\ j'' \leq j' }} (1 + |\nu|)^{i''}
    \|\hat V^{j''} \hat X^{k''} \hat f \|_{L_\nu^2(I_\lambda)}\,.
  \end{aligned}
  \]
  In the last summation of the above formula we have
  \[
  \begin{aligned}
    i + i' + i'' + j'' + k'' &\leq i+ i' + k' +2j' \leq i+ j+ k + \beta + 1\leq \alpha + 3 \beta + 1\,, \\
    j'' + k'' &\leq j'+k' \leq j+k +1 \leq \alpha +\beta +1 \,.
  \end{aligned}
  \]
  This concludes the proof of Lemma~\ref{lemma:discrete-I_lambda}.
\end{proof}

For a general $\lambda \in \R^*$ and for rescaled Sobolev norms we
have a similar statement which can be immediately derived from
Lemma~\ref{lemma:XV-away-lambda-discrete},
Lemma~\ref{lemma:discrete-I_lambda} and from a calculation similar to
the one in formula~\eqref{equa:1->lambda} in the proof of
Lemma~\ref{lemma:XVderivatives}.

\begin{lemma}
  \label{lemma:discrete-I_lambda_scaled}
  Let $\mu \leq 0$. For every $\alpha$, $\beta \geq 0$ there exists a
  constant $C^{(4)}_{\alpha, \beta}>0$ such that for all $\lambda \not
  =0$ and for all $\mathcal T\geq 1$, the solution $g\in L^2_\nu(\R)$
  of the rescaled equation $\mathcal T(U+i\lambda) g=f \in L^2_\nu(\R)
  $ satisfies the following estimate with respect to the rescaled
  Sobolev norms:
  \[
  \begin{aligned}
    \| V_{\mathcal T}^\beta X_{\mathcal T}^\alpha g \|_0 & \leq \frac{
      C^{(4)}_{\alpha, \beta} }{ \vert \lambda\vert {\mathcal T}^{1/3}
    } (1+\vert \nu \vert)^{3\beta} (1+\vert \lambda\vert^{-\beta}) \\
    & \times \sum_{i+j+k \leq \alpha+\beta+1 } [\mathcal T^{-1/3}
    \vert 1- \nu \vert]^i \| V_{\mathcal T}^j X_{\mathcal T}^k f \|_0
    \,.
  \end{aligned}
  \]
\end{lemma}
Now we may prove Theorem~\ref{theo:cohomology-discrete}.
\begin{proof}[Proof of Theorem~\ref{theo:cohomology-discrete}]
  We proceed as in the proof of
  Theorem~\ref{theo:cohomology-principal}.  We claim that, for $r$, $s
  \in \N$ even integers, there exist constants $C^{(1)}_{r,s}>0$ and
  $C^{(2)}_{r,s}>0$ such that, if the functions $f$ and $g$ belong to
  a single irreducible component, then
  \begin{equation}
    \label{eq:est_integral}
    \begin{aligned}
      \vert g \vert_{r,s; \mathcal{T}} & \leq \frac{C^{(1)}_{r,s}}{
        \mathcal T^{1/3}} \frac{1 +
        \vert\lambda\vert^{-s}}{\vert\lambda\vert} (1 + |\nu|)^{r+3s}
      \sum_{k = 0}^{s + 1} [\mathcal T^{-1/3} \vert 1- \nu \vert]^{k}
      \vert f  \vert_{0, s + 1 - k; \mathcal T} \\
      & \leq \frac{C^{(2)}_{r,s}}{ \mathcal T^{1/3}} \frac{1 +
        \vert\lambda\vert^{-s}}{\vert\lambda\vert} \vert f
      \vert_{r+3s, s + 1; \mathcal T}\,. \end{aligned}
  \end{equation}
  
  Since the Casimir operator $\Box$ takes the value $\mu= 1 - \nu^2$
  on any irreducible, unitary representation $H_{m,\mu}$, by expanding
  the operator $(I +\Box^2)^{r/2} (I+ \Box^2+ \widehat
  \triangle_{\mathcal T}^2)^{s/2}$ into polynomial expression in
  $X_{\mathcal T}$ and $V_\mathcal T$, and by the commutation
  relations, we derive that there exists a constant $C^{(3)}_{r,s}>0$
  such that
  \begin{equation}\label{equa:coeqn-fullnorm1}
    \vert  g \vert_{r,s; \mathcal T}  \leq  C^{(3)}_{r,s}(1+ \mu^2)^{r/4} \sum_{k+m + \alpha +\beta \leq s} 
    [\mathcal T^{-1/3} \vert 1- \nu \vert]^k \| K^m  V_{\mathcal T}^\beta X_{\mathcal T}^\alpha g\|_0 \,.
  \end{equation}
  By the above bound on the norms the estimate in
  formula~\eqref{eq:est_integral} follows directly from
  Lemma~\ref{lemma:discrete-I_lambda_scaled}.  From the estimate
  ~\eqref{eq:est_integral} , the conclusion for $r$, $s$ even integers
  and $m=1$ follows. The statement for $r$, $s \geq 0$ follows by
  interpolation.  As in the proof of
  Theorem~\ref{theo:cohomology-principal}, the estimate for $m \neq 1$
  follows by setting $\lambda = \lambda m$.

  The uniqueness of the solution holds as in the proof of
  Theorem~\ref{theo:cohomology-principal}.
\end{proof}

We can now prove Theorem~\ref{theo:Invariant_dist} on the
classification of invariant distributions and
Theorem~\ref{theo:cohomological_eqn} on Sobolev bounds for solutions
of the cohomological equation for the twisted horocycle flow.

\begin{proof}[Proof of Theorem~\ref{theo:Invariant_dist}]
  It follows from Theorem~\ref{theo:cohomology-principal} and
  Theorem~\ref{theo:cohomology-discrete} that the space of invariant
  distributions is one dimensional.  The regularity part of the
  statement follows from
  Lemma~\ref{lemma:regularity-D_lambda-XV-principal} and
  Lemma~\ref{lemma:D^lambda-reg-discrete}.
\end{proof}

\begin{proof}[Proof of Theorem~\ref{theo:cohomological_eqn}] 
  The bounds with respect to the foliated Sobolev norms follows from
  Theorems~\ref{theo:cohomology-principal}
  and~\ref{theo:cohomology-discrete} by orthogonality since the above
  estimates are uniform with respect to the Casimir parameter.
 
  The bounds with respect to the Sobolev norms can be proved as
  follows.  First let $s \in \N$ be even.  Since the vector fields $U$
  and $K$ commute, for all $j\in \N$ we have
  \[
  (U + \lambda K) U^j g = U^j f\,,
  \]
  hence the bound with respect to foliated Sobolev norms holds for the
  functions $U^j g$ in terms of the function $U^j f$, for all $j\in
  \mathbb N$.
  % We may expand the operator $(I + \triangle)^{s/2}$ and apply the
  % triangle inequality to bound
  % $\Vert g\Vert_s$ by a sum of terms of the form $\Vert B g
  % \Vert_0$, where $B$ is a differential
  % operator of order $s$ in the basis $X_{i_j} \in \{K, U, X, V\}$.
  Then
  \[
  \begin{aligned}
    \Vert g \Vert_s & \leq \sum_{j = 0}^s \vert U^j g \vert_{0, s - j} \\
    & \leq \frac{C_{s}}{|\lambda|} \sum_{j = 0}^s (1+\vert \lambda \vert^{-(s - j)}) \vert  U^j f \vert_{3(s - j),  s - j +1} \\
    & \leq \frac{C_{s}}{|\lambda|} (1+\vert \lambda \vert^{-s}) \Vert
    f \Vert_{4s+1} \,.
  \end{aligned}
  \]
  The estimates for general Sobolev norms follows by interpolation.
 
  Finally, the solution is unique in $L^2(M\times\T)$ by
  Theorem~\ref{theo:cohomology-principal} and
  Theorem~\ref{theo:cohomology-discrete} .
\end{proof}

\section{Scaling of invariant distributions}\label{sect:Scale}

In this section we prove estimates on the scaled foliated Sobolev
norms of invariant distributions for the twisted cohomological
equation.

Let us consider an irreducible, unitary representation $H:=H_{m,\mu}$
of the group $\SL(2, \R)\times \T$ and let $\lambda m \neq 0$.  For
${\mathcal T} > {\mathcal T}'$, we want to estimate
\[
\vert D^\lambda_{m,\mu}\vert_{-r,-s; \mathcal{T}} := \sup_{F \in \hat
  H^s}\left\{|D^\lambda_{m,\mu}(F)| : \vert F \vert_{r,s; \mathcal{T}}
  = 1\right\}
\]
in terms of $ \vert D^\lambda_{m,\mu}\vert_{-r,-s; \mathcal{T}'} :=
\sup_F\left\{|D^\lambda_{m,\mu}(F)| : \vert F \vert_{r,s;
    \mathcal{T}'} = 1\right\}\,.  $ We introduce the following
foliated Sobolev Lyapunov norms on $\widehat H^s$.  For all ${\mathcal
  T}\geq 1$,
\[
\vert F \vert^{\mathcal L}_{r,s; \mathcal{T}} := \sup_{\tau>{\mathcal
    T}} \, (\frac{\tau}{{\mathcal T}})^{1/6} \vert F \vert_{r,s; \tau}
\,.
\]
From definitions, the following holds.
\begin{lemma} \label{lemma:sob_est_one} Let $r$, $s \geq 0$.  For all
  ${\mathcal T}>{\mathcal T}'\geq 1$ and for all $F\in H^s$,
  \[
  \begin{aligned}
    \vert F \vert_{r,s; \mathcal{T}}  &\leq \vert F \vert^{\mathcal L}_{r,s; \mathcal{T}}   \,; \\
    \vert F \vert^{\mathcal L}_{r,s; \mathcal{T}} & \leq
    (\frac{{\mathcal T}'}{{\mathcal T}})^{1/6} \vert F \vert^{\mathcal
      L}_{r,s; \mathcal{T}'} \,.
  \end{aligned}
  \]
\end{lemma}

Lemma~\ref{lemma:sob_est_one} immediately gives
\begin{corollary}\label{coro:Sobolev-Lyapunov}
  Let $r$, $s > 1/2$ and ${\mathcal T} > {\mathcal T}' > 1$.  Then
  \[
  \begin{aligned}
    \vert D^\lambda_{m,\mu} \vert^{\mathcal L}_{-r,-s; \mathcal{T}} &\leq    \vert D^\lambda_{m,\mu} \vert_{-r,-s; \mathcal{T}} \,, \\
    \vert D^\lambda_{m,\mu} \vert^{\mathcal L}_{-r,-s; \mathcal{T}'} &
    \leq (\frac{{\mathcal T}'}{{\mathcal T}})^{1/6} \vert
    D^\lambda_{m,\mu} \vert^{\mathcal L}_{-r,-s; \mathcal{T}} \,.
  \end{aligned}
  \]
\end{corollary}
Our strategy is to prove comparison bounds between the foliated
Sobolev dual norms and the foliated Sobolev Lyapunov dual norms of the
invariant distribution in every irreducible, unitary representation.

Hence, it remains to prove a bound from above for the foliated Sobolev
norm $\vert D^\lambda_{m,\mu} \vert_{-r,-s; \mathcal{T}}$ in terms of
the foliated Sobolev Lyapunov norm $\vert D^\lambda_{m,\mu}
\vert^{\mathcal L}_{-r,-s; \mathcal{T}}$.  We consider the principal
and complementary series together, while the discrete series is
handled separately.

\smallskip
\subsection{Principal and complementary series}\label{sect:principal}
Throughout this subsection, given an integer $m\in \Z$ and a Casimir
parameter $\mu>0$, we let $H:=H_{m,\mu}$ be an irreducible, unitary
representation of the principal or complementary series for the group
$\SL(2, \R) \times \T$.  We prove the following theorem.

\begin{theorem}\label{theo:scaling-principal}
  For every $r\geq 0$ and $s>1/2$ there is a constant $C_{r,s} > 0$
  such that for all ${\mathcal T} \geq {\mathcal T}' \geq 1$ and
  $\lambda \in \R$ such that $\lambda m \neq 0$, the distribution
  $D^\lambda_{m,\mu} \in \widehat H^{-r,-s}$ satisfies the scaling
  estimates
  \[
  \vert D^\lambda_{m,\mu} \vert_{-(r+s),-s; \mathcal{T}'} \leq C_{r,s}
  (\frac{{\mathcal T}'}{{\mathcal T}})^{1/6} (1+ \vert\lambda m
  \vert^{-2s}) \vert D^\lambda_{m,\mu} \vert_{-r,-s; \mathcal{T}}\,.
  \]
\end{theorem}

By Corollary~\ref{coro:Sobolev-Lyapunov}, it is enough to prove the
following proposition.

\begin{proposition} \label{prop:principal_scale} For every $r\geq 0$
  and $s>1/2$ there is a constant $C_{r,s} > 0$ such that for all
  ${\mathcal T} \geq 1$ and $\lambda \in \R$ such that $\lambda m \neq
  0$, the distribution $D^\lambda_{m,\mu} \in \widehat H^{-r,-s}$
  satisfies
  \[
  \vert D^\lambda_{m,\mu} \vert_{-(r+s),-s; \mathcal{T}} \leq C_{r,s}
  (1+ \vert\lambda m \vert^{-2s}) \vert
  D^\lambda_{m,\mu}\vert^{\mathcal L}_{-r,-s; \mathcal{T}}\,.
  \]
\end{proposition}
Once more the general case can be derived from the case $m=1$. We will
therefore restrict our argument to that case and prove the statement
for the distributions $D^\lambda_{1, \mu}$ whenever $\lambda
\not=0$. We again let
\[
I_\lambda:= [\lambda- \vert \lambda\vert/2, \lambda+ \vert
\lambda\vert/2]\,.
\]

Now let us consider, for all $\tau \geq 1$, the operator $U_\tau$
formally defined on $H$ in Fourier transform as follows:
\begin{equation}\label{equa:U_T-def}
  \hat U_\tau (\hat f) (\xi) =  \tau^{1/6}  \hat f ( \lambda + \tau^{1/3} (\xi-\lambda) )\,,  
  \quad \text{ for all } f\in H\,.
\end{equation}

In fact, it can be proved that the following bounds hold:
\begin{lemma}\label{lemma:comp_UT_est}
  For all $\tau \geq 1$ and $\hat f \in C_0^\infty(I_\lambda)$,
  \[
  \frac{1}{\sqrt 3} \| f \|_0 \leq \|U_\tau f\|_0 \leq \sqrt 3 \| f
  \|_0\,.
  \]
\end{lemma}
Lemma \ref{lemma:comp_UT_est} is proved in Appendix \ref{appe:B}.

We recall from \eqref{equa:vect_coeqn} that the Fourier transforms of
the $X$ and $V$ operators are
\[
\hat{X} := (1 - \nu) + 2 \xi \frac{\partial}{\partial \xi}, \ \ \ \ \
\ \ \hat{V} := -i \left((1 - \nu) \frac{\partial}{\partial \xi} + \xi
  \frac{\partial^2}{\partial \xi^2}\right)\,.
\]

\begin{lemma}\label{lemma:intertwining} 
  The following formulas hold for all $\tau >1$:
  \[
  \begin{aligned}
    \hat U^{-1}_\tau \hat X  \hat U_\tau  &=  \hat X -2 \lambda \tau^{1/3} (1-\tau^{-1/3})  \frac{\partial}{\partial \xi}\,;   \\
    \hat U^{-1}_\tau \hat V \hat U_\tau &= \tau^{-1/3} \hat V + i
    \lambda \tau^{2/3} (1-\tau^{-1/3}) \frac{\partial^2}{\partial
      \xi^2} \,.
  \end{aligned}
  \]
\end{lemma}

We have the following scaling estimate.
\begin{lemma}
  \label{lemma:local_bound}
  For every $r$,$s\geq 0$, there exists a constant $C'_{r,s}>0$ such
  that for all $\lambda \not=0$, for all $\tau \geq {\mathcal T}\geq
  1$ and for all $\hat f\in C_0^\infty(I_\lambda)$ the following bound
  holds:
  \[
  \vert U_{\tau/{\mathcal T}} f \vert_{r,s; \tau} \leq C'_{r,s} (1+
  \vert\lambda\vert^{-s}) \vert f \vert_{r+s, s; \mathcal{T}} \,.
  \]
\end{lemma}
\begin{proof} By Lemma~\ref{lemma:intertwining}, we have
  \begin{equation}
    \label{equa:intertwine X, V}
    \begin{aligned}
      \hat U^{-1}_{\tau/\mathcal T} \hat X_\tau  \hat U_{\tau/\mathcal T}  &= (\frac{\tau}{{\mathcal T}})^{-1/3} \hat X_{\mathcal T} -2   [1- (\frac{\tau}{\mathcal T})^{-1/3}]   {\mathcal T}^{-1/3} \lambda \frac{\partial}{\partial \xi} \,;  \\
      \hat U^{-1}_{\tau/\mathcal T} \hat V_\tau \hat U_{\tau/\mathcal
        T} &= (\frac{\tau}{\mathcal T})^{-1/3} \hat V_{\mathcal T} + i
      [1- (\frac{\tau}{\mathcal T})^{-1/3}] {\mathcal T}^{-2/3}
      \lambda \frac{\partial^2}{\partial \xi^2} \,.
    \end{aligned}
  \end{equation}
  We then observe that
  \begin{equation}
    \label{equa:intertwine_id}
    \begin{aligned}
      {\mathcal T}^{-1/3} \lambda \frac{\partial}{\partial \xi} &=
      \frac{\lambda}{\xi}
      (\hat X_\mathcal T - \mathcal T^{-1/3}(1-\nu) ) \,; \\
      {\mathcal T}^{-2/3} \lambda \frac{\partial^2}{\partial \xi^2} &=
      \frac{\lambda}{\xi} \left(i \hat V_\mathcal T - \frac{(1-\nu)
          \mathcal T^{-1/3}}{2\xi} (\hat X_\mathcal T - \mathcal
        T^{-1/3}(1-\nu)) \right)\,.
    \end{aligned}
  \end{equation}
  We the recall the identities~\eqref{eq:Xder_xi}
  and~\eqref{eq:Vder_xi}. For all integers $s\in \N$ (by induction),
  we have
  \[
  \hat X^s(\frac{1}{\xi}) =(-1 - \nu)^s (\frac{1}{\xi}) \quad \text{
    and } \quad \hat V^s (\frac{1}{\xi}) = (-i)^s s! \left(
    \prod_{j=1}^s (j+\nu) \right) \, (\frac{1}{\xi^{s+1}})\,.
  \]
  It follows that for all $s\in \N$ there exists a constant $C'_s>0$
  such that for all $\hat f\in C_0^\infty(I_\lambda)$ we have
  \[
  \vert \hat U_{\tau/{\mathcal T}} \hat f \vert_{r,s; \tau} \leq C'_s
  (1+ \vert \nu\vert)^{s} (1+ \vert\lambda\vert^{-s}) \vert \hat f
  \vert_{r,s; \mathcal{T}} \leq C'_s (1+ \vert\lambda\vert^{-s}) \vert
  \hat f \vert_{r+s,s; \mathcal{T}} \,.
  \]
  The statement then follows by interpolation.
\end{proof}
\begin{lemma}
  \label{lemma:f_cutoff}
  For every $r$, $s\geq 0$ there exists a constant $C''_{r,s}>0$ such
  that the following holds.  For any function $\hat f\in \widehat
  H^\infty$ and for any $\lambda\not=0$, there exists a function $\hat
  f_\lambda \in C_0^\infty (I_\lambda)$ with $\hat f_\lambda
  (-\lambda)=\hat f(-\lambda)$ such that for any $r$, $s\geq 0$ and
  $\mathcal T \geq 1$ we have
  \[
  \vert f_\lambda \vert_{r,s; \mathcal{T}} \leq C''_{r,s} (1+ \vert
  \lambda \vert^{-s}) \vert f \vert_{r,s; \mathcal{T}} \,.
  \]
\end{lemma}
\begin{proof} Let $\hat \phi\in C_0^{\infty}(-1/2, 1/2)$ be any
  function such that $\hat \phi(0)=1$.  We let
  \[
  \hat \phi_{\lambda}(\xi) := \hat \phi\left( \frac{\xi+\lambda}{\vert
      \lambda\vert}\right)\,, \quad \text{ for all } \, \xi \in \R\,.
  \]
  By construction we have that the function $\hat \phi_{\lambda} \in
  C_0^\infty(I_\lambda)$. By an induction argument based on the
  formulas~\eqref{equa:vect_coeqn} for the Fourier transforms $\hat
  X$, $\hat V$ of the operators $X$, $V$, we derive the following
  bounds. For every $\alpha$, $\beta\in \N$ there exists a constant
  $C''_{\alpha,\beta}>0$ such that
  \[
  \Vert \hat V_\mathcal T^\beta \hat X_\mathcal T^\alpha \hat
  \phi_\lambda\Vert_0 \leq C''_{\alpha,\beta} (1+ \vert
  \lambda\vert^{-\beta}) (1+\mathcal T^{-1/3} \vert 1-
  \nu\vert)^{\alpha+\beta}\,.
  \]
  Let then $\hat f_\lambda:= \phi_\lambda \hat f $. By construction we
  immediately have that $\hat f_\lambda \in C_0^\infty(I_\lambda)$ and
  $\hat f_\lambda(-\lambda)=\hat f(-\lambda)$.  Finally from the
  Leibniz-type formula \eqref{eq:X_Leibniz} and from that of
  Lemma~\ref{lemma:V_Leibniz} we derive that for all $r\geq 0$ and all
  integer $s\in \N$ there exists a constant $C''_{r,s}>0$ such that
  \[
  \vert \hat f_\lambda \vert_{r,s; \mathcal{T}} \leq C''_{r,s} (1+
  \vert \lambda\vert^{-s}) \vert f \vert_{r,s; \mathcal{T}}\,.
  \]
  The estimate in the statement is thus proved for integer exponents
  and follows by interpolation in the general case.
\end{proof}

\begin{proof}[Proof of Proposition \ref{prop:principal_scale}]
  For simplicity of notation, we again let $D^\lambda:=
  D^\lambda_{1,\mu}$.  Let $\hat f_{\lambda} \in
  C^\infty_0(I_\lambda)$ be the function constructed in
  Lemma~\ref{lemma:f_cutoff}. By definition we have
  \[
  D^\lambda ( f) = D^\lambda ( f_{\lambda}) = (\frac{\tau}{\mathcal
    T})^{-1/6} D^\lambda ( U_{\tau/\mathcal T} f_{\lambda}) \,.
  \]
  By Lemma~ \ref{lemma:local_bound} and Lemma~\ref{lemma:f_cutoff}, it
  follows that, for all $\tau\geq {\mathcal T} \geq 1$, we have

  \begin{equation}
    \label{equa:Banach inequality}
    \begin{aligned}
      \vert D^\lambda ( f) \vert & \leq (\frac{\tau}{\mathcal
        T})^{-1/6} \vert D^\lambda \vert_{-r,-s; \tau}
      \vert U_{\tau/\mathcal T} f_{\lambda} \vert_{r, s; \tau}  \\
      & \leq C'_{r,s} (\frac{\tau}{\mathcal T})^{-1/6} \vert D^\lambda
      \vert _{-r,-s; \tau} (1+ \vert\lambda\vert^{-s}) \vert f_\lambda
      \vert_{r+s,s; \mathcal{T}}
      \\
      & \leq C'_{r,s} C''_{r,s} (\frac{\tau}{\mathcal T})^{-1/6} \vert
      D^\lambda \vert_{-r,-s; \tau} (1+ \vert\lambda\vert^{-2s}) \vert
      f \vert_{r+s,s; \mathcal{T}} \,.
    \end{aligned}
  \end{equation}

  Hence, by definition
  \begin{equation}
    \label{equa:bound2}
    \begin{aligned}
      \vert D^\lambda \vert_{-(r+s),-s; \mathcal{T}} &\leq C'_{r,s}
      C''_{r,s} (1+ \vert\lambda\vert^{-2s}) \inf_{\tau >{\mathcal T}}
      (\frac{\tau}{\mathcal T})^{-1/6} \vert D^\lambda \vert _{-r,-s;
        \tau} \\ & \leq C'_{r,s} C''_{r,s} (1+
      \vert\lambda\vert^{-2s}) \vert D^\lambda\vert^{\mathcal
        L}_{-r,-s; \mathcal{T}}\,.
    \end{aligned}
  \end{equation}
\end{proof}

As remarked above, Theorem \ref{theo:scaling-principal} immediately
follows from Corollary~\ref{coro:Sobolev-Lyapunov} and
Proposition~\ref{prop:principal_scale}, hence its proof is complete.

\subsection{Discrete series}
Let $\mathbb H$ be the upper half-plane model for a holomorphic
discrete series or mock discrete series irreducible, unitary
representation of $\SL(2, \R)$ with Casimir parameter $\mu := 1 -
\nu^2,$ where $\nu \geq 0$ is an integer, and let $m \in \Z \setminus
\{0\}$.  We prove the following theorem.
\begin{theorem}\label{theo:Discrete_scale}
  Let $r\geq 0$, $s > 1/2$ and $\lambda \in \R^*$.  Then there is a
  constant $C_{r,s} > 0$ such that for all ${\mathcal T}\geq {\mathcal
    T}' \geq 1$, $D^\lambda_{m,\mu} \in \widehat H^{-r,-s}$ satisfies
  \[
  \begin{aligned}
    \vert D^\lambda_{m,\mu} \vert_{-(r+2s),-s; \mathcal{T}'} \leq
    &C_{r,s} (\frac{{\mathcal T}'}{{\mathcal T}})^{1/6} (1+ \vert
    \lambda m\vert^{-3s}) \vert D^\lambda_{m,\mu}\vert_{-r,-s;
      \mathcal{T}}.
  \end{aligned}
  \]
\end{theorem}
\begin{proof}

  Notice that by Lemma~\ref{lemma:Fourier_Cauchy}, $D_{m, \mu}^\lambda
  = 0$ when $\lambda m < 0$.  Hence, we may assume $\lambda m > 0$.
  By Corollary \ref{coro:Sobolev-Lyapunov} again, it is enough to
  prove the following proposition.

\begin{proposition}\label{prop:discrete_scale}
  Let $r\geq 0$, $s > 1/2$, and $\lambda \in \R^*$ with $\lambda m >
  0$.  There is a constant $C_{r,s} > 0$ such that for all ${\mathcal
    T} \geq 1$, $D^\lambda_{m,\mu} \in \widehat H^{-r,-s}$ satisfies
  \[
  \begin{aligned}
    \vert D^\lambda_{m,\mu} \vert_{-(r+2s),-s; \mathcal{T}} \leq &
    C_{r,s} (1+ \vert \lambda m\vert^{-3s}) \vert D^\lambda_{m,\mu}
    \vert^{\mathcal L}_{-r,-s; \mathcal{T}} \,.
  \end{aligned}
  \]
\end{proposition}

The general case can again be derived from the case $\lambda \in \R^*$
and $m = 1$.  Define the dilation operator $U_\tau$ as in
\eqref{equa:U_T-def}.  We cannot immediately conclude
Proposition~\ref{prop:discrete_scale} from the proof of
Proposition~\ref{prop:principal_scale}, because
Lemma~\ref{lemma:comp_UT_est} does not hold for all $\lambda$ and
$\nu$.  Instead, we have

\begin{lemma}\label{lemma:U_T-comparable}
  Assume $\lambda \geq \nu$, then there is a constant $C > 1$ such
  that for all $\tau \geq 1$ and $\hat f \in C_0^\infty(I_\lambda)$,
  \[
  \frac{1}{C} \| f \|_0 \leq \|U_\tau f\|_0 \leq C \| f \|_0\,.
  \]
\end{lemma}
The proof is left for Appendix \ref{appe:B}.

\begin{proof}[Proof of Proposition \ref{prop:discrete_scale}]
  If $\lambda \geq \nu + 1$, then by Lemma~\ref{lemma:U_T-comparable},
  Proposition~\ref{prop:discrete_scale} follows as in
  Proposition~\ref{prop:principal_scale}.

  Now assume $\lambda < \nu + 1$.  To simplify notation, for any
  $\alpha \in \R^*$, set $D^\alpha := D^\alpha_{1, \mu}$.  Because
  $\lambda > 0$, let $\kappa \in \R^+$ satisfy $\lambda e^{\kappa} =
  \nu + 1$.  Observe that for any $f \in \widehat H^\infty$,
  \[
  D^\lambda(f) = (a_{-\kappa} D^\lambda)(f \circ a_{-\kappa})\,.
  \]
  Moreover,
  \[
  (U + i (\nu + 1)) (a_{-\kappa} D^\lambda) = 0\,.
  \]
  Because the space of invariant distributions in $\mathcal E'(H)$ is
  one-dimensional, we have $a_{-\kappa} D^\lambda \in \langle D^{\nu +
    1} \rangle$, where for any $h \in \widehat H^\infty$, and $D^{\nu
    + 1}$ is defined as usual by
  \[
  D^{\nu + 1}(h) = e^{\nu + 1} \int_\R h(t) e^{-i (\nu + 1) t} dt\,.
  \]

  Observe that Lemma~\ref{lemma:local_bound} and
  Lemma~\ref{lemma:f_cutoff} also holds for the discrete series.  Let
  $f_{\nu + 1} \in H^\infty$ correspond to $f \circ a_{-\kappa}$ as in
  Lemma~\ref{lemma:f_cutoff}.  Then
  \begin{align}\label{equa:push D_lambda}
    |D^{\lambda}(f)| %& = |(a_{-\kappa} D^\lambda_{m,\mu})(f \circ a_{-\kappa})|\\
    & = |(a_{-\kappa} D^\lambda)(f_{\nu + 1})| \nonumber \\
    % & = |(a_{-\kappa} D^\lambda_{m,\mu})(\phi_{\nu} \cdot f \circ
    % a_{-\kappa})| \\
    & = \left(\frac{\tau}{\mathcal T}\right)^{-1/6} |(a_{-\kappa}
    D^\lambda) (U_{\tau/\mathcal T} (f_{\nu + 1}))|\,.
  \end{align}
 
  By Lemma~\ref{lemma:local_bound} and Lemma~\ref{lemma:f_cutoff},
  there exists a constant $C^{(3)}_{r,s}:= C'_{r,s} C''_{r,s} >0$ such
  that
  \begin{align}\label{equa:geodesic_D_lambda}
    |D^{\lambda} (f)| &\leq \left(\frac{\tau}{\mathcal
        T}\right)^{-1/6} \vert a_{-\kappa} D^\lambda\vert_{-r,-s;
      \tau} \vert U_{\tau/\mathcal T} (f_{\nu + 1})\vert_{r,s;
      \tau}\nonumber
    \\
    & \leq C'_{r,s} \left(\frac{\tau}{\mathcal T}\right)^{-1/6} (1+
    \vert\lambda\vert^{-s}) \vert a_{-\kappa} D^\lambda\vert_{-r,-s;
      \tau} \vert f_{\nu+ 1} \vert_{r+s,s; \mathcal{T}}\nonumber
    \\
    & \leq C^{(3)}_{r,s} \left(\frac{\tau}{\mathcal T}\right)^{-1/6}
    (1+ \vert\lambda\vert^{-2s}) \vert a_{-\kappa}
    D^\lambda\vert_{-r,-s; \tau} \vert f \circ
    a_{-\kappa}\vert_{r+s,s; \mathcal{T}}\nonumber
    \\
    & \leq C^{(3)}_{r,s} \left(\frac{\tau}{\mathcal T}\right)^{-1/6}
    (1+ \vert\lambda\vert^{-2s}) \vert a_{-\kappa}
    D^\lambda\vert_{-r,-s; \tau} \vert f \vert_{r+s,s; \mathcal{T}}\,,
  \end{align}
  where the last inequality holds because
  \begin{equation}\label{equa:V_der-geodesic}
    \vert V(f \circ a_{-\kappa}) \vert = e^{-\kappa}  \vert (V f) \circ a_{-\kappa}\vert \leq 
    \frac{\lambda}{\nu +1} \vert (V f) \circ a_{-\kappa}\vert\,. 
  \end{equation}

  We then estimate the norm $\vert a_{-\kappa} D^\lambda\vert_{-r,-s;
    \tau}$ of the invariant distribution after geodesic scaling.
  \begin{lemma}
    We have
    \[
    \vert a_{-\kappa} D^\lambda\vert_{-r,-s; \tau} \leq
    \left(\frac{\nu+1}{\lambda}\right)^s \vert D^{\lambda}
    \vert_{-r,-s; \tau}\,.
    \]
  \end{lemma}
  \begin{proof}
    Notice that
    \[
    \begin{aligned}
      \vert a_{-\kappa} D^\lambda \vert_{-r,-s; \tau} & = \sup_{f\in
        \widehat H^\infty} \left\{|a_{-\kappa} D^\lambda(f)| : \vert f
        \vert_{r,s; \tau} = 1\right\}
      \\
      & = \sup_{f\in \widehat H^\infty} \left\{|D^\lambda(f \circ
        a_{\kappa})| : \vert f \vert_{r,s; \tau} = 1\right\}
      \\
      & = \sup_{f\in \widehat H^\infty} \left\{|D^\lambda(f)| : \vert
        f \circ a_{-\kappa} \vert_{r,s; \tau} = 1\right\}
      \\
      & \leq (\frac{\nu+1}{\lambda})^s \sup_{f\in \widehat H^\infty}
      \left\{|D^\lambda(f)| : \vert f \vert_{r,s; \tau} = 1
      \right\}\,,
    \end{aligned}
    \]
    where the last inequality holds since
    \[
    \vert V f \vert = \vert V(f \circ a_{-\kappa} \circ a_{\kappa} )
    \vert = e^{\kappa} \vert V (f \circ a_{-\kappa})\vert \leq
    \frac{\nu+1}{\lambda} \vert V (f\circ a_{-\kappa})\vert\,.
    \]
  \end{proof}

  Therefore,
  \[
  \eqref{equa:geodesic_D_lambda} \leq C^{(3)}_s
  \left(\frac{\tau}{\mathcal T}\right)^{-1/6}
  \left(\frac{\nu+1}{\lambda}\right)^s (1+ \vert\lambda\vert^{-2s})
  \vert D^{\lambda} \vert_{-s; \tau} \vert f \vert_{r+s, s;
    \mathcal{T}}\,.
  \]
  Proposition~\ref{prop:discrete_scale} now follows as in
  \eqref{equa:bound2}.
\end{proof}

This completes the proof of Theorem \ref{theo:Discrete_scale}.
\end{proof}

\section{Sobolev trace theorem}

In this section we prove a Sobolev trace theorem for horocycle
orbits. The main point of the result is that the constant in the
estimate is in terms of an ``average injectivity radius'' along the
orbit, with respect to rescaled Riemannian metrics.

\subsection{A priori bounds on ergodic integrals}

Let $\lambda \in \R^*$ and let $\mathcal T \geq 1$. Let
$(\phi^{\lambda, \mathcal T}_t)_{t\in \R}$ denote the rescaled twisted
horocycle flow on $M\times \T$, that is, the flow generated by the
rescaled vector field $\mathcal T (U +\lambda K)$.

Let $\bar x \in M\times \T$, let $I \subset \R$ be an interval, and
let $f \in C^\infty(M \times \T)$.  For $s > 1$, we will estimate in
terms of the Sobolev norm $\vert f\vert_{0, s; \mathcal T}$ the
ergodic integral
\[
|\int_I f \circ \phi_{t}^{\lambda, \mathcal T} (\bar x) dt | \,.
\]

We will follow the discussion of a Sobolev trace theorem in Section 3
of \cite{FF3} that concerns nilmanifolds in particular. We will see
that the method introduced there also gives the corresponding trace
theorem in the $\SL(2, \R)\times \T$ setting.

Let $\triangle_{\R^3}$ be the Euclidean Laplacian operator on $\R^3$
given by
\[
\triangle_{\R^3} := - \left(\frac{\partial^2}{\partial y^2} +
  \frac{\partial^2}{\partial z^2} + \frac{\partial^2}{\partial
    \theta^2} \right)\,.
\]
% A hyperrectangle is symmetric if it is centered at the origin.
Given an open set $O \subset \R^3$ containing the origin, we consider
the family $\mathcal R_O$ of all 3-dimensional hyperrectangles $R
\subset [-\frac{1}{2}, \frac{1}{2}]^3 \cap O$ that are centered at the
origin.  The \textit{inner width} of the open set $O\subset \R^3$ is
the positive number
\[
w(O) := \sup\{\text{Leb}(R) : R \in \mathcal R_O\}\,,
\]
where Leb$(R)$ is the Lebesgue measure of $R$.  The \textit{width
  function} of an open set $\Omega \subset \R \times \R^3$ containing
$\R \times \{0\}$ is the function defined for all $t\in \R$ by
\[
w_\Omega(t) := w( \{y \in \R^3 : (t, y) \in \Omega\})\,.
\]
  
We will now define average width, as in \cite{FF2}.  Let $\lambda >0$
and $\mathcal T\geq 1$.  Let $\bar x\in M\times \T$ and
$T>0$. Consider the family $\mathcal O_{\bar x, \lambda, \mathcal T,
  T}$ of open sets $\Omega \subset \R \times \R^3$ satisfying the
following two conditions:
\[ [0, T] \times \{0\} \subset \Omega \subset \R \times [-\frac{1}{2},
\frac{1}{2}]^3\,,
\]
and the map $\alpha_{\bar x, \lambda, \mathcal T} : \R^3 \to M \times
\T $ defined as
\begin{equation}
  \label{equa:alpha-def_1}
  \begin{aligned}
    \alpha_{\bar x, \lambda, \mathcal T} (t, \theta, y, z) = \bar x
    \exp(t \mathcal T(U + \lambda K) & \exp(y \mathcal T^{-1/3} X) \\
    & \times \exp (z \mathcal T^{-2/3} V)) \exp(\theta K)
  \end{aligned}
\end{equation}
is injective on the open set $\Omega\subset \R^3$.

The \textit{average width} of the orbit segment of the twisted
horocycle flow
\[
\{\alpha_{\bar x, \lambda, \mathcal T, T}(t, 0, 0, 0) : 0 \leq t \leq
T\}
\]
is the positive number
\[
w_\mathcal T(\bar x, \lambda, T) := \sup_{\Omega \in \mathcal O_{\bar
    x, \lambda, \mathcal T, T}} \left(\frac{1}{T}\int_0^T
  \frac{1}{w_\Omega(t)} dt\right)^{-1}\,.
\]

We will estimate from below the average with $w_\mathcal T(\bar x,
\lambda, T)$ of orbit segments of the twisted horocycle flow in the
Section \ref{sect:avg-width}.  In this section we derive a Sobolev
trace theorem for orbit segments and a Sobolev embedding-type theorem
with constants explicitly expressed in terms of the average width.

The following lemma is a special case of formula (32) of \cite{FF3}.
We prove it here for the convenience of the reader.
\begin{lemma}\label{lemma:trace_R^2}
  Let $I \subset \R$ be an interval, and let $\Omega \subset \R\times
  \R^3$ be a Borel set containing the segment $I \times \{0\} \subset
  \R \times \R^3$.  For every $s > 1$, there is a constant $C_{s} > 0$
  such that for all functions $F \in C^\infty (\Omega)$ and all $t\in
  I$, we have
  \[
  \left(\int_I |F(t, 0)| dt\right)^2 \leq C_{s} \left( \int_I
    \frac{1}{w_{\Omega}(t)} dt\right) \int_{\Omega} |(I -
  \triangle_{\R^3})^{s/2} F(t, y)|^2 dt dy\,.
  \]
\end{lemma}
\begin{proof}
  Define $\Omega_t := \{y\in \R^3 : (t, y) \in \Omega\}.$ By
  rescaling, Sobolev embedding gives
  \[
  |F(t, 0)| \leq C_s w_\Omega(t)^{-1/2} \left(\int_{\Omega_t} |(I -
    \triangle_{\R^3})^{s/2} F(t, y) |^2 dy\right)^{1/2}\,.
  \]
  Now we integrate over the interval $I$, and H$\ddot{\text{o}}$lder's
  inequality gives the result.
\end{proof}

For the remainder of this section, we follow Lemma 3.7 and Theorem 3.9
of \cite{FF3}.  Contrary to \cite{FF3}, the vector fields $ X_\mathcal
T$, $V_\mathcal T$ and $K$ that are transverse to the direction of
integration do not all commute, but they still form an integrable
distribution.

\begin{theorem}\label{theo:trace_thm}
  For all $s > 1$, there is a constant $C_s > 0$ such that the
  following holds. For $\lambda \in \R^*$, for all $\mathcal T \geq
  1$, for all $\bar x \in M \times \T$, all $T>0$ and for all $f \in
  C^\infty(M \times \T)$, we have
  \[
  |\frac{1}{T}\int_0^T f \circ \phi_{t }^{\lambda, \mathcal T} (\bar
  x) dt| \leq C_s T^{-1/2} w_\mathcal T(\bar x, \lambda,T)^{-1/2}
  \vert f \vert_{0, s; \mathcal T}\,.
  \]
\end{theorem}
\begin{proof}
  Let $\Omega \in \mathcal O_{\bar x, \lambda, \mathcal T,T}$.  We
  have
  \[
  \begin{aligned}
    \partial_\theta^2 f \circ \alpha_{\bar x, \lambda, \mathcal T}(t,y, z,\theta) & = K^2 f \circ \alpha_{\bar x, \lambda, \mathcal T}(t, y, z,\theta) \,, \\
    \partial_z^2 f \circ \alpha_{\bar x, \lambda, \mathcal T}(t, y,
    z,\theta) & = V_\mathcal T^2 f \circ \alpha_{\bar x, \lambda,
      \mathcal T}(t, y, z,\theta) \,,
  \end{aligned}
  \]
  and
  \begin{align*}
    \partial_y^2& f \circ \alpha_{\bar x, \lambda, \mathcal T}(t, y, z,\theta) = (X_\mathcal T - z \mathcal T^{-1/3} V_\mathcal T)^2 f \circ \alpha_{\bar x, \lambda, \mathcal T}(t, y, z,\theta) \\
    & = [ X_\mathcal T^2 + z^2 \mathcal T^{-2/3} V_\mathcal T^2 - z
    \mathcal T^{-1/3} (X_\mathcal T V_\mathcal T + V_\mathcal T
    X_\mathcal T) ] f \circ \alpha_{\bar x, \lambda, \mathcal T}(t, y,
    z,\theta)\,.
  \end{align*}

  Because $X$ and $V$ are essentially skew-adjoint,
  \[
  0 \leq -(X_\mathcal T + V_\mathcal T)^2 = -(X_\mathcal T^2 +
  V_\mathcal T^2) - (X_\mathcal T V_\mathcal T + V_\mathcal T
  X_\mathcal T)\,.
  \]
  Recall that $|z| \leq 1/2$, so
  \[
  - V_\mathcal T^2 - (X_\mathcal T - z \mathcal T^{-1/3} V_\mathcal
  T)^2 \leq 3 (- X_\mathcal T^2 - V_\mathcal T^2)\,,
  \]
  Because the operators on the left and right are essentially
  self-adjoint, the spectral theorem gives
  \[ [I - K^2 - V_\mathcal T^2 - (X_\mathcal T - z \mathcal T^{-1/3}
  V_\mathcal T)^2]^{s/2} \leq 3^{s/2} (I - K^2 - X_\mathcal T^2 -
  V_\mathcal T^2)^{s/2}\,,
  \]
  for any $s \geq 0$.  Next observe that
  \[
  \text{det}(D \alpha_{\bar x, \lambda, \mathcal T}(t, y, z,\theta)) =
  e^{-2 y \mathcal T^{-1/3}}\,.
  \]
  Then there is a constant $C_s > 0$ such that
  \begin{equation}\label{equa:spectral}
    \|(I - \triangle_{\R^3})^{s/2} f \circ \alpha_{\bar x, \lambda, \mathcal T}\|_{L^2(\Omega)}^2 
    \leq C_s \|(I - K^2 - X_\mathcal T^2 - V_\mathcal T^2)^{s/2} f\|_{L^2(M)}^2\,.
  \end{equation}

  By Lemma~\ref{lemma:trace_R^2} and formula~\eqref{equa:spectral}, we
  see that for any $s > 1$,
  \begin{align*}
    |\frac{1}{T}&\int_0^T f \circ \phi_{t}^{\lambda,\mathcal T}(\bar
    x) dt|^2
    \leq \left(\frac{1}{T} \int_0^T|f \circ \alpha_{\bar x, \lambda, \mathcal T}(t, 0, 0, 0)| dt\right)^2 \\
    & \leq C_{s} \frac{1}{T} \left(\frac{1}{T} \int_0^T \frac{1}{w_{\Omega}(t)} dt\right)  \int_{\Omega} |(I - \triangle_{\R^3})^{s/2} f\circ \alpha_{\bar x, \lambda, \mathcal T}(t, y, z, \theta)|^2 d\text{vol} \\
    & \leq C_{s} T^{-1} w_{\mathcal T}(\bar x, \lambda, T)^{-1} \|(I -
    K^2 - X_\mathcal T^2 - V_\mathcal T^2)^{s/2} f\|_{L^2(M)}^2\,.
  \end{align*}
  Because this holds for any set $\Omega \in \mathcal O_{\bar x,
    \lambda, \mathcal T, T}$, we may take the infimum over all sets
  $\Omega \in \mathcal O_{\bar x, \lambda, \mathcal T, T}$ and
  conclude the proof of Theorem~\ref{theo:trace_thm}.
\end{proof}

\subsection{Pointwise bounds for transfer
  functions}\label{subs:trace-transfer}

Following Lemma~3.7 and Theorem~3.9 of \cite{FF3}, we derive the
following bound on transfer functions of the twisted horocycle flow.
\begin{theorem}\label{theo:Sobolev-trace}
  Let $s > 1$.  Then there is a constant $C_s > 0$ such that for all
  $\lambda\in \R^*$ and $\mathcal T\geq 1$, for all $f \in C^\infty(M
  \times \T)$, if
  \[
  \mathcal T (U + \lambda K) g = f\,.
  \]
  then for all $\bar x \in M \times \T$, $T>0$ and for all $t \in [0,
  T]$,
  \[
  |g \circ \phi_t^{\lambda,\mathcal T}(\bar x)| \leq C_s T^{-1/2}
  w_\mathcal T(\bar x, \lambda, T)^{-1/2} \left( T \vert f \vert_{0,
      s; \mathcal T} + \vert g \vert_{0, s; \mathcal T}\right).
  \]
\end{theorem}
\begin{proof}
  Since $f \in H^\infty$, Theorems~\ref{theo:cohomology-principal}
  and~\ref{theo:cohomology-discrete} imply that $g \in H^\infty$. Let
  $t \in [0, T]$.  By the mean value theorem and by the fundamental
  theorem of calculus there exists $t_0 := t_0(\bar x, g) \in (0, T)$
  such that
  \begin{align}\label{equa:Sobolev-trace}
    |g \circ \phi_t^{\lambda,\mathcal T}(\bar x)| & = |\int_{t_0}^{t}
    \frac{d}{d\tau} g \circ \phi_{\tau}^{\lambda,\mathcal T}(\bar x)
    d\tau + \frac{1}{T}\int_0^T g\circ
    \phi_{\tau}^{\lambda,\mathcal T}(\bar x) d\tau| \notag \\
    & \leq |\int_{t_0}^{t} \mathcal T (U + \lambda K) g \circ
    \phi_{\tau}^{\lambda,\mathcal T}(\bar x) d\tau| +
    |\frac{1}{T}\int_0^T g \circ \phi_{\tau}^{\lambda,\mathcal T}(\bar x) d\tau| \notag \\
    & = |\int_{t_0}^{t} f \circ \phi_{\tau}^{\lambda,\mathcal T}(\bar
    x) d\tau| + | \frac{1}{T} \int_0^T g \circ
    \phi_{\tau}^{\lambda,\mathcal T}(\bar x) d\tau| \,. \notag
  \end{align}
  Now for $s > 1$, Theorem~\ref{theo:trace_thm} implies
  Theorem~\ref{theo:Sobolev-trace}.
\end{proof}

\section{Twisted horocycle flows:effective
  equidistribution}\label{sect:equi}

\subsection{Average width function}\label{sect:avg-width}
In this section we estimate the average width for horocycle segments,
which we define below.  Let $x\in M$, $\mathcal T\geq 1$ and
$T>0$. Consider the family $\mathcal O_{x, \mathcal T, T}$ of open
sets $\Omega \subset \R \times \R^2$ satisfying the following two
conditions:
\[ [0, T] \times \{0\} \subset \Omega \subset \R \times [-\frac{1}{2},
\frac{1}{2}]^2\,,
\]
and the map $\alpha_{x, \mathcal T} : \R^3 \to M \times \T $ defined
as
\begin{equation}
  \label{equa:alpha-def_2}
  \begin{aligned}
    \alpha_{\bar x, \mathcal T} (t, y, z) = x \exp(t \mathcal T U)
    \exp(y \mathcal T^{-1/3} X) \exp (z \mathcal T^{-2/3} V))
  \end{aligned}
\end{equation}
is injective on the open set $\Omega \subset \R^2$.

The \textit{average width} of the orbit segment of the horocycle flow
\[
\{\alpha_{x, \mathcal T, T}(t, 0, 0) : 0 \leq t \leq T\}
\]
is the positive number
\[
w_\mathcal T(x, T) := \sup_{\Omega \in \mathcal O_{x, \mathcal T, T}}
\left(\frac{1}{T}\int_0^T \frac{1}{w_\Omega(t)} dt\right)^{-1}\,.
\]

We remark that by our definitions, since the twisted horocycle flow
projects onto the horocycle flow under the projection of $M\times \T$
onto~$M$ and the vector field~$K$ tangent to the circle factor is not
scaled, the average width for an orbit segment of the twisted
horocycle flow is bounded below by the average for the projected orbit
segment of the horocycle flow.  In fact, the following holds:
\begin{lemma}
  \label{lemma:width_comparison}
  For all $\bar x= (x,\theta)$, for all $\lambda \in \R$, for all
  $\mathcal T \geq 1$ and all $T>0$, we have the inequality
  \begin{equation}
    w_\mathcal T(\bar x,  \lambda, T)     \geq    w_\mathcal T(x, T) \,.
  \end{equation}
\end{lemma}
\begin{proof}
  For any open set $\Omega\subset M$, let $\hat \Omega:= \Omega \times
  [-1/2, 1/2]$.  Since the vector field $K$ commutes with the vector
  fields $X$, $U$, $V$ and is tangent to the circle factor of the
  product $M\times \T$, it follows that for any $\bar x=(x,\theta)\in
  M\times \T$, for all $\mathcal T\geq 1$ and $T>0$, the function
  $\alpha_{\bar x, \lambda, \mathcal T, T}$, defined
  in~\eqref{equa:alpha-def_1}, is injective on $\hat \Omega$ if and
  only if the function $\alpha_{x, \mathcal T, T}$, defined
  in~\eqref{equa:alpha-def_2}, is injective on $\Omega$. The statement
  then follows from the definitions.
\end{proof}

It is therefore enough to estimate the average width $w_\mathcal T(x,
T)$ of the orbits segments of the horocycle flow on $M$.

\smallskip For any $x \in M$, we consider the map $\alpha_x$ defined
on $\R^3$ by
\[
\alpha_x (t, y, z) = x \exp(t U) \exp(y X) \exp(z V)\,.
\]
\begin{definition}
  \label{def:c(x,T)}
  For any $(x,T) \in M\times \R^+$, let $c_{\Gamma}(x, T)$ denote the
  smallest positive number $c \geq 1$ such that the map $\alpha_x$ is
  injective on the interval
  \[
  [-10 \ T, 10 \ T] \times [-1/2 , 1/2] \times \frac{1}{c} [-
  \frac{1}{T}, \frac{1}{T}]\,.
  \]
\end{definition}
\begin{remark}\label{rmrk:C-Gamma}
By the action of the geodesic flow, it follows from the definition that for all  $(x,T) \in M\times \R^+$ and
for all $y\in \R$ we have 
  \[
  c_{\Gamma}(x, T) = c_{\Gamma}(a_y(x), e^{-y }T)\,.
  \]
  The statement of the equidistribution theorems,
  Theorem~\ref{theo:equidistribution} and Theorem~\ref{theo:maps},
  involve the function $C_\Gamma(x,T)$ defined for all $(x,T)\in
  M\times \R^+$ by
  \begin{equation}
    \label{eq:C_Gamma1}
    C_\Gamma(x,T) := \sup_{ 1\leq t \leq T}  c_\Gamma(x,t) = 
    \sup_{-\log T \le y \le 0}c_\Gamma(a_y(x),T)\,.
  \end{equation}
\end{remark}

We begin by proving upper bounds on the function $c_{\Gamma}(x, T)$
under Diophantine conditions. For $A\in [0, 1)$ and $Q > 0$, recall
that we consider subsets of ``Diophantine points'' given by
\[
M_{A, Q} := \left\{x \in M : d_M(a_y(x)) \leq A y + Q \text{ for all }
  y> 0\right\}\,.
\]
In addition, by the logarithmic law of geodesics for almost all $x \in
M$ and for all $\epsilon>0$ there exists a constant
$Q_{\epsilon}(x)>0$ such that, for all $t\geq 1$,
\[
d_M(a_y(x)) \leq (\frac{1}{2} + \epsilon) \log y + Q_\epsilon(x) \,.
\]
For all $A>1/2$ and $Q$, $y_0>0$, we will therefore introduce the sets
\[
\widetilde M_{A, Q, y_0} := \left\{x \in M : d_M(a_y(x)) \leq A \log (
  y+y_0) + Q \text{ for all } y \geq 1-y_0 \right\}\,.
\]
\smallskip The proof of the following basic lemma is left for
Appendix~\ref{appe:C}.
\begin{lemma}
  \label{lemma:fund_box}
  For all $x\in M$, let $d_M(x) := \text{\rm dist}(x,\Gamma)$.  There
  exists a constant $C_\Gamma \in (0,1)$ such that for all $x \in M$,
  the map $\alpha_x : \R^3 \to M$ defined by the formula
  \[
  \alpha_x (t, y, z) = x \exp(t U) \exp(y X) \exp(z V) \,,
  \]
  is injective on the interval
  \[ [-C_\Gamma e^{-d_M(x)}, C_\Gamma e^{-d_M(x)}]
  \times [-1/2, 1/2] \times [-C_\Gamma e^{-d_M(x)}, C_\Gamma
  e^{-d_M(x)}] \,.
  \]
\end{lemma}

For all $x\in M$ and all $t>0$, let
\[
d_M(x,t) := \max_{0\leq y\leq t} d_M(a_y(x))\,.
\]
\begin{lemma}
  \label{lemma:c(x,T)}
  For all $x\in M$, for all $t>0$ and $T \in [1, 10^{-1}{C_\Gamma}
  e^{t-d_M(a_t(x))}] $ we have
  \[
  c_\Gamma (x,T) \leq \left( \frac{10}{C_\Gamma}\right)^2
  e^{2d_M(x,t)} \,;
  \]
  in particular for all $x\in M$ with bounded forward geodesic orbit,
  which is always the case whenever $M$ is compact, and for all $T>0$
  we have
  \[
  c_{\Gamma}(x, T) \leq \left( \frac{10}{C_\Gamma}\right)^2 \max_{x
    \in M} e^{2d_M(x)}\,.
  \]
  For all $x \in M_{A,Q}$ and for all $T\geq 1$ we have the estimate
  \[
  c_\Gamma(x, T) \leq (10 C^{-1}_\Gamma e^{Q} )^{\frac{2}{1-A}}
  T^{\frac{2A}{1-A}} \,.
  \]
  For every $A>1/2$ there exists a constant $C_{\Gamma,A}>0$ such that
  if $x \in \widetilde M_{A,Q, y_0}$ then for all $T\geq 1$ we have
  \[
  c_\Gamma(x, T) \leq C_{\Gamma,A} e^{2Q} (1+Q +y_0 + \log T)^{2A}\,.
  \]
\end{lemma}
\begin{proof}
  Let $C_\Gamma > 0$ be as in Lemma~\ref{lemma:fund_box}.  Let $x$,
  $t$ and $T$ as in the statement of the Lemma and set
  \[
  y_* = \min \{ y >0 \mid e^{-y} 10 T \leq C_\Gamma
  e^{-d_M(a_{y}(x))} \} \,,
  \]
  We remark that, since
  \[
  e^{-t} 10 T \leq C_\Gamma e^{-d_M(a_{t}(x))}\,,
  \]
  we have $y_*\leq t$, as the set in the definition of $y_*$ contains
  $t>0$. In fact have
  \begin{equation}
    \label{eq:y_Gamma_bound1}
    y_* \leq  \log  \left(\frac{10 T}{C_\Gamma}\right) + d_M(x,t)\,.
  \end{equation}
  By definition of the positive real number $y_*$ we have
  \[
  e^{-y_*} 10 T= C_\Gamma e^{-d_M(a_{y_*}(x))}\,.
  \]
  By Lemma~\ref{lemma:fund_box} the map $\alpha_{a_{y_*}(x)}$ is
  injective on
  \[ [-e^{-y_*} 10 T, e^{-y_*} 10 T]
  \times [-1/2, 1/2] \times [-e^{-y_*} 10 T , e^{-y_*} 10 T] \,.
  \]
  Recall that for any $t \in \R$, we have the commutation relations,
  \[
  \begin{aligned}
    a_{-y_*} \circ h_t \circ a_{y_*} & = h_{t e^{y_*}} \,; \\
    a_{-y_*} \circ \bar{h}_t \circ a_{y_*} & = \bar h_{t e^{-y_*}} \,.
  \end{aligned}
  \]
  Then by right multiplication of $a_{-y_*}$ on the image of
  $\alpha_{a_{y_*}(x)}$ restricted to the above set, we get that the
  map $\alpha_x$ is injective on the interval
  \[ [-10 T, 10 T] \times [-1/2, 1/2] \times \frac{T^2}{e^{2y_*}} [-
  \frac{1}{T} , \frac{1}{T} ] \,.
  \]
  By the estimate in formula~\eqref{eq:y_Gamma_bound1} it then follows
  that
  \[
  c_\Gamma(x,T) \leq  \frac{e^{2y_*}} {T^2} \leq
   \left(\frac{10}{C_\Gamma}\right)^2 e^{2d_M(x,t)}\,.
  \]
  If $x$ has bounded forward geodesic orbit we have in particular
  that, for all $T\geq 1$,
  \[
  c_\Gamma(x,T) \leq \left( \frac{10}{C_\Gamma}\right)^2 \max_{x\in M}
  e^{2d_M(x)} <+\infty\,.
  \]
  A straightforward estimate then shows that if $x\in M_{A,Q}$ then by
  the definitions we have
  \[
  c_\Gamma(x,T) \leq \frac{e^{2y_*}}{T^2} \leq (10 C^{-1}_\Gamma e^{Q}
  T)^{\frac{2}{1-A}} / T^2 = (10 C^{-1}_\Gamma e^{Q} )^{\frac{2}{1-A}}
  T^{\frac{2A}{1-A}}\,.
  \]
  Similarly if $x\in \widetilde M_{A,Q,y_0}$, that is, if $d_M(a_y(x)
  ) \leq A \log (y+y_0) + Q$ for all $y\geq 1-y_0$, it follows by the
  definition that $y_*$ is bounded above either by $A>0$ or by the
  unique solution $Y\geq A$ of the identity
  \[
  Y = A \log (Y+y_0) + Q + \log (\frac{10 T}{C_\Gamma}) \,.
  \]
  By change of variable $Y+y_0$ is equal to the unique solution $Z\geq
  A+y_0$ of the equation
  \[
  Z= A \log Z + Q + y_0+ \log (\frac{10 T}{C_\Gamma}) \,.
  \]
  By a straightforward calculation there exists a constant $C_A>0$
  such that
  \[
  Z \leq C_A + Q + y_0+ \log (\frac{10 T}{C_\Gamma}) + A \log \left( Q
    + y_0+ \log (\frac{10 T}{C_\Gamma}) \right)\,.
  \]
  We conclude that there exists a constant $C_{\Gamma,A}>0$ such that
  \[
  c_\Gamma(x,T) \leq \frac{e^{2y_*}}{T^2} \leq C_{\Gamma,A} e^{2Q} (1+
  Q + y_0+ \log T)^{2A}\,.
  \]
  The argument is concluded.
\end{proof}

For any $x \in M$ and $\mathcal T\geq 1$, we consider the scaled map
$\alpha_{x, \mathcal T}$ defined on $\R^3$ by
\[
\alpha_{x, \mathcal T} (t, y, z) = x \exp(t \mathcal T U) \exp(y
\mathcal T^{-1/3} X) \exp(z \mathcal T^{-2/3} V)\,.
\]
By the above definition and by change of variable we have the
following statement.
\begin{lemma}
  \label{lemma:V_limit}
  For all $\mathcal T\geq 1$ and all $(x,T)\in M\times\R^+$, the map
  $\alpha_{x, \mathcal T}$ is injective on the interval
  \[
  [-10 \ T, 10 \ T] \times [-\mathcal T^{1/3}/2 , \mathcal T^{1/3}/2]
  \times \frac{1}{ c_{\Gamma}(x,\mathcal T T)} [- \frac{\mathcal
    T^{-1/3}}{T}, \frac{\mathcal T^{-1/3}}{T}]\,.
  \]
\end{lemma}

\begin{theorem}\label{theo:average_width} 
There exists a constant $K_\Gamma>0$ such that the following holds.
  For any $x \in M$, for any $T \geq 1$ and $\mathcal T \in [1, T]$,
  there is an open tubular neighborhood $\Omega_{\mathcal T, T}(x)$ of
  $[0,T] \times \{(0,0)\}$ in  $[0,T] \times [-\frac{1}{2}, \frac{1}{2}]^2$ such that the map 
  $\alpha_{x, \mathcal T} : \Omega_{\mathcal T, T}(x) \to M$ is injective and
  the following estimate holds
  \[
  \frac{1}{T} \int_0^T \frac{1}{w_{\Omega_{\mathcal T, T}}(t)} dt \leq
K_\Gamma \cdot c_\Gamma^2(x,\mathcal T T) T (1+ \log (\mathcal
  T^{1/3}T)) \,.
  \]
\end{theorem}

From the above theorem and from Lemma~\ref{lemma:width_comparison}, we
derive our main result on the average width of orbits segments of the
twisted horocycle flow.
\begin{corollary}
  \label{cor:average_width}
  For any $\lambda\in \R^*$, for any $\bar x= (x,\theta)\in M\times
  \T$, for any $T \geq 1$ and for any $\mathcal T \in [1, T]$, we have
  the estimate
  \[
  w_\mathcal T(\bar x, \lambda, T)^{-1} \leq w_\mathcal T(x, T)^{-1}
  \leq K_\Gamma \cdot c_\Gamma^2(x,\mathcal T T) T (1+\log (\mathcal
  T^{1/3}T))\,.
  \]
\end{corollary}

We prove below Theorem~\ref{theo:average_width}. A simple calculation
shows

\begin{lemma}
  \label{lemma:commutation_rules}
  Let $t_0, s, z \in \R$ be such that $2 |z s| {\mathcal T}^{1/3} <
  1$.  Then
  \begin{equation}\label{equa:simplify}
    \begin{split}
      &\exp (t_0 \mathcal T U) \exp (y \mathcal T^{-1/3} X) \exp (z
      \mathcal T^{-2/3}V) \exp( s \mathcal T U) \\ &\qquad =
      \exp(t_0(s) \mathcal T U) \exp(y(s) \mathcal T^{-1/3} X)
      \exp(z(s) \mathcal T^{-2/3} V)\,,
    \end{split}
  \end{equation}
  where
  \begin{equation}\label{equa:separation1}
    \begin{array}{lll}
      t_0(s) = s e^{2 y \mathcal T^{-1/3}} (1 + z s \mathcal T^{1/3})^{-1} + t_0\\
      y(s) = y - \mathcal T^{1/3} \log (1 + \mathcal T^{1/3} z s) \\ 
      z(s) =  z  (1+ \mathcal T^{1/3} z s)^{-1}\,.    
    \end{array}
  \end{equation}
\end{lemma}
We now introduce certain closest returns of horocycle orbits.
\begin{definition}
  \label{defi:beta_close} Let $\mathcal T$, $T \geq 1$. A pair $(t_0, t_1) \in[-10 T, 10 T]^2$ is called \emph{a $(\beta,\mathcal T, T)$-return for $x\in M$} if  $\beta$
  is an integer in $[0, \log( \mathcal T^{1/3} T)]$ such that for some
  $|z| \in \frac{1}{c_\Gamma (x, \mathcal T T) } \, (e^{-(\beta +
    1)}, e^{-\beta}] $ we have
  \[
  x \exp(t_1 \mathcal T U) = x \exp(t_0 \mathcal T U) \exp(z \mathcal
  T^{-2/3} V)\,.
  \]
  The  $(\beta,\mathcal T, T)$-return is called \emph{non-degenerate}
  if $t_0 \neq t_1$ and  \emph{degenerate} if $t_0=t_1$.
  
   We denote by $n^\beta_{ \mathcal T, T}(x)$ the number of non-degenerate $(\beta,\mathcal T, T)$-returns and by $n^{\beta,deg}_{ \mathcal T, T}(x)$ 
  the number of degenerate $(\beta,\mathcal T, T)$-returns for $x\in M$.
\end{definition}

Let $\mathbb H :=  \{w\in \C\vert \im(w)>0\}$ denote the upper half-plane.
Let $\{C_i\}$  be the collection of disjoint cusps of the surface $S:=\Gamma \backslash \mathbb H$ 
bounded by a cuspidal horocycles of length $\ell_\Gamma<1$. By a cusp of $M$
we mean the tangent unit bundle $\tilde C_i \subset  M$ of a cusp $C_i$, that is, the pull-back
to $M$ of the cusp $C_i \subset S$.
The manifold $M$ can be decomposed as a disjoint union of a \emph{thin part}, defined as 
the union of the finite collection of the disjoint cusps $\tilde C_i$,  and of a compact 
\emph{thick part}.

\smallskip
The following Lemma is proven in Appendix~\ref{appe:C}.  

\begin{lemma} 
\label{lemma:returns_loc}
For any non-degenerate $(\beta,\mathcal T, T)$-return
for $x\in M$ the return points $x \exp(t_1 \mathcal T U)$ and $x \exp(t_2 \mathcal T U)$
belong to the thick part of $M$. For any degenerate $(\beta,\mathcal T, T)$-return
for $x\in M$ the point $x \exp(t_1 \mathcal T U) = x \exp(t_0 \mathcal T U)$ belongs
to the cuspidal horocycle for the unstable horocycle $\{\bar h_t\}$ generated by the
vector field $V$ on $M$.
\end{lemma}

Next lemma shows that  non-degenerate $(\beta,\mathcal T, T)$-returns cannot be too close.

\begin{lemma} \label{lemma:return_separated} There exists a constant $C_\Gamma>1$ 
(depending only on the thick part of $M$) such that the following holds. Let $(t_0, t_1) \in [-10
  \ T, 10 \ T]^2 $ be a non-degenerate $(\beta,\mathcal T, T)$-return for $x \in M$.
  If $(t_0', t_1')\in [-10 \ T, 10 \ T]^2 $ is another non-degenerate $(\beta,\mathcal T, T)$-return for 
  $x$ and
  \[
  |t_0' - t_0| < \frac{1}{2C_\Gamma} e^{\beta} \mathcal T^{-1/3} , \ |t_1' -
  t_1| < \frac{1}{2C_\Gamma} e^\beta \mathcal T^{-1/3}\,,
  \]
  then $t_0' = t_0$ and $t_1' = t_1$.
\end{lemma}
\begin{proof}
  If $(t_0, t_1)$ and $(t_0', t_1')$ are $(\beta,\mathcal T, T)$-return
  pairs in $[-10 \ T, 10 \ T]^2 $ , by definition there exist $z$ and
  $z'$ with $\vert z\vert$, $\vert z'\vert \in \frac{1}{c_\Gamma
    (x, \mathcal T T)} (e^{-(\beta+1)}, e^{-\beta}]$ such that
  \begin{equation}\label{eq:beta_returns}
    \begin{aligned}
      x \exp(t_1 \mathcal T U)
      &= x \exp(t_0 \mathcal T U) \exp(z \mathcal T^{-2/3} V) \,, \\
      x \exp(t'_1 \mathcal T U) &= x \exp(t'_0 \mathcal T U) \exp(z'
      \mathcal T^{-2/3} V)\,.
    \end{aligned}
  \end{equation}
 By geodesic scaling for geodesic time $\sigma:=\beta+ \frac{2}{3} \log \mathcal T$ of the identities
in formula \eqref{eq:beta_returns} we derive
$$
 \begin{aligned}
      a_\sigma (x) \exp(e^{-\sigma} t_1 \mathcal T U)
      &= a_\sigma (x)  \exp(e^{-\sigma} t_0 \mathcal T U) \exp(e^\sigma z \mathcal T^{-2/3} V) \,, \\
      a_\sigma(x) \exp(e^{-\sigma}  t'_1 \mathcal T U) &= a_\sigma (x)  \exp(e^{-\sigma} t'_0 \mathcal T U) 
      \exp(e^\sigma z' \mathcal T^{-2/3} V)\,.
    \end{aligned}
 $$
The pairs $(e^{-\sigma} t_0 \mathcal T, e^{-\sigma} t'_0 \mathcal T)$ and
$(e^{-\sigma} t_1 \mathcal T, e^{-\sigma} t'_1\mathcal T)$ are  non-degenerate 
$(0, 1, \mathcal T T)$-returns for the point $a_\sigma(x)$ with 
\begin{equation}
\label{eq:rescaled_dist}
  e^{-\sigma}  |t_0' - t_0| \mathcal T  < \frac{1}{2C_\Gamma} , \  e^{-\sigma} |t_1' -
  t_1| \mathcal T< \frac{1}{2C_\Gamma} \,.
\end{equation}
By Lemma~\ref{lemma:returns_loc}  the points
$$
\begin{aligned}
&a_\sigma (x) \exp(e^{-\sigma} t_0 \mathcal T U)\,,\quad a_\sigma (x) \exp(e^{-\sigma} t'_0 \mathcal T U)
\\ &a_\sigma (x) \exp(e^{-\sigma} t_1 \mathcal T U)\,, \quad a_\sigma (x) \exp(e^{-\sigma} t'_1 \mathcal T U)
\end{aligned}
$$
 all belong to the thick part of $M$.  Let $s = t'_1 -t_1$. Since 
 $|s| \le \frac{e^{\beta}}{2C_\Gamma} \mathcal T^{-1/3}$, we have
  \begin{equation}
    \label{eq:inj_bound}
   \mathcal T^{1/3} \vert zs \vert  \,   \leq \frac {e^{\beta}}{2C_\Gamma}
    \frac{e^{-\beta}}{c_\Gamma (x, \mathcal T T)}  <\frac{1}{2C_\Gamma}   \,.
  \end{equation}
  Hence, by Lemma \ref{lemma:commutation_rules} we obtain
  \[
  \begin{aligned}
    a_\sigma(x) &\exp(e^{-\sigma} t'_1 \mathcal T U)  
    = a_\sigma(x) \exp( e^{-\sigma} t_1 \mathcal T U) \exp(  e^{-\sigma} s \mathcal T U) \\
    &= a_\sigma(x) \exp(e^{-\sigma} t_0 \mathcal T U) \exp(e^\sigma z \mathcal T^{-2/3} V) 
    \exp(e^{-\sigma} s \mathcal T U) \\ 
    &= a_\sigma(x) \exp(e^{-\sigma} t_0(s) \mathcal T U) \exp(y(s) \mathcal
    T^{-1/3} X) \exp(e^{\sigma} z(s) \mathcal T^{-2/3} V) \,,
  \end{aligned}
  \]
  where $t_0(s)$, $y(s)$ and $z(s)$ are given by formulas
  \eqref{equa:separation1} with $y=0$, that is,
  \begin{equation}
    \label{eq:separation2}
    \begin{split}
      t_0(s) &= s (1 + \mathcal T^{1/3} z s )^{-1} + t_0\\
      y(s) &= - \mathcal T^{1/3} \log (1 + \mathcal T^{1/3} z s) \\
      z(s) &= z (1+ \mathcal T^{1/3} z s)^{-1}\,.
    \end{split}
  \end{equation}
  It follows that
  \begin{equation}
    \label{eq:overlap1}
    \begin{aligned}
      a_\sigma(x) & \exp(e^{-\sigma} t'_0 \mathcal T U) \exp(e^\sigma z' \mathcal T^{-2/3} V) \\ =
      &a_\sigma(x) \exp(e^{-\sigma} t_0(s) \mathcal T U) \exp(y(s) \mathcal T^{-1/3} X)
      \exp(e^\sigma z(s) \mathcal T^{-2/3} V) \,.
    \end{aligned}
  \end{equation}
  By the estimate \eqref{eq:inj_bound} the above expression for
  $t_0(s)$, $y(s)$ and $z(s)$,
  \begin{equation}
  \label{eq:coord_bounds}
  \begin{aligned}
    t_0(s) &\in [t_0 - 2s, t_0 + 2s]  \,, \\
    y(s) &\in [-\mathcal T^{1/3}/2,  \mathcal T^{1/3}/2]  \,, \\
    \vert z(s) \vert & \in [z/2, 2z]  \,.
  \end{aligned}
  \end{equation}
Let now $x_\sigma\in M$ denote the intermediate point
  $$
  x_\sigma:=  a_\sigma(x) \exp(e^{-\sigma} \frac{t'_0+t_0(s)}{2} \mathcal T U)\,.
  $$
  By the above bounds, since the points 
  $$
  a_\sigma(x) \exp(e^{-\sigma} t_0 \mathcal T U)\quad
  \text{ and } \quad a_\sigma(x) \exp(e^{-\sigma} t'_0 \mathcal T U)
  $$ 
  belong to the thick part of $M$, it follows that $x_\sigma$ belongs to the compact set of points at distance at most $1/C_\Gamma$ from the thick part. 
  
    The identity in formula~\eqref{eq:overlap1} can be rewritten as
  \begin{equation}
    \label{eq:overlap2}
    \begin{aligned}
      x_\sigma & \exp(e^{-\sigma} \frac{t'_0-t_0(s)}{2} \mathcal T U) 
      \exp(e^\sigma z' \mathcal T^{-2/3} V) \\ =
      &x_\sigma \exp(e^{-\sigma} \frac{t_0(s)-t_0}{2} \mathcal T U) \exp(y(s) \mathcal T^{-1/3} X)
      \exp(e^\sigma z(s) \mathcal T^{-2/3} V) \,.
    \end{aligned}
  \end{equation}
  Since $x_\sigma$ is at distance at most $1/C_\Gamma$ from the thick part and by 
  the bounds in formulas~\eqref{eq:rescaled_dist} and~\eqref{eq:coord_bounds}, it follows
  that there exists a constant $C_\Gamma>1$ for which  the identity~\eqref{eq:overlap2} 
  implies that
 \[
  t_0(s) -t'_0 = y(s) = z(s)-z' =0\,.
  \]
  Formula \eqref{eq:separation2} shows that the $y(s)=0$ implies
  that $sz=0$ and consequently $s=t_1-t'_1=0$ since $z\not=0$ and 
  $t_0(s)=t_0$ and $z(s) =z$.  In turn
  this implies $t'_0=t_0$ and $z =z'$. The argument is now
  concluded.
\end{proof}

As a consequence, we derive the following bound for the number $n_{ \mathcal T, T}^\beta(x)$  of
  non-degenerate $(\beta,\mathcal T, T)$-returns.

\begin{proposition}\label{prop:count}
  We have
  \[
  n_{\mathcal T,T}^\beta(x) \leq 4 \cdot 10^2 C^2_\Gamma e^{-2\beta} \mathcal
  T^{2/3} T^2 \,.
  \]
\end{proposition}
\begin{proof}
  By Lemma \ref{lemma:return_separated}, the maximum number of
  disjoint squares of side length $\frac{1}{2C_\Gamma} e^{\beta} \mathcal
  T^{-1/3}$ that fit inside the square $[-10 T, 10 T]^2$ is bounded by
  \[
  \frac{(10 T)^2}{(e^{\beta} \mathcal T^{-1/3}/2C_\Gamma )^2} \leq 4\cdot 10^2 C^2_\Gamma e^{-2\beta} \mathcal T^{2/3} T^2 \,.
  \]
\end{proof}

The number $n_{ \mathcal T, T}^{\beta,deg}(x)$  of degenerate $(\beta,\mathcal T, T)$-returns
is estimated as follows.

\begin{proposition}\label{prop:count_deg}  There exists a constant $C'_\Gamma>0$ such that
$$
n_{\mathcal T,T}^{\beta,deg} (x) \leq  C'_\Gamma (1 +  e^{-\beta} \mathcal T^{1/3} T).
$$
\end{proposition}
\begin{proof} There exists a constant $c_\Gamma>0$ such that the following holds.
By definition and by Lemma~\ref{lemma:returns_loc},  for any degenerate 
$(\beta, \mathcal T, T)$-return $t_0=t_1$ for $x\in M$, the point $x \exp (t_0 \mathcal T U)$ must be a point of the unstable cuspidal horocycle at distance from the thick part of $M$ larger than 
$$
c_\Gamma (\beta + \frac{2}{3} \log \mathcal T)  \,.
$$
It is therefore enough to bound the number of points of intersections of a stable horocycle
arc of length $T>0$  with the unstable cuspidal horocycles at a distance larger then $d>0$
from the thick part.  We claim is that there there are at most $ 1 + e T/ e^d$. The statement
will then follow immediately from the claim since we are counting the degenerate close returns
of a horocycle arc of rescaled length $T>0$, hence of hyperbolic length $\mathcal T T>0$.

By applying the geodesic flow for a time $t=d -1$ we get a horocycle arc of hyperbolic length $e T/e^d$. 
The points of intersection of the stable horocycle of hyperbolic length $T>0$ with the unstable cuspidal horocycle which are at distance larger than $d$ from the thick part are sent by the geodesic time map to points of intersection of the shortened stable horocycle of length $e T/e^d$ with the unstable cuspidal horocycle which are at distance larger than $1$ from the thick part. Since between any two such 
intersections the stable horocycle has to enter the thick part, their total number is at most 
$1 + e T/e^d$.  The argument is therefore completed.
\end{proof}

We now construct tubular neighborhoods $\Omega_{\mathcal T, T}(x)$ of $[0,T] \times 
\{(0,0)\}$ in  $[0,T] \times [-1/2, 1/2]^2$ with the properties claimed in Theorem~\ref{theo:average_width}. 

Suppose that $t_0\in [-10 T, 10 T]$ belongs to a $(\beta, \mathcal T,T)$-return pair.  Let
\[
\Omega_{\mathcal T, T}^{\beta, t_0} \subset \left([-10 T, 10 T] \cap [t_0 -
  e^{\beta}\mathcal T^{-2/3}, t_0 + e^{\beta} \mathcal
  T^{-2/3}]\right) \times [-\frac{1}{2}, \frac{1}{2}]^2
\]
be a partial tubular neighborhood of $[-10 T, 10 T] \times \{(0,0)\}$ defined 
locally about $t_0$ whose cross-section at a point $t \in [-10 T, 10 T] \cap [t_0 - e^{\beta}\mathcal
T^{-2/3} , t_0 + e^{\beta}\mathcal T^{-2/3}]$ is a square centered at $(0,0)$
contained in $[-\frac{1}{2}, \frac{1}{2}]^2$ with side-length equal to
\begin{equation}\label{equa:Omega-dimensions}
  \frac{1}{100} \frac{e^{-\beta}}{c_\Gamma(x,\mathcal T T)} \max\{\mathcal T^{2/3} |t - t_0|, 1\}\,.
\end{equation}
Extend then the tube $\Omega_{\mathcal T, T}^{\beta, t_0}$ to a full tubular neighborhood $\Omega_{\mathcal T, T}^{\beta, t_0, E}$ of $[-10 T, 10 T] \times \{(0,0)\}$ whose cross-sections are squares centered at $(0,0)$ with side-length equal to $\frac{1}{100 \ c_\Gamma(x,\mathcal T  T)}$, for all $t\in [-10 T, 10 T] \backslash [t_0 - e^{\beta} \mathcal T^{-2/3}, t_0 + e^{\beta} \mathcal T^{-2/3}]$. 
Let then
\[
\mathcal A_{\mathcal T, T}(x) := \left\{(\beta, t_0) \in
  \N \times [-10 T, 10 T] : t_0 \text{ belongs to a
  }\text{$(\beta,\mathcal T,T)$-return pair}\right\} \,,
\]
and define $\Omega_{\mathcal T, T}(x) \subset [0,T] \times [-1/2,1/2]^2$ by
\[
\Omega_{\mathcal T, T}(x) := \bigcap_{(\beta, t_0) \in \mathcal A_{
    \mathcal T,T}(x)}\ \Omega_{\mathcal T, T}^{\beta, t_0, E}\quad \bigcap \quad [0,T] \times [-1/2, 1/2]^2\,.
\]
\begin{lemma}\label{lemma:no_intersection}
  The map $\alpha_{x, \mathcal T}: \Omega_{\mathcal T, T}(x) \to M$ is
  injective.
\end{lemma}
\begin{proof}
  Suppose that $(t_0, y_0, z_0), (t_1, y_1, z_1) \in \Omega_{\mathcal T,T}(x)$ are such that
  \[
  \alpha_{x, \mathcal T}(t_0, y_0, z_0) = \alpha_{x, \mathcal T}(t_1,
  y_1, z_1)\,.
  \]
  From this condition we will derive that $(t_0, y_0, z_0) = (t_1, y_1, z_1)$.

  Let $y_2 := y_0 - y_1$ and $z_2 := (z_0 - z_1) e^{y_1 \mathcal
    T^{-1/3}}$. A calculation shows
  \begin{equation}\label{equa:to-beta_close}
    x \exp(t_1 \mathcal T U) = x \exp(t_0 \mathcal T U) \exp(y_2 \mathcal T^{-1/3} X) 
    \exp(z_2\mathcal T^{-2/3} V).
  \end{equation}
 From the definition of the tube $\Omega_{\mathcal T, T}(x)$,   we have,  for $i=0, 1$,
 $$
 \vert y_i \vert, \quad  \vert z_i\vert  \leq \frac{1}{100} \frac{1}{c_\Gamma(x, \mathcal T T)}\,.
 $$
which by the above formulas for $y_2$ and $z_2$ implies that  
  \begin{equation}\label{equa:y_2, z_2}
    |y_2| \leq \frac{1}{50}\frac{1}{ c_\Gamma(x,\mathcal T T)}\,, \qquad
    |z_2| \leq \frac{1}{25} \frac{1}{ c_\Gamma(x,\mathcal T T) } \,.
  \end{equation}  
  We therefore assume that $(t_0, y_0, z_0) \neq  (t_1, y_1, z_1)$ and derive a contradiction.

  Since by assumption $t_0, t_1 \in [0, T]$, from
  Lemma~\ref{lemma:V_limit} and  formula~\eqref{equa:to-beta_close} it follows that $(t_0, y_0, z_0)=(t_1, y_1,z_1)$ whenever $|z_2| <
  \frac{1}{c_\Gamma(x,\mathcal T T)} \frac{\mathcal T^{-1/3}}{T}$.
  Since by our assumption $(t_0, y_0, z_0) \neq (t_1,y_1, z_1)$, it follows that
  \[
  |z_2| \in \frac{1}{c_\Gamma(x,\mathcal T T)} [\frac{\mathcal
    T^{-1/3}}{T}, 1]\,.
  \]
  Then there is a number $\beta \in \mathbb{Z} \cap [0, \log(\mathcal
  T^{1/3}T)]$ such that
  \begin{equation}
    \label{eq:z_2range}
    |z_2| \in \frac{1}{c_\Gamma(x,\mathcal T T)} (e^{-(\beta + 1)}, e^{-\beta}]\,.
  \end{equation}
  By formulas~\eqref{equa:simplify} and~\eqref{equa:to-beta_close}, 
  on the interval $I_{\mathcal T, \beta} := \{s : \mathcal T^{1/3} \vert z_2 s \vert < 7/8\}$ we can write 
  \[
  x \exp((t_1 + s) \mathcal T U) = x \exp(t_0(s) \mathcal T U)
  \exp(y_2(s) \mathcal T^{-1/3} X) \exp(z_2(s) \mathcal T^{-2/3} V),
  \]
  where $t_0(s)$, $y_2(s)$ and $z_2(s)$ are given by
  formula~\eqref{equa:separation1}, that is,
  \begin{equation}
   \label{equa:separation1_again}
  \begin{array}{lll}
    t_0(s) = s e^{2 y_2 \mathcal T^{-1/3}}  (1 +  \mathcal T^{1/3} z_2 s)^{-1} + t_0\\
    y_2(s) = y_2- \mathcal T^{1/3} \log (1 + \mathcal T^{1/3} z_2 s) \\ 
    z_2(s) =  z_2  (1+ \mathcal T^{1/3} z_2 s)^{-1}\,.    
  \end{array}
  \end{equation}
  A calculation based on Taylor formula or the intermediate value theorem shows that there is a smooth function $h_{z_2}$ on $I_{\mathcal T, \beta}$ such that 
  \begin{equation}\label{equa:estimate_y(s)4}
    y_2(s) = y_2 - \mathcal T^{2/3} z_2 s + h_{z_2}(s)\,. 
  \end{equation}
  and
  \begin{equation}\label{eq:h_est}
    |h_{z_2}(s)| \leq 4 \mathcal T  |z_2 s|^2  \,.
  \end{equation}
  So if $y_2 < 0$, then for $s$ satisfying $\vert s\vert =
  \frac{1}{20}\mathcal T^{-2/3} e^{\beta + 1}$ with $z_2 s < 0$, since
  $\mathcal T\geq 1$, we have
  \[
  y_2 - \mathcal T^{2/3} z_2 s + h_{z_2}(s) \geq
  \frac{1}{c_\Gamma(x,\mathcal T T)}\left(-\frac{1}{50} + \frac{1}{20}
    - \frac{1}{50} \right) > 0\,.
  \]
  Because $y_2(0) = y_2 < 0$ and $h_{z_2}$ is continuous, it follows
  that there is
  \begin{equation}
    \label{eq:sstar}
    s^* \in [-\mathcal T^{-2/3} e^{\beta + 1}/20, \mathcal T^{-2/3} e^{\beta + 1}/20] \subseteq [-10 \ T, 10 \ T]
  \end{equation}
  such that
  \[
  y_2(s^*) = 0\,.
  \]
  A similar argument holds when $y_2 > 0$, with $\vert s\vert =
  \frac{1}{20}\mathcal T^{-2/3} e^{\beta + 1}$ and $z_2 s > 0$.

 The above formula for $z_2(s)$ gives
  \begin{equation}\label{equa:z_3 beta-close}
    \frac{1}{c_\Gamma(x,\mathcal T T)} e^{-(\beta + 2)} < |z_2(s^*)| \leq 
    \frac{1}{c_\Gamma(x,\mathcal T T)} e^{-\beta + 1}\,. 
  \end{equation}
  In other words, there is some $\delta \in \{\beta - 1, \beta, \beta
  + 1\}$ such that $(t_0(s^*), t_1 + s^*)$ is a $\delta$-close pair.
  Now we will use the definitions of the open sets $\Omega_{
    \mathcal T,T}^{\delta, t_0(s^*)}$ and $\Omega_{\mathcal
    T,T}^{\delta, t_1 +s^*}$ to derive the contradiction that
  \begin{equation}\label{equa:points-notequal}
    (t_0, y_0, z_0) \not \in \Omega_{\mathcal T, T}^{\delta, t_0(s^*)} \quad \text{or} \quad
    (t_1, y_1, z_1) \not \in \Omega_{\mathcal T, T}^{\delta, t_1 +s^*}\,.
  \end{equation}
  From formula~\eqref{equa:Omega-dimensions},  for all $s\in \R$ such that 
  $\vert t_0(s) -t_0(s^*)\vert \leq e^\delta \mathcal
  T^{-2/3}$, let
  \[
  E^\delta_{t_0(s^*)}(s) := \frac{1}{100}
  \frac{e^{-\delta}}{c_\Gamma(x,\mathcal T T)} \max\{\mathcal T^{2/3}
  \vert t_0(s) - t_0(s^*)\vert, 1\}
  \]
  be the edge length of the cross-section at $(t_0(s), 0, 0)$ of the
  tube $\Omega_{\mathcal T, T}^{\delta, t_0(s^*)}$, and  for all $s\in \R$
  such that $\vert s-s^*\vert \leq e^\delta \mathcal T^{-2/3}$, let
  \[
  E^\delta_{t_1 + s^*}(s) := \frac{1}{100}
  \frac{e^{-\delta}}{c_\Gamma(x,\mathcal T T)} \max\{\mathcal T^{2/3}
  \vert s-s^* \vert, 1\}
  \]
  be the edge length of the cross-section at $(t_1 + s, 0, 0)$ of the
  tube $\Omega_{\mathcal T, T}^{\delta, t_1 + s^*}$.
  
  By formulas~\eqref{equa:separation1_again} and~\eqref{eq:sstar} for all $s \in [0, s^*]$ and for
  $\delta \in \{\beta - 1, \beta, \beta + 1\}$ we have
  \begin{equation}
    \label{eq:t_bound}
    \vert t_0 - t_0(s^*) \vert \leq  2 \vert s-  s^*\vert  \leq 2 \vert s^*\vert \leq  e^\delta \mathcal T^{-2/3} \,.
  \end{equation}
  In particular, the edge lengths at the points $(t_0, y_0, z_0)$ and
  $ (t_1, y_1, z_1)$ are respectively $E^{t_0(s^*)}(0)$ and $E^{t_1 +
    s^*}(0)$.

  By the assumption that $(t_0, y_0, z_0), (t_1, y_1, z_1) \in
  \Omega_{\mathcal T, T}(x)$ and by formula~\eqref{equa:to-beta_close}
  we deduce the inequality
  \begin{equation}\label{equa:edge-z_3-y_3}
    \max\{|z_2|, |y_2|\} \leq  E^\delta_{t_0(s^*)}(0)+ E^\delta_{t_1 + s^*}(0) \,.
  \end{equation}
  By the above expression and by formula~\eqref{eq:t_bound} we derive
  the bound
  \begin{equation}
    \label{eq:yz_upper}
    \max\{|z_2|, |y_2|\} \leq  \frac{1}{50} \frac{e^{-\delta}}{ c_\Gamma(x,\mathcal T T)} 
    \max\{2 \mathcal T^{2/3} \vert s^*\vert, 1\}\,.
  \end{equation}
  However, on the one hand if $\vert s^* \vert \leq \mathcal T^{-2/3}
  /2$, we derive the inequality
  \[
  \frac{1}{c_\Gamma(x,\mathcal T T)} e^{-(\beta + 1)} < \vert z_2\vert
  \leq \frac{1}{50} \frac{e^{-\delta}}{c_\Gamma(x,\mathcal T T)}\,,
  \]
  which cannot hold as $\delta \leq \beta +1$, on the other hand if
  $\vert s^* \vert \geq \mathcal T^{-2/3} /2$ by
  formula~\eqref{equa:estimate_y(s)4} since $y_2(s^*)=0$ we derive the
  inequality
  \[
  \vert y_2 \vert = \vert \mathcal T^{2/3} z_2 s^* + h_{z_2}(s^*)
  \vert < \frac{1}{25} \frac{e^{-\delta}}{c_\Gamma(x,\mathcal T T)}
  \mathcal T^{2/3} \vert s^*\vert\,,
  \]
  which cannot hold as, by formulas~\eqref{eq:z_2range}
  and~\eqref{eq:h_est}, we have
  \[
  \vert \mathcal T^{2/3} z_2 s^* + h_{z_2}(s^*) \vert \geq \frac{1}{2}
  \vert z_2\vert \mathcal T^{2/3} \vert s^*\vert \geq \frac{1}{2}
  \frac{e^{-(\beta+1)}}{c_\Gamma(x,\mathcal T T)} \mathcal T^{2/3}
  \vert s^*\vert\,.
  \]
  Thus if $\alpha_{x, \mathcal T}(t_0, y_0, z_0) = \alpha_{x, \mathcal
    T}(t_1, y_1, z_1)$ for $(t_0, y_0, z_0), (t_1, y_1, z_1) \in
  \Omega_{\mathcal T, T}(x)$, under the assumption that $(t_0, y_0, z_0) \not
  =(t_1, y_1, z_1)$ we have reached a contradiction in all cases. It follows that 
  $(t_0, y_0, z_0)= (t_1, y_1,z_1)$. Thus the map $\alpha_{x, \mathcal T}$ is injective 
  on $\Omega_{\mathcal T, T}(x)$.  This concludes the proof of
  Lemma~\ref{lemma:no_intersection}\,.
\end{proof}

Now we estimate the contribution to the average width function from
each "pinched" regions $\Omega_{\mathcal T, T}^{\beta, t_0}$ of the
tube $\Omega_{\mathcal T, T}(x)$.  Let $w_{\Omega_{\mathcal
    T,T}^{\beta, t_0}}$ be the width function for $\Omega_{\mathcal
  T,T}^{\beta, t_0}$.
\begin{lemma}\label{lemma:beta_contribution}
  We have
  \[
  \int_{t_0- e^\beta \mathcal T^{-2/3}}^{t_0+ e^\beta \mathcal
    T^{-2/3}} \frac{1}{w_{\Omega_{\mathcal T, T}^{\beta, t_0}(t)}} dt
  \leq 4\cdot 10^4 c_\Gamma^2(x,\mathcal T T)\mathcal T^{-2/3} e^{2
    \beta} \,.
  \]
\end{lemma}
\begin{proof}
  A computation shows that, for all $\beta \geq 0$, $\mathcal T \geq
  1$ we have the estimate
  \[
  I_{\beta, \mathcal T} := \int_{0}^{\mathcal T^{-2/3}} \frac{1}{e^{-2
      \beta}} dt + \int_{\mathcal T^{-2/3}}^{+\infty}
  \frac{1}{e^{-2\beta} (\mathcal T^{2/3} t)^2} dt \, \leq \, 2
  \mathcal T^{-2/3} e^{2\beta} \,.
  \]
  By formula~\eqref{equa:Omega-dimensions}, we observe that
  \[
  \begin{aligned}
    \int_{t_0- e^\beta \mathcal T^{-2/3}}^{t_0+ e^\beta \mathcal
      T^{-2/3}} \frac{1}{w_{\Omega_{\mathcal T, T}^{\beta, t_0}(t)}}
    dt &\leq 2\cdot 10^4 c_\Gamma^2(x,\mathcal T T) I_{\beta, \mathcal
      T} \\ &\leq 4\cdot 10^4 c_\Gamma^2(x,\mathcal T T) \mathcal
    T^{-2/3} e^{2\beta}\,.
  \end{aligned}
  \]
\end{proof}

\begin{proof}[Proof of Theorem~\ref{theo:average_width}]
  By Lemma~\ref{lemma:beta_contribution}, we get
  \begin{align*}
    & \frac{1}{T} \int_0^T  \frac{1}{w_{\Omega_{\mathcal T, T}}(t)} dt  \leq  c_\Gamma^2(x,\mathcal T T)     \\
    & +\frac{1}{T} \sum_{\beta = 0}^{[ \log (\mathcal T^{1/3}T)] + 1}  4\cdot 10^4 n_{\mathcal T,T}^\beta(x)  \mathcal T^{-2/3} c_\Gamma^2(x,\mathcal T T)  e^{2\beta} \\
    & + \frac{1}{T} \sum_{\beta = 0}^{[ \log (\mathcal T^{1/3}T)] + 1}  4\cdot 10^4 n_{\mathcal T,T}^{\beta, deg}(x)  \mathcal T^{-2/3} c_\Gamma^2(x,\mathcal T T)  e^{2\beta} \,.
 \end{align*}
 Now by Proposition ~\ref{prop:count}  we have 
  \begin{align*}
 \sum_{\beta = 0}^{[ \log (\mathcal T^{1/3}T)] + 1}  &n_{\mathcal T,T}^\beta(x)  \mathcal T^{-2/3}   e^{2\beta}\\ 
 &\leq  4\cdot 10^2 C^2_\Gamma    \sum_{\beta = 0}^{[ \log (\mathcal T^{1/3}T)] + 1}   
 (e^{-2\beta} \mathcal T^{2/3} T^2) \mathcal T^{-2/3}   e^{2\beta} \\ 
 & \leq 4\cdot 10^2 C^2_\Gamma T^2 (1 + \log (\mathcal T^{1/3} T) )\,.
 \end{align*}
 By Proposition ~\ref{prop:count_deg}  we have
 \begin{align*}
 \sum_{\beta = 0}^{[ \log (\mathcal T^{1/3}T)] + 1}  &n_{\mathcal T,T}^{\beta,deg}(x)  
 \mathcal T^{-2/3}   e^{2\beta}\\ 
 &\leq  4 \cdot 10^2 C^2_\Gamma C'_\Gamma   \sum_{\beta = 0}^{[ \log (\mathcal T^{1/3}T)] 
 + 1}   (1+ e^{-\beta} \mathcal T^{1/3} T)     \mathcal T^{-2/3}   e^{2\beta} \\ 
 & \leq 8 \cdot 10^2 e^2 C^2_\Gamma C'_\Gamma  T^2    \,.
 \end{align*}
 The argument is therefore complete.
\end{proof}

\subsection{The rescaling argument}

In this section we prove Theorem~\ref{theo:equidistribution}.  Let us
recall that the twisted horocycycle flow on $M \times \T$ projects to
the horocycle flow on $M$.  Because the equidistribution of the
horocycle flow on $M$ is well understood (see \cite{BuFo}, \cite{Bur},
\cite{FF1}, \cite{St2}), we restrict our considerations to functions
with zero average along the circle action on $M \times \T$. Such
functions are in the orthogonal complement of functions constant along
the circle action with respect to any of the Sobolev norms considered
in this paper.

Let $r\geq 0$, $s > 1/2$, and let $T \geq 1$.  Let us denote the
ergodic integral \eqref{equa:Twisted integral1} by
\[
\gamma_{\bar x, \lambda}^T := \frac{1}{T} \int_0^{T} (\phi_t^\lambda
(\bar x))^* dt\,.
\]
We will estimate $\gamma_{\bar x, \lambda}^T$ by iteratively rescaling
the (intermediate) Sobolev norms as in \cite{FF2} and \cite{FF3}.  Let
$h \in [1, 2]$ be a real number and $l \in \N$ be such that $T =
e^{lh}$.  For all integers $j \in [0, l]$, let
\[
\mathcal T_j := e^{(l - j)h}\,, \ \ \, T_j := T/ \mathcal T_j =
e^{jh}\,.
\]
Then as $j$ decreases from $l$ to $0$, the scaling parameter $\mathcal
T_j$ becomes larger, while the scaled length $T_j$ of the arc
$\{\phi_t^\lambda\}_{t = 0}^T$ becomes shorter.

Let $(\phi_t^{\lambda, \mathcal T_j})$ again denote the flow of the
scaled vector field $\mathcal T_j (U+\lambda K)$. Observe that by
change of variable we have
\[
\gamma_{\bar x, \lambda}^{T} = \frac{1}{T_j} \int_0^{T_j}
(\phi_t^{\lambda, \mathcal T_j} (\bar x))^* dt\,.
\]
Moreover, notice that if $\mathcal D \in \mathcal
I_\lambda^s(\Gamma)$, then for all $\mathcal T \geq 1$,
\[
\mathcal T (U + \lambda K) \mathcal D = 0\,.
\]
Hence, the space $\mathcal I_\lambda^s(\Gamma)$ is independent of the
scaling parameter.

By orthogonality it is enough to estimate $\gamma_{\bar x, \lambda}^T$
in each irreducible, unitary representation of $\widehat
W^{-r,-s}(M\times \T)$.  For any $\mu \in \text{spec}(\Box)$ and for
any $m\in \Z\setminus\{0\}$, let $H := H_{m, \mu}$ be an irreducible,
unitary representation of $\SL(2, \R)\times \T$.  The space $\widehat
H_{\mathcal T_j}^{-r,-s}$ has an orthogonal decomposition
\[
\widehat H_{\mathcal T_j}^{-r,-s} = (\mathcal I_\lambda^s(\Gamma) \cap
\widehat H_{\mathcal T_j}^{-r,-s}) \oplus^\perp (\mathcal
I_\lambda^s(\Gamma)^\bot \cap \widehat H_{\mathcal T_j}^{-r,-s})\,.
\]
Then $\gamma_{\bar x, \lambda}^{T}$ has a corresponding orthogonal
decomposition in $\widehat H_{\mathcal T_j}^{-r,-s}$ written as
\begin{equation}\label{equa:gamma_T-decomp-R-D}
  \gamma_{\bar x, \lambda}^T|_{\widehat H^{-r,-s}} = \mathcal D^j \oplus^\perp \mathcal R^j 
\end{equation}
with
\[
\begin{aligned}
  \mathcal D^j &:= \mathcal D^{r,s}_{x,\lambda,T,m,\mu,j} \in
  \mathcal I_\lambda^s(\Gamma) \cap \widehat H^{-r,-s}  \,, \\
  \mathcal R^j &:= \mathcal R^{r,s}_{x,\lambda,T,m,\mu,j} \in \mathcal
  I_\lambda^s(\Gamma)^\bot \cap \widehat H^{-r,-s} \,,
\end{aligned}
\]
We begin by estimating the scaled foliated Sobolev norms of the
remainder distribution $\mathcal R^j \in \widehat H^{-r,-s} $.

\begin{lemma}\label{lemma:Remainder}
  Let $s > 2$ and let $r\geq 3(s-1)$. There is a constant $C''_s:=C''_s(\Gamma) > 0$
  such that the following holds.  For all $\lambda \in \R^*$, $m\in
  \Z\setminus\{0\}$, for all $\bar x=(x,\theta) \in M \times \T$ and
  for all $T \geq 1$, for all $j\in [1,l] \cap \Z$, the distribution
  $\mathcal R^{j} \in \widehat H^{-r,-s}$ satisfies the estimate
  \[
  \begin{aligned}
    \vert\mathcal R^j\vert_{-r, -s; \mathcal T_j}
    \leq \frac{C''_s}{T_j}  & [ c_\Gamma(x,\mathcal T_j) + c_\Gamma(h_T(x),\mathcal T_j)] \\
    &\times (1+ \log^{1/2}\mathcal T_j ) \frac{1 + |\lambda m|^{-(s -
        1)}}{|\lambda m|} \,.
  \end{aligned}
  \]
\end{lemma}
\begin{proof}
  Let $s > 2$, let $r\geq 3(s-1)$ and let $f \in H^\infty$.  Let $f^*$
  be the orthogonal projection in $\widehat H_{\mathcal T_j}^{-s}$ of
  $f$ into $\text{Ann}(\mathcal I_\lambda^s(\Gamma))$. Because $f^*
  \in \text{Ann}_\lambda(\Gamma)$, by
  Theorem~\ref{theo:cohomology-principal} and
  Theorem~\ref{theo:cohomology-discrete} there exists a solution $g^*
  \in \widehat H^\infty$ of the equation
  \[
  \mathcal T_j (U + \lambda K) g^* = f^*\,,
  \]
  such that the following bounds holds:
  \begin{equation}\label{equa:coeqn-remainder1}
    \vert g^* \vert_{0, s-1; \mathcal T_j}  \leq 
    \frac{C_s}{{\mathcal T}^{1/3}} \frac{1 + \vert\lambda m\vert^{-(s-1)}}{\vert \lambda m\vert}  
    \vert  f\vert_{r, s;\mathcal T_j} \,.
  \end{equation}
  
  By orthogonality, again we have
  \[
  \mathcal R^j(f) = \gamma_{\bar x, \lambda}^T(f^*)\,.
  \]
  Then the fundamental theorem of calculus gives
  \begin{align}\label{equa:remainder_FTC}
    |\mathcal R^j(f)|
    & = |\gamma_{\bar x, \lambda}^T(f^*)| \notag \\
    & = |\frac{1}{T_j} \int_0^{T_j} f^* \circ \phi_t^{\lambda, \mathcal T_j}(x) dt| \notag \\
    & = |\frac{1}{T_j} \int_0^{T_j} \mathcal T_j (U + K) g^* \circ \phi_t^{\lambda, \mathcal T_j}(x) dt| \notag \\
    & = |\frac{1}{T_j} \left(g^* \circ \phi_{T_j}^{\lambda, \mathcal
        T_j}(x) - g^*(x)\right)|
  \end{align}

  By Theorem~\ref{theo:Sobolev-trace}, for all $s>2$ there exists a
  constant $C'_s>0$ such that
  \[
  \begin{aligned}
    \eqref{equa:remainder_FTC} \leq \frac{C'_s}{T_j} [w_{\mathcal T_j}(x, \lambda, 1)^{-1/2} &+ w_{\mathcal T_j}(\phi_{T_j}^{\lambda, \mathcal T_j}(x), \lambda, 1)^{-1/2}] \\
    &\times (\vert f^* \vert_{0, s-1; \mathcal T_j} + \vert g^*
    \vert_{0, s-1; \mathcal T_j})\,,
  \end{aligned}
  \]
  any by Corollary ~\ref{cor:average_width} there exists a constant
  $K_\Gamma>0$ such that
  \[
  \begin{aligned}
    w_{\mathcal T_j}(x, \lambda, 1)^{-1/2} &+ w_{\mathcal
      T_j}(\phi_{T_j}^{\lambda, \mathcal T_j}(x), \lambda, 1)^{-1/2}
    \\ &\leq K_\Gamma [c_\Gamma(x,\mathcal T_j)+ c_\Gamma(h_T(x),\mathcal
    T_j)] (1+ \log^{1/2} \mathcal T_j)\,.
  \end{aligned}
  \]
  Lemma~\ref{lemma:Remainder} then follows from the above estimates.
\end{proof}

Next, we estimate the scaled foliated norms of invariant
distributions.
\begin{lemma}\label{lemma:iterate}
  For every $s > 2$ and $r\geq 2s$, there is a constant $C_{r,s} > 0$
  such that
  \[
  \begin{aligned}
    \vert \mathcal D^l \vert_{-r,-s} &\leq C_{r,s} T^{-1/6} (1 +
    |\lambda m|^{-3s}) \\ & \times \left(\vert\mathcal
      D^0\vert_{-(r-2s),-s; T} + \sum_{j = 1}^{l} T_{j}^{1/6}\vert
      \mathcal R^{j }\vert_{-(r-2s), -s; \mathcal T_{j}}\right)\,.
  \end{aligned}
  \]
\end{lemma}
\begin{proof}
  We follow the proof of Lemma 5.7 of \cite{FF2}.  For each integer $j
  \in [1, l]$, let $\mathbf I_j^{r,s} := \mathbf I_j^{r,s}(m, \mu, T)$
  on $\widehat H^{-r,-s}$ be orthogonal projection onto $\langle
  \mathcal D_j\rangle$ in the Hilbert space $\widehat H_{\mathcal
    T_j}^{-r,-s}$\,.  We get from definitions that
  \[
  \mathcal D^j = \mathbf I_j^{r,s}(\mathcal D^{j - 1} + \mathcal R^{j
    - 1}) = \mathcal D^{j - 1} + \mathbf I_j^{r,s}(\mathcal R^{j -
    1})\,.
  \]
  Iteratively applying the triangle inequality, we get
  \begin{align}\label{equa:iterate-step1}
    \vert \mathcal D^l\vert_{-r,-s} & \leq \vert \mathcal D^{l - 1}
    \vert_{-r,-s} +
    \vert \mathbf I_l^{r,s}(\mathcal R^{l - 1})\vert_{-r,-s} \notag  \\
    & \leq \vert\mathcal D^0\vert_{-r,-s} + \sum_{j = 1}^{l }
    \vert \mathbf I_{l - j+1}^{r,s}(\mathcal R^{l - j})\vert_{-r,-s} \notag \\
    & = \vert\mathcal D^0\vert_{-r,-s} + \sum_{j = 1}^{l} \vert
    \mathbf I_{j}^{r,s}(\mathcal R^{j-1})\vert_{-r,-s}\,.
  \end{align}

  By Theorem~\ref{theo:scaling-principal} and
  Theorem~\ref{theo:Discrete_scale}, for any $s>2$ and any $r\geq 2s$
  there exists a constant $C'_{r,s} > 0$ such that
  \begin{equation}
    \label{eq:inv_dist_0}
    \begin{aligned}
      \vert\mathcal D^0\vert_{-r,-s} &\leq C'_{r,s} \mathcal
      T_0^{-1/6} (1 + |\lambda m|^{-3s}) \vert\mathcal
      D^0\vert_{-(r-2s), -s; \mathcal T_0} \\ & = C'_{r,s} T^{-1/6}
      (1 + |\lambda m|^{-3s}) \vert\mathcal D^0\vert_{-(r-2s), -s; T}
      \,,
    \end{aligned}
  \end{equation}
  and for all integer $j\in [1,l]$, since $\mathbf
  I_{j}^{r,s}(\mathcal R^{j - 1})\in \langle D^j \rangle$, we have
  \begin{equation}
    \label{eq:inv_dist_j}
    \begin{aligned}
      \vert \mathbf I_{j}^{r,s}&(\mathcal R^{j - 1})\vert_{-r,-s} \leq
      C'_{r,s} \mathcal T_j^{-1/6} (1 + |\lambda m|^{-3s}) \vert
      \mathbf I_{j}^{r,s}(\mathcal R^{j - 1})\vert_{-(r-2s), -s;
        \mathcal T_j} \\ & = C'_{r,s} T^{-1/6} (1 + |\lambda
      m|^{-3s})T_j^{1/6} \vert \mathbf I_{j}^{r,s}(\mathcal R^{j -
        1})\vert_{-(r-2s), -s; \mathcal T_j} \,,
    \end{aligned}
  \end{equation}
  Finally we observe that $\frac{\mathcal T_{j - 1}}{\mathcal T_{j}} =
  e^{h},$ hence there is a constant $C''_{r,s} > 0$ such that
  \[
  \vert \mathcal R^{j - 1}\vert_{-(r-2s), -s; \mathcal T_{j}} \leq
  C''_{r,s} \vert\mathcal R^{j - 1}\vert_{-(r-2s), -s; \mathcal T_{j -
      1}} \,.
  \]
  The lemma then follows from the bounds in
  formulas~\eqref{equa:iterate-step1},~\eqref{eq:inv_dist_0}
  and~\eqref{eq:inv_dist_j}.
\end{proof}

Recall from \eqref{eq:C_Gamma1} that for all $(x,T)\in M\times \R^+$
\[
C_\Gamma(x,T) = \sup_{ 0\leq t \leq T} c_\Gamma(x,t) \,.
\]
From Lemma~\ref{lemma:Remainder} and Lemma~\ref{lemma:iterate}, we
prove
\begin{theorem}\label{theo:D-lambda-estimate}
  For every $s > 2$ and $r\geq 5s-3$, there is a constant
  $C^{(3)}_{r,s}:=C^{(3)}_{r,s}(\Gamma) > 0$ such that for all $T\geq 1$ we have
  \[
  \begin{aligned}
    \vert\mathcal D^l\vert_{-r,-s} \leq C^{(3)}_{r,s} &[C_\Gamma(x,T) + C_\Gamma(h_T(x),T)]  \\
    \times & (1 + |\lambda m|^{-4s}) T^{-1/6} (1+\log^{1/2} T) \,.
  \end{aligned}
  \]
\end{theorem}
\begin{proof}
  Orthogonality shows
  \begin{equation}\label{equa:D^0-to-integral}
    \vert\mathcal D^0\vert_{0, -s; T} \leq \vert  \int_0^1 (\phi_t^{\lambda, T}(x))^* dt 
    \vert_{0, -s; T} \,.\end{equation}
  By Theorem~\ref{theo:trace_thm} and Theorem~\ref{theo:average_width},
  we get a constant $C^{(3)}_s> 0$ such that 
  \[
  \vert\mathcal D^0\vert_{-(r-2s),-s; T} \leq \vert\mathcal
  D^0\vert_{0, -s; T} \leq C^{(3)}_s [c_\Gamma(x,T) +
  c_\Gamma(h_T(x),T)] (1+ \log^{1/2} T)\,.
  \]
  We observe that by the definitions since $\mathcal T_j \leq T$ for
  all $j \in [1,l]$, we have
  \[
  c_\Gamma(x,\mathcal T_j) + c_\Gamma(h_T(x),\mathcal T_j) \leq
  C_\Gamma(x,T) + C_\Gamma(h_T(x),T)\,.
  \]
  By Lemma \ref{lemma:Remainder} and Lemma \ref{lemma:iterate}, we get
  \[
  \begin{aligned}\label{equa:final-d-est}
    \vert\mathcal D^l \vert_{-r, -s} & \leq C^{(3)}_{r,s} (1 +
    |\lambda m |^{-3s}) [C_\Gamma(x,T) + C_\Gamma(h_T(x),T)]
    \notag \\
    &\times T^{-1/6} (1+ \log^{1/2} T) [1 + (\frac{1 + |\lambda m
      |^{-(s -1)}} {|\lambda m|} ) \sum_{j = 1}^{l-1} T_j^{-5/6}]
    \,. \notag
  \end{aligned}
  \]
  Since the series converges there exists a constant $C>0$ such that
  \[ [1 + (\frac{1 + |\lambda m |^{-(s -1)}} {|\lambda m|} ) \sum_{j =
    1}^{l-1} T_j^{-5/6}] \leq C (1 + |\lambda m|^{-4s}) \,,
  \]
  hence the argument is complete.
\end{proof}

\begin{proof}[Proof of Theorem~\ref{theo:equidistribution}] 
  Let $s > 2$, and let $r \geq 5s - 3$.  Set $j = l$ in the orthogonal
  decomposition \eqref{equa:gamma_T-decomp-R-D}.  Define $\mathcal
  R_{\bar x, T}^{r,s}$ to be the direct integral in $\widehat W^{-r,
    -s}(M)$ of the distributions $T \mathcal R^l$ taken across all
  irreducible, unitary representations of the group $\SL(2, \R) \times
  \T$ on $L^2(M\times \T)$.  Analogously define $\mathcal
  D^{r,s}_{\bar x, \lambda, T}$ to be the direct integral of the
  distributions $T^{1/6} \mathcal D^l$.  By
  Theorem~\ref{theo:D-lambda-estimate} and orthogonality, there is a
  constant $C_{r,s}:=C_{r,s}(\Gamma) > 0$ such that for any $\bar x= (x,\theta) \in
  M\times \T$ and for any $T\geq e$,
  \[
  \vert\mathcal D^{r,s}_{\bar x, \lambda, T}\vert_{-r, -s}^2\leq C_{r,
    s} [C_\Gamma(x,T) + C_\Gamma(h_T(x),T)]^2 (1 + |\lambda|^{-8s})
  (1+\log T)\,.
  \]
  Similarly, Lemma~\ref{lemma:Remainder} and orthogonality imply
  \[
  \vert \mathcal R_{\bar x, T}^{r,s}\vert_{-r,-s}^2 \leq C_{r, s}
  [C_\Gamma(x,T) + C_\Gamma(h_T(x),T)]^2 \frac{1 + |\lambda|^{-2(s -
      1)}}{|\lambda|^2} (1+\log T)\,.
  \]
  These two estimates together give the estimates
  \eqref{equa:Remainder-inv-dist-twist_flow} in
  Theorem~\ref{theo:equidistribution}. This concludes the argument.
\end{proof}

\section{Horocycle maps}

In this section we prove Theorem \ref{theo:maps} on the effective
equidistribution of horocycle maps. The argument is based on the
effective equidistribution of horocycle flows~\cite{FF1} and of
twisted horocycle flows (see Theorem~\ref{theo:equidistribution}) as
well as bounds on solution of the cohomological equation for horocycle
maps, which we prove below.

\subsection{Cohomological equation}

Given $f \in C^\infty(M)$, we obtain Sobolev estimates of solutions to
the cohomological equation
\begin{equation}\label{equa:coeqn-map}
  f = g \circ h_L - g\,,
\end{equation}
for $L > 0$.  Such estimates have already been obtained in Theorem 1.2
of \cite{T} by studying the horocycle map directly and without
reference to the twisted horocycle flow.  Here we derive improved
estimates from our solution to the cohomological equation of the
twisted horocycle flow.

We recall that the horocycle flow is \emph{stable} in the sense that
for every non-trivial irreducible representation $H_\mu$ and for every
function $f\in H_\mu^\infty$ which belongs to the kernel of all
horocycle flow invariant distributions, the cohomological equation $U
g = f$ has a unique solution $g \in H^\infty_\mu$ (see Theorem 1.2 of
\cite{FF1}).  It is therefore possible to define a Green operator
$G^U_\mu$ for the horocycle flow defined on the kernel
$\text{Ann}(\mathcal I^{0}_\mu) \subset H^\infty_\mu$ of the space
$\mathcal I^{0}_\mu$ of all horocycle flow invariant
distributions. The Green operator is uniquely defined by the following
identity:
\[
U G^U_\mu (f) =f \, \quad \text{ for all } \, f \in
\text{Ann}(\mathcal I^{0}_\mu)\,.
\]
As the horocycle vector field $U$, as a linear differential operator
on $H^\infty_\mu$ is represented in Fourier transform on the space
$L^2_\nu(\R)$ by the multiplier $\hat U = i \xi$, it follows that the
Green operator $G^U_\mu$ is represented by the multiplier
\[
\hat G^U_\mu (\hat f) (\xi) = \frac{\hat f (\xi)}{i\xi} \,, \quad
\text{ for all } \,\hat f \in \text{Ann}(\mathcal I^{0}_\mu)\,.
\]
For every $\epsilon >0$ let $\vert U\vert ^\epsilon$ denote the
self-adoint operator defined by the spectral theorem as a function of
the skew-adjoint operator $U$ on every unitary representation.

\begin{theorem}\label{theo:coeqn-map-principal}
  For all $s \geq 0$ there is a constant $C_{s} > 0$ such that the
  following holds.  For all irreducible, unitary representations
  $H_\mu$ of the principal or complementary series, for all $\epsilon
  \in (0,1)$, for all $r, a \geq 0$ and for all $f \in H_\mu^\infty
  \cap \text{Ann}^{L}(\Gamma)$, there is a unique solution $g \in
  H_\mu^\infty$ of the cohomological equation \eqref{equa:coeqn-map}
  for the time-$L$ horocycle map $h_L$ and we have the following
  estimates:
  \[
  \Vert g \circ h_{L/2}\Vert_{r, s, a} \leq C_{s}
  \left(\frac{1+L^{2s}}{L} \Vert G^U_\mu(f)\Vert_{r,s, a} +
    \frac{L^\epsilon (1+ L^s)}{\epsilon} \Vert f \Vert_{r + s, s + 1,
      a + \epsilon} \right),
  \]
  and, with respect to full Sobolev norms, we have
  \[
  \| g \circ h_{L/2} \|_s \leq C_{s} \left( \frac{1+L^{2s}}{L} \|
    G^U_\mu(f)\|_s + \frac{L^\epsilon (1+ L^s)}{\epsilon} \Vert f
    \Vert_{2s+1+\epsilon} \right)\,.
  \]
\end{theorem}
\begin{proof}
  By taking the Fourier transform of both sides of equation
  \eqref{equa:coeqn-map}, we get the formula
  \[
  \hat g(\xi) = \frac{\hat f(\xi)}{e^{i L \xi} - 1} \,.
  \]
  The estimate will be carried out separately on the intervals $\vert
  \xi\vert \leq \pi/L$, $\xi \geq \pi/L$ and $\xi \leq -\pi/L$.  On
  the bounded interval $I_L:= [-\pi/L, \pi/L]$ by the definition of
  $G_\mu^U(f)$ we can write
  \[
  e^{ i \xi L/ 2} \hat g(\xi) = \frac{1}{L} \left(e^{i \xi
      L/2}\frac{iL\xi}{e^{i L \xi} - 1}\right) \hat G^U_\mu(\hat f)\,.
  \]
  Observe that the function
  \[
  \phi (\eta) := e^{i \eta / 2} \frac{i\eta}{e^{i \eta} - 1}
  \]
  is infinitely differentiable and bounded on $[- \pi, \pi]$, and note
  $L \xi \in [- \pi, \pi]$ whenever $\xi \in I_L$.  Let $\alpha, \beta
  \in \N$.  For $0\leq \ell \leq \alpha$ and $i+m \leq \beta $, let
  \[
  \phi^{(\ell)}_{i,m}(\xi):= (\frac{d}{d\xi})^m \hat V^i (2 \xi
  \frac{d}{d\xi})^{\alpha-\ell} \phi(L \xi) \,.
  \]
  There exists a constant $K_{\alpha, \beta}^{(0)} > 0$ such that
  \[
  \|\phi^{(\ell)}_{i,m}\|_{L^\infty(I_L)} \leq K_{\alpha, \beta}^{(0)}
  (1 + |\nu|)^i L^{m + 2i + \alpha - \ell} \,.
  \]

  By formula (\ref{equa:VXproduct}) we have universal coefficients
  $(a^{(\alpha)}_\ell) (b^{(\beta)}_{ijkm})$ such that
  \[
  \begin{aligned}
    (\hat V^\beta \hat X^\alpha) [ e^{i \xi L/ 2} \hat g (\xi)] &=
    -\frac{1}{L} \sum_{\ell\leq \alpha} \, \sum_{ \substack {
        i+j+ m\leq \beta  \\
        k \leq m }}
    a^{(\alpha)}_\ell b^{(\beta)}_{ijkm} \phi_{i, m}^{(\ell)} (\xi) \\
    &\ \ \ \ \ \ \ \ \ \ \ \ \ \ \ \ \ \ \ \ \ \ \ \ \ \ \times [(\hat
    X-(1-\nu))^k \hat V^j \hat X^\ell G_\mu^U(f)](\xi) \,.
  \end{aligned}
  \]
  By the triangle inequality, it follows that there exists a constant
  $K_{\alpha, \beta}^{(1)}>0$ such that
  \begin{align}\label{equa:transfer-map-small-xi}
    \Vert \hat V^\beta \hat X^\alpha \hat g \Vert_{L^2(I_L)} &\leq
    K_{\alpha, \beta}^{(1)} L^{-1} (1+ L^{\alpha+2\beta}) \notag \\ &
    \ \ \ \ \ \ \ \ \ \ \ \ \ \ \ \ \ \ \ \ \ \ \ \ \ \ \times
    \sum_{i+j+k \leq \alpha+\beta} (1+ \vert \nu\vert )^i \Vert \hat
    V^j \hat X^k G^U_\mu (f) \Vert_{L^2_\nu(\R)} \,.
  \end{align}
 
  On the half-lines $\xi \geq \pi/L$ and $\xi \leq -\pi/L$ we proceed
  in a different way.  By a formula in complex analysis (see for
  instance Chap. V, \S 4 in~\cite{Cartan}) we have
  \[
  \frac{1}{\sin(\xi L/2)} = \frac{1}{\pi}\left(\frac{2\pi}{\xi L} +
    \sum_{n \geq 1} (-1)^n \frac{\xi L/\pi}{(\xi L/ 2 \pi)^2 - n^2}
  \right)\,.
  \]
  Hence, we can write
  \[
  \begin{aligned}
    i e^{i \xi L/2} \hat g(\xi) & = \hat f(\xi) \frac{2}{\xi L} +
    \frac{1}{\pi}\sum_{n \geq 1} \frac{4\pi (-1)^n}{L} \hat f(\xi)
    \frac{\xi}{\xi^2 - (2\pi n/L)^2} \,.
  \end{aligned}
  \]
  For all $\alpha$, $\beta \geq 0$, we will estimate the $L^2_\nu(\R)$
  norm of
  \[
  \hat V^\beta \hat X^\alpha (i e^{i \xi L/2} \hat g)\,.
  \]

  Let $\epsilon \in (0,1)$. For $\xi \geq 0$ we can write
  \begin{equation}
    \label{eq:sol_expansion}
    \begin{aligned}
      i e^{i \xi L/2} \hat g(\xi) & = \frac{2}{L} \frac{\hat
        f(\xi)}{\xi} + \frac{4}{L}\sum_{n \geq 1}
      (-1)^n\left(\frac{\xi^{1-\epsilon}}{\xi + 2\pi n/L}\right)
      \left(\frac{\xi^\epsilon \hat f(\xi)}{\xi - 2\pi n/L}\right) \,.
    \end{aligned}
  \end{equation}
  For $0\leq \ell \leq \alpha$ and $i+m \leq \beta $, let
  \begin{equation}
    \label{eq:phi_n}
    \phi^{(\ell), \epsilon}_{i,m, n}(\xi):= (\frac{d}{d\xi})^m \hat V^i (2 \xi \frac{d}{d\xi})^{\alpha-\ell} 
    (\frac{\xi^{1-\epsilon}}{\xi+ 2\pi n/L})\,.
  \end{equation}
  Then as in formula \eqref{equa:VXproduct}, there exist universal
  coefficients $(a^{(\alpha)}_\ell)$, $(b^{(\beta)}_{ijkm})$ such that
  \begin{equation}
    \begin{aligned}\label{eq:VXproduct3}
      \hat V^\beta \hat X^\alpha &\left(\frac{\xi^{1-\epsilon}}{\xi +
          2\pi n/L}\right) \left(\frac{\xi^\epsilon \hat f(\xi)}{\xi -
          2\pi n/L}\right) = -i \sum_{\ell\leq \alpha} \, \sum_{
        \substack { i+j+ m\leq \beta \\ k \leq m}}
      a^{(\alpha)}_\ell b^{(\beta)}_{ijkm}  \\
      &\times \phi^{(\ell), \epsilon}_{i,m, n}(\xi) ((\hat
      X-(1-\nu))^k \hat V^j \hat X^\ell) \left(\frac{\xi^{\epsilon}
          \hat f(\xi)}{\xi - 2\pi n/L}\right) \,.
    \end{aligned}
  \end{equation}
   
  We prove below a bound for the functions
  $\phi^{(\ell),\epsilon}_{i,m, n}$, defined in
  formula~\eqref{eq:phi_n}.

  For $(\ell, i, m) = (\alpha, 0, 0)$, since $\xi \geq 0$ and $2\pi n
  /L\geq 0$, for all $\epsilon \in (0, 1)$ we have the estimate
  \begin{equation}
    \label{eq:max_est}
    \xi + \frac{2\pi n}{L} \geq  \xi^{1-\epsilon} (\frac{2\pi n}{L})^{\epsilon}\,.
  \end{equation}
  It follows that, for all $\epsilon>0$ and for all $\xi \geq 0$ we
  have
  \begin{equation}
    \label{eq:phi_1}
    \Vert \phi^{(\ell),\epsilon}_{i, 0, 0}\Vert_{L^\infty(\R^+)} \leq  (\frac{L}{2\pi n})^{\epsilon}\,.
  \end{equation}
  For $(\ell, i, m) \neq (\alpha, 0, 0)$, let
  \begin{align}\label{equa:W-coefficients}
    W_{\ell,i,m} := \{(w_0,w_1, w_2) \in &\N^2 \times \N \setminus\{0\} \vert w_0\leq i \,, \notag \\
    & w_1 \leq w_2 \leq w_1 +i+m \leq \alpha-\ell + 2 i +m \}\,.
  \end{align}
  By induction we prove that there exist universal constants
  $\{c_w\vert w\in W_{\ell,i,m}\}$ such that
  \begin{equation}\label{equa:mialphaell-decomp}
    (\frac{d}{d\xi})^m \hat V^i (2 \xi \frac{d}{d\xi})^{\alpha-\ell} = \sum_{w\in 
      W_{\ell,i,m}} c_w \, \nu^{w_0} \xi^{w_1} \frac{d^{w_2}}{d\xi^{w_2}} \,.
  \end{equation}
  For $j\in \N$, we also compute that
  \[
  \frac{d^j}{d\xi^j} \left(\frac{\xi}{\xi+ 2\pi n/L}\right) =
  \frac{2\pi n}{L} \frac{(-1)^{j + 1} j!}{(\xi + 2\pi n/L)^{j + 1}}\,,
  \]
  hence there exist universal constants $\{c'_{w, k}\vert w\in
  W_{\ell,i,m}, 0\leq k \leq w_2\}$ such that
  \[
  \begin{aligned}
    \phi^{(\ell),\epsilon}_{i,m, n} (\xi) = & \sum_{w\in W_{\ell,i,m}}
    c'_{w,w_2}
    \begin{pmatrix} -\epsilon \\ w_2 \end{pmatrix} \, \nu^{w_0}
    \frac{\xi^{w_1-w_2 +1-\epsilon} }{\xi+ 2\pi n/L} \\
    & \ \ \ \ \ \ \ \ \ \ \ + \sum_{w\in W_{\ell,i,m}}
    \sum_{k=0}^{w_2-1} c'_{w,k} \begin{pmatrix} -\epsilon \\
      k \end{pmatrix} \,\nu^{w_0} \frac{n}{L}
    \frac{\xi^{w_1-\epsilon}}{ \xi^{k}(\xi + 2\pi n /L)^{w_2 -k + 1}}
    \,.
  \end{aligned}
  \]
  Therefore, we derive from formulas~\eqref{eq:max_est} and from the
  above formula that there is a constant $K_{\alpha, \beta}^{(2)} > 0$
  such that, for all $\epsilon\in (0,1)$
  \begin{equation}\label{eq:phi_2}
    \vert \phi^{(\ell),\epsilon}_{i, m, n}(\xi) \vert \leq K_{\alpha, \beta}^{(2)} (1 + |\nu|)^i  
    (1+ L^{i+m}) (\frac{L}{n} )^{\epsilon} \,, \quad \text{ for all } \,\xi \geq \pi/L\,.
  \end{equation}

  By the uniform bounds~\eqref{eq:phi_1} and~\eqref{eq:phi_2} on the
  functions $\phi^{(\ell),\epsilon}_{i, m, n}$ on the half-line
  $\R^+_L:=\{ \xi\vert \xi \geq \pi/L\}$, it follows that there is a
  constant $K_{\alpha, \beta}^{(3)} > 0$ such that
  \[
  \begin{aligned}
    \|& \sum_{ \substack {
        i+j+ m\leq \beta  \\
        \ell \leq \alpha, \, k \leq m }} a^{(\alpha)}_\ell
    b^{(\beta)}_{ijkm} \phi^{(\ell)}_{i,m, n} (\hat X-(1-\nu))^k \hat
    V^j \hat X^\ell
    \left(\frac{ \xi^\epsilon \hat f(\xi)}{\xi - \frac{2\pi n}{L}}\right) \|_{L^2_\nu(\R^+_L)} \notag \\
    & \leq K_{\alpha, \beta}^{(3)} \left(\frac{L}{n}\right)^{\epsilon}
    \sum_{\substack{i + j + k \leq \alpha + \beta\\ \ell \leq \alpha}}
    (1 + |\nu|)^i (1+L^{\beta-j}) \notag \| \hat V^j \hat X^k
    \left(\frac{ \xi^\epsilon \hat f(\xi)}{\xi - \frac{2\pi
          n}{L}}\right) \|_{L^2_\nu(\R^+)} \label{equa:UXV-f-n}.
  \end{aligned}
  \]
  By formula~\eqref{eq:VXproduct3}, by the above estimate and by
  Theorem~\ref{theo:cohomology-principal} it follows that there exists
  a constant $K_{\alpha, \beta}^{(4)} > 0$ such that
  \[
  \begin{aligned}
    \| \hat V^\beta \hat X^\alpha \left(\frac{\xi^{1-\epsilon}}{\xi +
        2\pi n/L}\right)& \left(\frac{\xi^\epsilon \hat f(\xi)}{\xi -
        2\pi n/L}\right) \|_{L^2_\nu(\R^+_L)} \leq K_{\alpha,
      \beta}^{(4)} \left( \frac{L }{n}\right)^{1+\epsilon} (1
    +L^\beta) \\ \times (1+&\vert \nu \vert)^{\beta} \sum_{i+j+k \leq
      \alpha+\beta+1 } (1+ \vert \nu \vert)^i \| \hat V^j \hat X^k
    \vert \hat U\vert^\epsilon f \|_{L^2_\nu(\R^+)} \,.
  \end{aligned}
  \]
  Since $\epsilon >0$, there exists a constant $K >0$ such that
  \[
  \sum_{n\geq 1} \left( \frac{L }{n}\right)^{1 + \epsilon} \leq
  \frac{K}{\epsilon} L^{1+\epsilon}\,,
  \]
  hence by formula~\eqref{eq:sol_expansion}, by the above estimates
  and by the triangle inequality, there exists a constant $K_{\alpha,
    \beta}^{(+)} >0$ such that

  \begin{equation}
    \label{eq:est_+}
    \begin{aligned}
      \|\hat V^\beta& \hat X^\alpha (e^{i \xi L/2} \hat
      g)\|_{L^2_\nu(\R^+_L)} \leq \frac{2}{L} \| \hat V^\beta \hat
      X^\alpha \hat G^U_\mu(f) \|_{L^2_\nu(\R^+)} + \frac{K_{\alpha,
          \beta}^{(+)}}{\epsilon}\, L^{\epsilon}
      \\
      & \times (1+ L^{\beta}) (1+\vert \nu \vert)^{\beta} \sum_{i+j+k
        \leq \alpha+\beta+1 } (1+ \vert \nu \vert)^i \| \hat V^j \hat
      X^k \vert \hat U\vert^\epsilon f \|_{L^2_\nu(\R^+)} \,.
    \end{aligned}
  \end{equation}
  The argument is analogous for $\xi \leq 0$.  In this case, write
  \[
  \begin{aligned}
    i e^{i \xi L/2} \hat g(\xi) & = \frac{2}{L} \frac{\hat
      f(\xi)}{\xi} + \frac{4}{L}\sum_{n \geq 1}
    (-1)^{n+1}\left(\frac{\vert \xi\vert^{1-\epsilon}}{\xi - 2\pi
        n/L}\right) \left(\frac{ \vert \xi\vert^{\epsilon}\hat
        f(\xi)}{\xi + 2\pi n/L}\right) \,,
  \end{aligned}
  \]
  and proceed as before.  We conclude that on $\R^-_L: =\{ \xi \vert
  \xi \leq -\pi/L\}$ we have
  \begin{equation}
    \label{eq:est_-}
    \begin{aligned}
      \|\hat V^\beta& \hat X^\alpha (e^{i \xi L/2} \hat
      g)\|_{L^2_\nu(\R^-_L)} \leq \frac{2}{L} \| \hat V^\beta \hat
      X^\alpha \hat G^U_\mu(f) \|_{L^2_\nu(\R^-)} + \frac{K_{\alpha,
          \beta}^{(-)}}{\epsilon}\, L^{\epsilon}
      \\
      & \times (1+ L^{\beta}) (1+\vert \nu \vert)^{\beta} \sum_{i+j+k
        \leq \alpha+\beta+1 } \vert 1- \nu\vert^i \| \hat V^j \hat X^k
      \vert \hat U\vert^\epsilon f \|_{L^2_\nu(\R^-)} \,.
    \end{aligned}
  \end{equation}

  For $r \in \N$, and let $s \in \N$ be an even integer.  For $\alpha
  + \beta \leq s$, it follows that is a constant $C_{r, s}^{(6)} > 0$
  such that
  \[
  \begin{aligned}
    |g \circ h_{L/2}|_{r, s} & \leq C_{r, s}^{(6)} (1 + \vert\nu\vert)^r [ \frac{1}{L} | G_\mu^U(f)|_{0, s}  \\
    & + L^{\epsilon} (1+ L^{\beta}) (1 + \vert\nu\vert)^s \sum_{k = 0}^{s + 1} \vert 1 - \nu\vert^k \vert \vert U\vert^\epsilon f\vert_{0, s + 1 - k} ]  \\
    & \leq C_{r, s}^{(6)} \left( \frac{1}{L} | G_\mu^U(f)|_{r, s} +
      L^{\epsilon} (1+ L^{\beta}) \vert \vert U\vert^\epsilon
      f\vert_{r + s, s + 1}\right) \,.
  \end{aligned}
  \]
  This concludes the proof of the first estimate of the theorem, for
  $r \in \N$ and $s \in 2 \N$.  The estimate for arbitrary real $r
  \geq 0$ and $s \geq 0$ follows by interpolation, which concludes the
  estimate in the case $a = 0$.

  Now let $a \in \N$ be an even integer.  As $(I - U^2)^{a/2}$
  commutes with the horocycle map $h_L$, the function $(I - U^2)^{a/2}
  g$ is a solution to the cohomological
  equation~\eqref{equa:coeqn-map} with coboundary $(I - U^2)^{a/2} f$.
  Hence, the above argument applies, and the first estimate holds for
  all even $a \in \N$. It then holds for all real exponents $a\geq 0$
  by interpolation.  The second estimate follows from the first.
\end{proof}

The case of irreducible unitary representations of the discrete series
is completely analogous.  We have
\begin{theorem}\label{theo:coeqn-map-discrete}
  For all $s \geq 0$ there is a constant $C_{s} > 0$ such that the
  following holds.  For all irreducible, unitary representations
  $H_\mu$ of the mock discrete series or discrete series, for all
  $\epsilon \in (0,1)$, for all $r, a \geq 0$ and for all $f \in
  H_\mu^\infty \cap \text{Ann}^{L}(\Gamma)$, there is a unique
  solution $g \in H_\mu^\infty$ of the cohomological equation
  \eqref{equa:coeqn-map} for the time-$L$ horocycle map $h_L$ and we
  have the following estimates:
  \[
  \Vert g \circ h_{L/2}\Vert_{r, s, a} \leq C_{s} \left(
    \frac{1+L^{2s}}{L} \Vert G^U_\mu(f)\Vert_{r,s, a}+
    \frac{L^\epsilon (1+ L^s)}{\epsilon} \Vert f \Vert_{r + 3s, s + 1,
      a + \epsilon}\right).
  \]
  and, with respect to full Sobolev norms, we have
  \[
  \Vert g \circ h_{L/2} \Vert_s \leq C_{s} \left( \frac{1+L^{2s}}{L}
    \Vert G^U_\mu(f)\Vert_s + \frac{L^\epsilon (1+ L^s)}{\epsilon}
    \Vert f \Vert_{4s+1+\epsilon}\right)\,.
  \]
\end{theorem}
\begin{proof}
  The argument follows the proof of
  Theorem~\ref{theo:coeqn-map-principal}.  This time the Fourier
  transform is defined on $\R^+$, and the estimates for the twisted
  cohomological equation are derived from
  Theorem~\ref{theo:cohomology-discrete} instead of
  Theorem~\ref{theo:cohomology-principal}.  This accounts for the
  worse loss of derivatives we have for solutions of the cohomological
  equation in the discrete series compared to the principal and
  complementary series (compare the statements of
  Theorem~\ref{theo:coeqn-map-principal} and
  Theorem~\ref{theo:coeqn-map-discrete}).
\end{proof}

\subsection{Effective equidistribution}

We recall that $H_\mu$ denotes an irreducible, unitary representation
space of $L^2(M)$ with Casimir parameter $\mu\in \text{spec}(\Box)$.
For each $k \in \Z / \{0\}$ and $L > 0$, let $\mathcal D_{\mu, k, L}$
be the distribution defined on smooth functions by
\begin{equation}\label{equa:dist-invar-map-twist}
  \mathcal D_{\mu, L, k} (f) = D_{k, \mu}^{2\pi/L}(e_k \otimes f)\,.
\end{equation} 
It follows from Lemma~\ref{lemma:regularity-D_lambda-XV-principal} and
Lemma~\ref{lemma:D^lambda-reg-discrete} that $\mathcal D_{\mu, L, k}
\in \widehat H_\mu^{-(1/2 +)}.$ By Theorem~1.1 of \cite{FF1},
$H_\mu^{-\infty}$ contains a normalized basis of at most two invariant
distributions for the horocycle flow, which can be taken to be
generalized eigendistributions for the geodesic flow.  Of course,
these basis elements are also invariant under the map $h_{L}$, and we
denote them by $\mathcal D_{\mu}^{+}$ and $\mathcal D_{\mu}^{-}$.

Observe that for any $(x,N) \in M\times \N$, the ergodic sum
$\frac{1}{N} \sum_{n = 0}^{N - 1} (h_{nL}(x))^* $ is a measure.  Then
for any $s > 1$, $r \geq 0$, $a> 1/2$, for any $L>0$ and for any
$(x,N) \in M\times \N$, there exist a horocycle flow-invariant
distribution $\mathcal D^0_{x, N, L, r, s, a}$ an $h_L$-invariant
distribution ${\mathcal D}^{\text{twist}}_{x, N, L, r, s, a}$ and a
remainder distribution $\mathcal{R}_{x, N, L, r, s, a}$ such that the
following (orthogonal) decomposition holds in the Sobolev space
$W^{-r, -s, -a}(M)$:
\begin{equation}\label{equa:asymptotic-formula}
  \sum_{n = 0}^{N - 1} (h_{n L}(x))^* =  (\mathcal D^0_{x, N, L, r, s, a} + \mathcal D^{\text{twist}}_{x, N, L, r, s, a}) + \mathcal{R}_{x, N, L, r, s, a} \,.
\end{equation}
In addition, since each $\mathcal D_{\mu, L, k} \in \widehat
H_\mu^{-(1/2+)}$, the following decompositions hold: there are a
sequence of complex numbers $\left\{c_{\mathcal D_{\mu, k, L}}(x, N,
  L, r, s, a)\right\}_{\mu \in \text{spec}(\Box)}$ such that
\[
\mathcal D^{\text{twist}}_{x, N, L, r, s, a}:= \bigoplus_{\mu \in
  \text{spec}(\Box)} \sum_{k \in \Z - \{0\}}c_{\mathcal D_{\mu, k,
    L}}(x, N, L, r, s, a) \mathcal D_{\mu, k, L}\,,
\]
and a sequence of complex numbers $\left\{c_{\mathcal D_{\mu,
      L}^\pm}(x, N, L, r, s, a)\right\}_{\mu \in \text{spec}(\Box)}$
such that
\[
\mathcal D^0_{x, N, L, r, s, a} := \bigoplus_{\mu \in \sigma_{pp}}
c^s_{\mathcal D_\mu^{+}}(x, N, L, r, s, a) \mathcal D_{\mu, L}^{+} +
c^s_{\mathcal D_\mu^{-}}(x, N, L, r, s, a) \mathcal D_{\mu, L}^{-}\,.
\]
Lemma~7.6 of \cite{T} and \eqref{equa:sobolev-inclusions_map} show
that for sufficiently large $r, s$ and $a\geq 0$, the distribution
$\mathcal D^{\text{twist}}_{x, N, L, r, s, a}$ is given by a
convergent series in each irreducible Sobolev subspace of $W^{-r, -s,
  -a}(M)$.

In what follows, we estimate the distributions in the decomposition
\eqref{equa:asymptotic-formula}.  The first step is to derive Sobolev
estimates for $\mathcal{R}_{x, N, L, r, s, a}$ from estimates for the
solutions of the cohomological equation of horocycle maps proved
above.

Recall that for all $(x,L, N)\in M\times \R^+ \times \N$, we set
\begin{equation}
  \label{eq:D_const}
  D_\Gamma(x,L,N):= e^{d_M(h_{-L/2}(x))} +e^{ d_M (h_{L(N - 1/2)}(x))}\,.
\end{equation}
\begin{lemma}\label{lemma:remainder-map}
  For any $r>3$, $s>2$, $a>2$ and $\epsilon \in (0, 1)$, there is a
  constant $C_{r,s, a, \epsilon} := C_{r,s, a, \epsilon}(\Gamma) > 0$
  such that
  \[
  \Vert \mathcal{R}_{x, N, L, r, s, a} \Vert_{-r, -s, -a} \leq C_{r,s,
    a, \epsilon} D_\Gamma(x,L,N) \frac{1+L^{2+\epsilon}}{L} \,.
  \]
\end{lemma}
\begin{proof}
  Let $H_\mu$ be an irreducible, unitary representation of $\SL(2,
  \R)$.  By Theorem~\ref{theo:coeqn-map-principal} and
  Theorem~\ref{theo:coeqn-map-discrete}, for any function $f \in
  Ann^{L}(\Gamma) \cap H^\infty_\mu$ there is a unique solution $g \in
  H^\infty_\mu$ of the cohomological equation~\eqref{equa:coeqn-map}
  for the time-$L$ horocycle map, and for any $\sigma\geq 0$ the
  function $g$ satisfies the estimate
  \begin{equation}
    \label{equa:apply-coeqn-map_1}
    \begin{aligned}
      \Vert g \circ h_{L/2}\Vert_{0, \sigma, 0} \leq &C_{\sigma}
      \frac{1+L^{2\sigma}}{L} \Vert G^U_\mu(f)\Vert_{0,\sigma, 0} \\ +
      &C_{\sigma} \frac{L^\epsilon (1+ L^\sigma)}{\epsilon} \Vert f
      \Vert_{3\sigma, \sigma + 1, \epsilon}\,.
    \end{aligned}
  \end{equation}
  Because $f\in H^\infty_\mu$ is a smooth coboundary, we get as in
  Lemma~\ref{lemma:Remainder} that
  \[
  \begin{aligned}\label{equa:remarinder-est-map}
    \vert \mathcal R_{x, N, r, s, a} (f)\vert & \leq \vert g\circ h_{NL}(x)\vert + \vert g(x)\vert \notag \\
    & \leq \vert g\circ h_{L/2} (h_{L(N - 1/2)}(x))\vert + \vert
    g\circ h_{L/2} (h_{-L/2}(x))\vert \,.
  \end{aligned}
  \]
  By Theorem~\ref{theo:Sobolev-trace} and by
  Lemma~\ref{lemma:fund_box}, for all $\sigma>1$ there exists a
  constant $C_{\sigma, \Gamma}>0$ such that
  \begin{equation}\label{equa:map-trace}
    \vert \mathcal R_{x, N, r, s, a} (f)\vert \leq C_{\sigma, \Gamma} D_\Gamma(x,L,N)  (\Vert U g \circ h_{L/2} \Vert_{0, \sigma, 0} + \Vert g \circ h_{L/2} \Vert_{0, \sigma, 0})
  \end{equation}
  Notice that $U g$ is a solution to the cohomological equation
  \eqref{equa:coeqn-map} for the time-$L$ horocycle map with
  coboundary $U f$, hence by Theorem~\ref{theo:coeqn-map-principal}
  and Theorem~\ref{theo:coeqn-map-discrete}, we get a constant
  $C'_{\sigma} > 0$ such that for all $\epsilon \in (0,1)$,
  \begin{equation}
    \label{equa:apply-coeqn-map_2}
    \begin{aligned}
      \Vert U g \circ h_{L/2} \Vert_{0, \sigma, 0} \leq &C'_\sigma
      \frac{1+L^{2\sigma}}{L} \Vert G^U_\mu(U f)\Vert_{0,\sigma, 0} \\
      + 2&C'_\sigma \frac{L^{\epsilon/2} (1+ L^{\sigma})}{\epsilon}
      \Vert f \Vert_{ 3\sigma, \sigma + 1, 1 + \epsilon/2}\,.
    \end{aligned}
  \end{equation}
  By definition, $U G^U_\mu(U f) = U f$, hence by uniqueness of
  solutions for the cohomological equation of the horocycle flow, (see
  Theorem~1.2 of \cite{FF1} or the proof of
  Theorem~\ref{theo:cohomology-principal}), we have $G^U_\mu(U f) =
  f$.  By the bounds in formulas~\eqref{equa:apply-coeqn-map_1},
  \eqref{equa:map-trace} and~\eqref{equa:apply-coeqn-map_2} it then
  follows that
  \begin{equation}\label{equa:apply-coeqn-map2}
    \begin{aligned}
      \vert\mathcal R_{x, N, r, s, a} (f)\vert &\leq
      C'_{\sigma,\Gamma} D_\Gamma(x,L,N)
      \frac{1+L^{2\sigma}}{L} (\Vert f \Vert_{0,\sigma,0} + \Vert G^U_\mu(f)\Vert_{0,\sigma, 0}) \\
      &+ 2C'_{\sigma, \Gamma} D_\Gamma(x,L,N) \frac{L^{\epsilon/2} (1+
        L^{\sigma})}{\epsilon} \Vert f \Vert_{ 3\sigma, \sigma + 1,
        1+\epsilon/2} \,.
    \end{aligned}
  \end{equation}

  Now observe that Theorem~1.2 of \cite{FF1} shows that for all
  $\epsilon' > 0$, there is a constant $C_{{\sigma}, \epsilon'} :=
  C_{{\sigma}, \epsilon'}(\Gamma) \geq 1$ such that
  \[
  \Vert G_\mu^U(f) \Vert_{0, \sigma, 0} \leq \Vert G_\mu^U(f)
  \Vert_\sigma \leq C_{\sigma, \epsilon'} \Vert f \Vert_{{\sigma} +1 +
    \epsilon'} \leq C_{\sigma, \epsilon'} \Vert f \Vert_{0, {\sigma}
    +1 + \epsilon', {\sigma} +1 + \epsilon'} \,.
  \]
  Taking $\sigma$, $\epsilon, \epsilon'$ so that $r >3\sigma$, $s$, $a
  > \sigma+1 +\epsilon'$ and $2\sigma < 2 +\epsilon$, we get that
  \[
  \vert\mathcal R_{x, N, r, s, a} (f)\vert \leq C_{r,s,a, \epsilon}
  D_\Gamma(x,L,N) \frac{1+L^{2+\epsilon}}{L} \Vert f \Vert_{ r, s, a}
  \,.
  \]

  This proves Lemma~\ref{lemma:remainder-map} for coboundaries in each
  irreducible, unitary representation of $\SL(2, \R)$, and the general
  statement now follows by orthogonality of the decomposition of
  $W^{r,s,a}(M)$ into irreducible components and of every function
  $f\in W^{r,s,a}(M)$ into a coboundary component and an orthogonal
  component.
\end{proof}

For each $k \in \Z / \{0\}$, let
\[
\mathcal D_{L, k} := \bigoplus_{\mu \in \text{spec}(\Box)} c_{\mathcal
  D_{\mu, L, k}}(x, N, L, r, s, a) \mathcal D_{\mu, L, k}\,.
\]
By applying $\frac{1}{L}\int_0^L e^{2\pi i \frac{k}{L} t} (h_t(x))^*
dt$ or $\frac{1}{L} \int_0^L (h_t(x))^* dt$ to the ergodic sum, Lemma
7.2 of \cite{T} and the definition of $\mathcal D_{L, k}$ and
$\mathcal D^0$ give

\begin{lemma}\label{lemma:isolate-dist}
  For all $(x, N, L) \in M \times \Z^+ \times \R^+$ and $k \in
  \Z\setminus \{0\}$, we have the following distributional identities
  in $\mathcal E'(M)$:
  \[
  \begin{aligned}
    \mathcal{D}_{L, k}  & = \frac{1}{L} \int_0^{N L} e^{2\pi i t k / L} (h_t(x))^* dt  - \frac{1}{L} \int_0^L e^{2\pi i t k / L} h_{-t} \mathcal{R}_{x, N, L, r, s, a} dt \,; \notag \\
    \mathcal{D}^0 & = \frac{1}{L} \int_0^{N L} (h_t(x))^* dt -
    \frac{1}{L}\int_0^L h_{-t} \mathcal{R}_{x, N, L, r, s, a} dt \,.
  \end{aligned}
  \]
\end{lemma}

As noted, Lemma~7.6 of \cite{T} and
\eqref{equa:sobolev-inclusions_map} show that for sufficiently large
$r, s$ and $a$, the distribution $ \mathcal D^{\text{twist}}_{x, N, L,
  r, s, a} $ is given by a convergent series in each irreducible
component of $W^{-r, -s, -a}(M)$.  For such $(r, s, a)$, we have
\begin{equation}
  \label{equa:D^twist-D_{L, k}}
  \mathcal D^{\text{twist}}_{x, N,L, r, s, a} = \sum_{k \in \Z / \{0\}} \mathcal D_{L, k}\,.
\end{equation}

The distribution $\mathcal{R}_{x, N, L, r, s, a}$ is controlled by
Lemma~\ref{lemma:remainder-map}.  We now estimate the integral of
$\mathcal{R}_{x, N, L, r, s, a}$ along the horocycle flow.
\begin{lemma}\label{lemm:integral-remainder}
  For all $r>3$, $s>2$, $a > 2$ and $\epsilon\in (0,1)$ there exists a
  constant $C'_{r,s, a,\epsilon} > 0$ such that
  \[
  \begin{aligned}
    \frac{1}{L} \int_0^L | \mathcal R_{x, N, L, r, s, a}(f \circ h_t)|
    dt & \leq
    C_{r,s, a,\epsilon}' D_\Gamma(x,L,N)\\
    &\times \frac{1+ L^{2+\epsilon}}{L} (1+ L^{2s}) \Vert f \Vert_{r,
      s, s+a}\,.
  \end{aligned}
  \]
\end{lemma}
\begin{proof}
  By Lemma~\ref{lemma:remainder-map}, for all $r>3$, $s>2$, $a>2$ and
  $\epsilon\in (0,1)$ there exists a constant $C_{r,s,a, \epsilon}>0$
  such that, for all $f \in C^\infty(M)$, we have
  \[
  \begin{aligned}
    \frac{1}{L} \int_0^{L} |\mathcal{R}_{x, N, L, r, s, a}(f \circ
    h_t)| dt &\leq C_{r,s, a, \epsilon} D_\Gamma(x,L,N)\\ &\times
    \frac{1+ L^{2+\epsilon}}{L^2} \int_0^{L } \Vert f \circ
    h_t\Vert_{r, s, a} dt \,.
  \end{aligned}
  \]
  Also notice that for any $f \in H_\mu^\infty$,
  \begin{align}\label{equa:commutation}
    U (f \circ h_t) & = (U f) \circ h_t \,, \notag \\
    X ( f\circ h_t) & = [(X + t U) f] \circ h_t\,, \notag \\
    V (f\circ h_t) & = [(V - 2 t X - t^2 U) f] \circ h_t\,.
  \end{align}
  Hence, for all $r, s, a \in \N$, there is a constant $C_{s} > 0$
  such that
  \[
  \Vert f \circ h_t\Vert_{r, s, a} \leq C_{s} (1 + t^{2s}) \Vert
  f\Vert_{r, s, s+a}\,.
  \]
  By interpolation, this estimate holds for all $r, s, a \geq 0$,
  hence the lemma follows.
\end{proof}

Next we estimate the Sobolev norms of the distribution $ \mathcal
D^{\text{twist}}_{x, N, L, r, s, a}$.  For all $(x,L,N)\in
M\times\R^+\times N$, let $C_\Gamma(x,L,N)$ denote the positive
constant defined by the formula
\begin{equation}
  \label{eq:C_Gamma2}
  C_\Gamma(x,L,N):= C_\Gamma(x,NL)+C_\Gamma(h_{NL}(x),NL) \,.
\end{equation}
\begin{lemma}\label{lemma:D-twist}
  Let $s > 2$, $a>2$ and $r> 5s-3$. For all $\epsilon > 0$, there
  exists a constant $C_{r,s,a, \epsilon}^{(2)} >0$ and for $(x,N)\in
  M\times \N$ and $L>0$ there exists a decomposition
  \[
  \mathcal D^{\text{twist}}_{x, N, L, r, s, a} = \widetilde{\mathcal
    D}^{\text{twist}}_{x, N, L, r, s, a, \epsilon} +
  \widetilde{\mathcal R}^{\text{twist}}_{x, N, L, r, s, a, \epsilon}
  \]
  such that the following estimates hold:
  \[
  \begin{aligned}
    \Vert \widetilde{\mathcal D}^{\text{twist}}_{x, N, L, r, s, a,
      \epsilon} \Vert_{-r,-s,-(1+\epsilon)}
    &\leq C_{r,s,a, \epsilon}^{(2)} C_\Gamma(x,L,N) \\
    &\times   (1+ L^{8s+\epsilon}) (NL)^{5/6}  \log^{1/2} (e+ NL)]\,; \\
    \Vert \widetilde{\mathcal R}^{\text{twist}}_{x, N, L, r, s, a,
      \epsilon} \Vert_{-r,-s,-(s+a+1+\epsilon)} &\leq C_{r,s,a,
      \epsilon}^{(2)} D_\Gamma(x,L,N) (1+ L^{2s+2+\epsilon}) \,.
  \end{aligned}
  \]
\end{lemma}
\begin{proof}

  Notice that for any $\mu\in \text{spec}(\Box)$, for any $L>0$, $k
  \in \Z / \{0\}$ and for any $\epsilon \in \R^+$, by the definition
  of the invariant distributions $\mathcal D_{\mu, L, k}$ we have
  \begin{equation}\label{D_{mu, k}-U^sigma}
    \mathcal D_{\mu, L, k}
    = ( 1 + \frac{4\pi^2 k^2}{L^2})^{-\epsilon/2} (I- U^2)^{\epsilon/2} \mathcal D_{\mu, L, k} \,.
  \end{equation}
  Then by Lemma~\ref{lemma:isolate-dist} we derive the formula
  \begin{align}\label{equa:D_{L,k}-twist-remainder}
    \mathcal D_{L, k} &=(1 + \frac{4\pi^2 k^2}{L^2})^{-\epsilon/2} (I- U^2)^{\epsilon/2} \left(\frac{1}{L} \int_0^{N L}e^{2\pi i  t k/L} (h_t(x))^* dt\right) \notag \\
    & -(1 + \frac{4\pi^2 k^2}{L^2})^{-\epsilon/2}(I- U^2)^{\epsilon/2}
    \left( \frac{1}{L} \int_0^L e^{2\pi i t k/L} h_{-t}
      \mathcal{R}_{x, N, L, r, s, a} dt \right) \,.
  \end{align}
  Then integration by parts shows
  \[
  \begin{aligned}
    \frac{1}{L} \int_0^{NL } &e^{2\pi i t k/L} (I- U^2)^{\epsilon/2}
    f\circ h_t(x) dt \\ &= \frac{1}{2\pi i k} [(I- U^2)^{\epsilon/2} f
    \circ h_{N L}(x) - (I- U^2)^{\epsilon/2} f(x)] \\ &- \frac{1}{2\pi
      i k} \int_0^{N L} e^{2\pi i t k/L} U (I- U^2)^{\epsilon/2}f\circ
    h_t(x) dt \,.
  \end{aligned}
  \]
  By Theorem~\ref{theo:equidistribution}, for all $s > 2$, for $r>
  5s-3$, there is a constant $C_{r,s} > 0$ such that, for all
  $\epsilon\in \R^+$ we have
  \begin{align}\label{equa:twist_integral-Dtwist}
    \Vert (I-U^2)^{\epsilon/2} & \left( \frac{1}{ L } \int_0^{N L}
      e^{2\pi i  t k/L}  (h_t(x))^* dt \right) \Vert_{-r,-s, -(1+\epsilon)} \notag \\
    & \leq \frac{C_{r,s}}{|k|} (1+ L^{8s}) C_\Gamma(x,L,N) (NL)^{5/6}
    \log^{1/2}(e+\vert NL \vert ) \,.
  \end{align}
  Now we estimate the integral of the remainder distribution from
  \eqref{equa:D_{L,k}-twist-remainder}.  Integration by parts gives
  \[
  \begin{aligned}
    |\frac{1}{L}&\int_0^L e^{2\pi i t k/L} \mathcal{R}_{x, N, L, r, s, a}  ( (I-U^2)^{\epsilon/2} f \circ h_t) dt|  \\ & \leq \frac{1}{2\pi i k}|\mathcal{R}_{x, N, L, r, s, a} ((I-U^2)^{\epsilon/2} f \circ h_{L}) -\mathcal{R}_{x, N, L, r, s, a} (((I-U^2)^{\epsilon/2} f)| \\
    & + \frac{1}{2\pi i k}|\int_0^L e^{2\pi i t k/L} \mathcal{R}_{x,
      N, L, r, s, a} (U (I-U^2)^{\epsilon/2} f) \circ h_t) dt |\,.
  \end{aligned}
  \]
  Then Lemma~\ref{lemm:integral-remainder} shows that for all $r>3$,
  $s>2$, $a > 2$ and $\epsilon'\in (0,1)$, there is a constant
  $C_{r,s, a, \epsilon'} > 0$ such that
  \begin{align}\label{equa:remainder_Dtwist1}
    \Vert (I-U^2)^{\epsilon/2} &\left( \frac{1}{L}\int_0^L e^{2\pi i t
        k/L} h_{-t} \mathcal{R}_{x, N, L, r, s, a}
      dt  \right)\Vert_{-r, -s, -(s+a  + 1+\epsilon)} \notag \\
    &\,\,\,\,\,\,\,\,\,\,\,\,\,\,\,\, \leq \frac{C_{r, s, a,
        \epsilon'} } {|k|} D_\Gamma(x,L,N) (1+ L^{2s + 2 + \epsilon'})
    \,.
  \end{align}
  A similar estimate holds for the finite factor.

  \smallskip Now define
  \begin{align}
    \widetilde{\mathcal D}^{\text{twist}}_{x, N, L, r, s, a, \epsilon}
    &:=
    \sum_{k \in \Z / \{0\}}  (1 + \frac{4\pi^2 k^2}{L^2})^{-\epsilon/2} \notag \\
    &\times (I- U^2)^{\epsilon/2}   \left(\frac{1}{L} \int_0^{N L} e^{2\pi i k t / L} (h_t(x))^* dt \right)\; \label{equa:def-dtwist_tilde}\\
    \widetilde{\mathcal R}^{\text{twist}}_{x, N, L, r, s, a, \epsilon}
    & :=
    \sum_{k \in \Z / \{0\}} (1 + \frac{4\pi^2 k^2}{L^2})^{-\epsilon/2} \notag \\
    & \times (I- U^2)^{\epsilon/2} \left( \frac{1}{L} \int_0^L e^{2\pi
        i t k/L} h_{-t} \mathcal{R}_{x, N, L, r, s, a} dt \right)
    \,.\notag
  \end{align}
  By construction and by formula~\ref{equa:D_{L,k}-twist-remainder},
  we have that
  \[
  \mathcal D^{\text{twist}}_{x, N, L, r, s, a} = \widetilde{\mathcal
    D}^{\text{twist}}_{x, N, L, r, s, a, \epsilon} +
  \widetilde{\mathcal R}^{\text{twist}}_{x, N, L, r, s, a,
    \epsilon}\,.
  \]
  In addition, by the estimate in
  formula~\eqref{equa:twist_integral-Dtwist}, for any $\epsilon>0$,
  there exists a constant $C_{r,s,\epsilon}>0$ such that
  \[
  \begin{aligned}
    \Vert \widetilde{\mathcal D}^{\text{twist}}_{x, N, L, r, s, a,
      \epsilon}& \Vert_{-r,-s, -(1+\epsilon)} \leq C_{r,s,\epsilon}
    (1+ L^{8s+\epsilon}) \\ &\times C_\Gamma(x,L,N) (NL)^{5/6}
    \log^{1/2} (e+ NL)\,.
  \end{aligned}
  \]
  By the estimate in formula~\eqref{equa:remainder_Dtwist1}, there
  exists a constant $C_{r,s,a,\epsilon}>0$ such that
  \[
  \Vert \widetilde{\mathcal R}^{\text{twist}}_{x, N, L, r, s, a,
    \epsilon} \Vert_{-r,-s, -(s+a+1+ \epsilon)} \leq
  C_{r,s,a,\epsilon} D_\Gamma(x,L,N) (1+ L^{2s + 2 + \epsilon}) \,.
  \]
\end{proof}

\begin{proof}[Proof of Theorem~\ref{theo:maps}]
  By formula \eqref{equa:asymptotic-formula} the distribution given by
  the ergodic sum of the time-$L$ horocycle map for a point $x\in M$
  and up to time $N\in \N$ can be decomposed into a distribution
  $\mathcal D^0_{x, N, L, r, s, a}$ invariant under the horocycle
  flow, a distribution $\mathcal D^{\text{twist}}_{x, N, L, r, s, a}$
  invariant under the horocycle map (but not under the horocycle flow)
  and a remainder distribution $\mathcal{R}_{x, N, L, r, s, a}$.

  The estimate for the distribution $\mathcal D^0_{x, N, L, r, s, a}$
  follows from
  Lemmas~\ref{lemma:isolate-dist}~and~\ref{lemm:integral-remainder},
  and the estimate for the distribution $\mathcal D^{\text{twist}}_{x,
    N, L, r, s, a} $ follows from Lemma~\ref{lemma:D-twist}.  Finally,
  the estimate for the remainder term $\mathcal{R}_{x, N, L, r, s, a}$
  is given by Lemma~\ref{lemma:remainder-map}.  This concludes the
  proof of Theorem~\ref{theo:maps}.
\end{proof}

\section{A result on Shah's question}

In this this section we prove Theorem~\ref{thm:Main_Shah}. We follow
the approach from Theorem 3.1 of the paper \cite{V} by A.~Venkatesh
where, for $f \in C^\infty(M)$ of zero average, the sum
\[
\frac{1}{N} \sum_{n= 1}^N f\circ h_{n^{1 + \delta}}(x)
\]
is controlled by sampling $f$ along suitable arithmetic progressions.

\begin{proof}[Proof of Theorem~\ref{thm:Main_Shah}]
  Let $f \in C^\infty(M)$ such that $\int_M f d\vol = 0$.
  
By Taylor formula for every fixed $N\in \N\setminus\{0\}$ and for every $t\geq 0$ we have
\begin{equation}
\label{eq:linearized}
  (N +t)^{1+\delta} = N^{1+\delta} +  (1+\delta) N^\delta t + O \left( N^{\delta-1} t^2 \right) \,.
 \end{equation}
  Thus the function $(N +t)^{1+\delta}$ is well approximated by its linear Taylor polynomial
  as long as $N^{\delta-1} t^2$ is small for $N$ large. 
  
 Motivated by the above remark we fix $\epsilon > 0$, we set $N_1 := [N^{1-\epsilon}]+1$, and
 for all $j \in \N \setminus \{0\}$ we define
  \[
  N_{j+1} := N_j + [N_j^{(1-\delta)/2-\epsilon}]\,.
  \]
  Let $J \in \N$ be such that $N_J \leq N \leq N_{J+1}$.  This
  implies in particular that $N-N_J \leq N_{J+1} -N_J =
  [N_J^{(1-\delta)/2-\epsilon}] \leq N^{(1-\delta)/2-\epsilon}$.  Hence, there
  is a constant $C_{f} > 0$ such that
  \[
  \begin{aligned}
   &\frac{1}{N} |\sum_{n = 0}^{N_1 - 1} f\circ h_{n^{1 + \delta}}(x)|
  \leq C_{f} N^{ - \epsilon} \,, \\ 
  &\frac{1}{N} |\sum_{n = N_{J}}^{N - 1} f\circ h_{n^{1 + \delta}}(x)|
  \leq C_{f} N^{-(1+\delta)/2 - \epsilon} \,,
  \end{aligned}
  \]
  and both terms converge to zero as $N \to +\infty$.  Then we need to estimate
  \[
  \frac{1}{N} |\sum_{n= N_1}^{N_J - 1} f\circ h_{n^{1 + \delta}}(x)|\,.
  \]

We  then let, for all $j\in \{1, \dots, J-1\}$,
  \begin{equation}
    \label{eq:Lj}
    L_j :=  (1+\delta) N_j^\delta \,.
  \end{equation}
  By the triangular inequality we have
  \begin{align}
    \frac{1}{N}|\sum_{n= N_1}^{N_J - 1} & f\circ h_{n^{1 + \delta}}(x)|
    \leq \frac{1}{N} |\sum_{j =
      1}^{J-1}\sum_{k=0}^{[N_j^{1-\delta-\epsilon}]-1}
    f\circ h_{N_j^{1 + \delta} + k L_j}(x)| \notag \\
    & + \frac{1}{N}|\sum_{n =N_1 }^{N_J - 1} f\circ h_{n^{1 +
        \delta}}(x) - \sum_{j =
      1}^{J-1}\sum_{k=0}^{[N_j^{1-\delta-\epsilon}]-1} f\circ
    h_{N_j^{1 + \delta} + k L_j}(x)|\,.
    \label{equa:Shah-to-arithmetic}
  \end{align}

  We begin by estimating the first term on the RHS in the above inequality which
  is composed of several sums along arithmetic progressions.

  By Theorem~\ref{thm:Main_Maps} and by the effective equidistribution
  of the horocycle flow (see Theorem 1.5 in \cite{FF1}), for every
  $\epsilon >0$ there exists a constant $C_{f,\epsilon}>0$ such that,
  for all $j \in \{1, \dots, J-1\}$ and for all $x \in M$, we have
  \begin{equation}
    \label{eq:Shah_est}
    \begin{aligned}
      |\sum_{k=0}^{[N_j^{(1-\delta)/2-\epsilon}]-1}& f\circ h_{N_j^{1 +
          \delta} + k L_j}(x)| \leq C_{f,\epsilon}
      L_j^{-1} (L_j N_j^{(1-\delta)/2-\epsilon})^{1 -\mathcal S^-_{\mu_0}} \\
      &+ C_{f,\epsilon} \left( L_j^{1/6+\epsilon} (L_j
        N_j^{(1-\delta)/2-\epsilon})^{5/6} (\log N_j)^{1/2} +
        L_j^{5+\epsilon}\right)\,.
    \end{aligned}
  \end{equation}

  By these inequalities, since $1-\mathcal S^-_{\mu_0}<1$, under the
  hypothesis that $\delta<1/13$, it follows from the estimate in
  formula~\eqref{eq:Shah_est} that there exists $\epsilon >0$
  (sufficiently small) such that for all $x\in M$ we have
  \[
  |\sum_{k=0}^{[N_j^{(1-\delta)/2-\epsilon}]-1} f\circ h_{N_j^{1 + \delta}
    + k L_j}(x)| \leq C_{f,\epsilon} [N_j^{(1-\delta)/2-\epsilon}]
  N_j^{-\epsilon} \,.
  \]
  \begin{sloppypar}
      Since by construction we have that
  $[N_j^{(1-\delta)/2-\epsilon}]=N_{j+1}-N_j$ and $N_j \geq N_1=
  [N^{1-\epsilon}]+1$ for all $j\in \{1,\dots, J-1\}$, and also $N_J \leq
  N$, by telescopic summation it then follows that for all $x\in M$ we
  have
  \[
  \frac{1}{N} |\sum_{j =
    1}^{J-1}\sum_{k=0}^{[N_j^{(1-\delta)/2-\epsilon}]-1} f\circ h_{N_j^{1
      + \delta} + k L_j}(x)| \leq C_{f,\epsilon} \frac{N_J}{N}
  N^{-\epsilon (1-\epsilon)} \leq C_{f,\epsilon} N^{-\epsilon
    (1-\epsilon)}\,.
  \]
  In particular we have proved that, uniformly over $x\in M$,
  \begin{equation}
    \label{eq:Shah_conv}
    \lim_{N\to +\infty} \frac{1}{N} \sum_{j = 1}^{J-1}\sum_{k=0}^{[N_j^{(1-\delta)/2-\epsilon}]-1} f\circ h_{N_j^{1 + \delta} + k L_j}(x)  \,=\, 0\,.
  \end{equation}
  \end{sloppypar}

  Now we estimate the second term on the RHS  of formula \eqref{equa:Shah-to-arithmetic}, that is
  \[
  \frac{1}{N}|\sum_{n = N_1}^{N_J - 1} f\circ h_{n^{1 + \delta}}(x) -
  \sum_{j = 1}^{J-1}\sum_{k=0}^{N_{j+1} -N_j -1} f\circ h_{N_j^{1 + \delta} + k
    L_j}(x)| \,.
  \]
   By Taylor formula~\eqref{eq:linearized}, for $0\leq k \leq
  [N_j^{(1-\delta)/2-\epsilon}] -1$, we have
  $$
   (N_j +k)^{1+\delta} - (N_j^{1+\delta} + k L_j) = O ( N_j^{-2\epsilon}) \,,
  $$
 hence there is a constant $C_f > 0$ such that, for all $0\leq k<N_{j+1}-N_j$,
 we have 
  \[
  |f\circ h_{(N_j+k)^{1+\delta}}(x) - f\circ h_{N_j^{1 + \delta} + k
    L_j}(x)| \leq C_f N_j^{-2\epsilon}\,.
  \]
It follows that there is a constant $C_f > 0$ such that
  \[
  \begin{aligned}
    |\sum_{n = N_1}^{N_J - 1} & f\circ h_{n^{1 + \delta}}(x)
    - \sum_{j = 1}^{J-1}\sum_{k=0}^{N_{j+1} -N_j-1} f\circ h_{N_j^{1 + \delta} + k L_j}(x)| \\
    & \leq \sum_{j = 1}^{J-1} \sum_{k=0}^{N_{j+1} -N_j-1}
    |f\circ h_{(N_j+k)^{1+\delta}}(x) -
    f\circ h_{N_j^{1 + \delta} + k L_j}(x) |  \\
    & \leq  C_f \sum_{j=1}^{J-1}  (N_{j+1}-N_j) N_j^{-2\epsilon} \leq C_f N N_1^{-2\epsilon}  \leq C_f 
   N  N^{-2\epsilon(1 - \epsilon)} \,.
  \end{aligned}
  \]
  Thus, we have that, uniformly over $x\in M$,
  \[
  \lim_{N \to +\infty} \frac{1}{N} \left(\sum_{n = N_1}^{N_J - 1} f\circ
    h_{n^{1 + \delta}}(x) - \sum_{j = 1}^{J-1}\sum_{k=0}^{N_j-1} f\circ
    h_{N_j^{1 + \delta} + k L_j}(x) \right) = 0\,.
  \]
  Theorem~\ref{thm:Main_Shah} follows from this, formulas
  \eqref{equa:Shah-to-arithmetic} and \eqref{eq:Shah_conv}.
\end{proof}

\appendix
\section{}\label{appe:A}
\subsection{Line and upper half-plane models of $\SL(2, \R)$}

The irreducible representation spaces for $\SL(2, \R)\times \T$ can be
studied in concrete, unitarily equivalent models.  We presently
describe the line and upper half-plane models for $\SL(2, \R)$.

Let $A = \left(\begin{array}{rr}
    a & b\\
    c & d
  \end{array}\right) \in SL(2, \mathbb{R}).$
Let $\mu \in \text{spec}(\Box)$ be a Casimir parameter, and let
$H_{\mu}$ be an irreducible, unitary representation space in the
kernel of $(\mu - \Box)$.  Let $\nu = \sqrt{1 - \mu}$ be a
representation parameter.  We denote by $H_\mu$ the following models
for the principal and complementary series representation spaces.  In
the first model (the line model) the Hilbert space is a space of
functions on $\R$ with the following norms.  If $\mu \geq 1$, then
$\nu \in i \mathbb{R}$ and $\| f \|_0 = \| f \|_{L^2(\mathbb{R})}.$ If
$0 < \mu < 1$, then $0 < \nu < 1$ and
\[
\| f \|_{H_{\mu}} = \left(\int_{\mathbb{R}^2} \frac{f(x)
    \overline{f(y)}}{|x - y|^{1 - \nu}} dx dy\right)^{1/2}.
\]
The group action is defined by
\[
\pi_\nu : SL(2, \mathbb{R}) \rightarrow \mathcal{U}(H)
\]
\[
\pi_{\nu} (A) f(x) = |-c x + a|^{-(\nu + 1)} f(\frac{d x - b}{-c x +
  a}),
\]
where $x\in \mathbb{R}$.

The vector fields for the model $H_\mu$ on $\mathbb{R}$ are
\[
\begin{array}{lll}
  X = -(1 + \nu) - 2 x \frac{\partial}{\partial x}; \\
  U = -\frac{\partial}{\partial x}; \\
  V = (1 + \nu) x + x^{2}\frac{\partial}{\partial x}.
\end{array}
\]

For $\mu \leq 0$, we let $\mathbb H$ be the upper half-plane.  The
upper half-plane model is also denoted by $H_\mu$, where now $\mu
\in\{ -n^2 + 2n :n \in \mathbb Z^+\}$, and its norm is
\[
\|f\|_{H_\mu} =
\begin{cases}
  \left(\int_{\mathbb H} |f(x+iy)|^2 \,y^{n-2}\,dx\,dy\right)^{1/2},  & n\ge 2\,;\\
  \left(\sup_{y>0} \int_{\R} |f(x+iy)|^2 \,dx\right)^{1/2} , & n=1\,.
\end{cases}
\]
This model has the group action $\pi_n: SL(2, \mathbb{R})\rightarrow
\mathcal{U}\left(H_\mu\right)$ defined by
% \begin{equation}\label{551}
\[
\pi_{n} (A): f(z) \rightarrow (-c z + a)^{-n} f(\frac{d z - b}{-c z +
  a})\,.
\]
%\end{equation}  
The anti-holomorphic discrete series is similar, but we only consider
the holomorphic case because there is a complex anti-linear
isomorphism between two series of the same Casimir parameter.

Then the vector fields in the model $H$ are:
\[
\begin{array}{lll}
  X = -(1 + \nu) - 2 z \frac{\partial}{\partial z}\; \\
  U = -\frac{\partial}{\partial z}\; \\
  V = (1 + \nu) z + z^{2}\frac{\partial}{\partial z}\,.
\end{array}
\]

\section{}\label{appe:B}

\begin{proof}[Proof of Lemma \ref{lemma:Fourier:comp}]
  If $\mu \geq 1$, then Lemma \ref{lemma:Fourier:comp} is immediate.
  So say $0 < \mu < 1$.  Then
  \begin{equation}\label{equa:norm}
    \|f\|_0^2 = \int_\R \int_\R \frac{f(x) f(y)}{|x - y|^{1 - \nu}} dx dy = \langle f * K, f\rangle\,,
  \end{equation}
  where $K(x) = |x|^{1-\nu}.$ A computation shows that the Fourier
  transform $\hat{K}(1)$ is defined, and moreover, for any $\xi \in
  \R$,
  \[
  \hat{K}(\xi) = |\xi|^{-\nu} \hat{K}(1)\,.
  \]
  Because $\hat K$ is not identically zero, we have $\hat{K}(1) \neq
  0$.

  Thus, Plancherel's equality gives
  \[
  \eqref{equa:norm} = \hat{K}(1) \int_\R |\hat f(\xi)|^2 |\xi|^{-\nu}
  d\xi\,.
  \]
\end{proof}

\begin{proof}[Proof of Lemma \ref{lemma:inversion}]
  Let $z = x + i y$.  Using Lemma~\ref{lemma:Fourier_Cauchy}, we have
  \begin{align}\label{equa:Fourier_Cauchy2}
    \int_\R f(x + i y) e^{-i \xi x} dx &= e^{-\xi y } \int f(z) e^{- i \xi z} dx\notag \\
    & = e^{-\xi y } \hat{f}^y(\xi) \notag  \\
    & = e^{-\xi y } \hat f(\xi)\,.
  \end{align}
  Because $\nu \geq 1$, Sobolev embedding shows $f(\cdot + i y)\in
  L^1(\R)$ for any $y \in \R^+$.  Moreover, as $f$ is smooth, $\int_\R
  f(x + iy) e^{- i \cdot x } dx$ is also in $L^1(\R)$.  , So the
  Fourier inversion formula followed by
  Lemma~\ref{lemma:Fourier_Cauchy} gives
  \begin{align}\label{equa:Fourier_inversion-Discrete}
    f(z) & = \frac{1}{2\pi} \int_\R \left(\int_\R f(t + iy) e^{- i \xi
        t } dt\right) e^{ i \xi x } d\xi \notag
    \\
    & = \frac{1}{2\pi} \int_\R \left(e^{\xi y} \int_\R f(t + iy) e^{-
        i \xi t } dt\right) e^{ i \xi z } d\xi \notag
    \\
    & = \frac{1}{2\pi} \int_{\R^+} \hat{f}^y(\xi) e^{ i \xi z } d\xi
    \notag
    \\
    & = \frac{1}{2\pi} \int_{\R^+} \hat f(\xi) e^{i \xi z } d\xi\,.
    \notag
  \end{align}

  We now consider the $L^2$ norm for $\nu = 0$.  The Plancherel
  theorem and formula~\eqref{equa:Fourier_Cauchy2} give
  \begin{align}
    \|f\|_0^2 & = \sup_{y > 0} \int_\R |f(x + i y)|^2 dx \notag \\
    & = \frac{1}{2 \pi} \sup_{y > 0} \int_{\R^+} |\int_\R f(x + i y) e^{-i \xi x} dx|^2 d\xi \notag \\
    & = \frac{1}{2\pi} \sup_{y > 0} \int_{\R^+} e^{-2 \xi y} |\hat f(\xi)|^2 d\xi \notag \\
    & = \frac{1}{2\pi} \int_{\R^+} |\hat f(\xi)|^2 d\xi\,.  \notag
  \end{align}

  For $\nu \geq 1$, we have
  % \begin{align}\label{equa:discrete_norm}
  \begin{align}%\label{equa:discrete-norm}
    \| f \|_0^2 & = \int_0^\infty \int_\R |f(x + i y)|^2 y^{\nu - 1}
    dx dy
    \notag \\
    & = \frac{1}{2\pi} \int_0^\infty \langle e^{-2 \xi y} \hat f, \hat
    f \rangle_{L^2(\R)} y^{\nu - 1} dy
    \notag \\
    & = \frac{1}{2\pi} \int_{\R} |\hat f(\xi)|^2 \int_0^\infty y^{\nu
      - 1} e^{-2 \xi y } dy d\xi\,.
  \end{align}
  % \end{align}
  Using integration by parts $\nu - 1$ times, we conclude
  \[
  \| f \|_0^2 = \frac{(\nu - 1)!}{2\pi} \int_{\R} |\hat f(\xi)|^2
  \frac{d\xi}{(2 \xi)^{\nu}}\,.
  \]
\end{proof}

\begin{proof}[Proof of Lemma \ref{lemma:comp_UT_est}:]

  When $H$ is a principal series representation for $G$, the operator
  $U_\mathcal T$ is unitary, so $H$ is in the complementary series.
  Recall that the norm for the line models is
  \[
  \| f \|_0 = \left(\int_{\mathbb{R}^2} \frac{f(x) \overline{f(y)}}{|x
      - y|^{1 - \nu}} dx dy\right)^{1/2}\,.
  \]
  From Lemma~3.1 of \cite{T}, the Fourier transform of $f$ is defined
  and continuous everywhere.

  Without loss of generality, assume $\lambda > 0$.  We have
  \[
  \|U_{\mathcal T} f\|_0^2 = \mathcal T^{1/3} \int_\R | \hat f (
  \lambda + {\mathcal T}^{1/3} (\xi-\lambda) )|^2 |\xi|^{-\nu} d\xi\,.
  \]

  Let $y - \lambda = \mathcal T^{1/3} (\xi - \lambda)$.  Then
  % \begin{equation}\label{equa:U_T control2}
  \[
  \xi^{-\nu} = \frac{\mathcal T^{\nu/3}}{(y + (\mathcal T^{1/3} -
    1)\lambda)^\nu}\,.
  \]
  % \end{equation}
  Because $y \in (\frac{\lambda}{2}, \frac{3\lambda}{2})$, we get the
  upper and lower bounds
  \[
  \lambda^{-\nu} \left(\frac{\mathcal T^{1/3}}{\mathcal T^{1/3} +
      1/2}\right)^{\nu} \leq \xi^{-\nu} \leq \left(\frac{\mathcal
      T^{1/3}}{\mathcal T^{1/3} - 1/2}\right)^{\nu} \lambda^{-\nu}\,.
  \]
  These bounds are made largest and smallest by setting $\nu = 1$ and
  $\mathcal T = 1$.  We get,
  \[
  \sqrt \frac{2}{3} \lambda^{-\nu/2} \| \hat f\|_{L^2(\R)} \leq
  \|U_{\mathcal T} f\|_0^2 \leq \sqrt 2 \lambda^{-\nu/2} \| \hat f
  \|_{L^2(\R)}\,.
  \]
  We have the same upper and lower bounds for $\| f \|_0$, so
  \[
  \frac{1}{\sqrt 3} \leq \frac{\| U_{\mathcal T} f\|_0}{\| f \|_0}
  \leq \sqrt 3\,.
  \]
  This completes the proof of Lemma~\ref{lemma:comp_UT_est}.
\end{proof}

\begin{proof}[Proof of Lemma \ref{lemma:U_T-comparable}]
  In this case, $I_\lambda = [\lambda - 1/2, \lambda + 1/2].$ Setting
  $(-1)! := 1$, Lemma \ref{lemma:inversion} gives
  \[
  \|U_{\mathcal T} f\|_0^2 = \frac{(\nu - 1)!}{(2 \pi)^{\nu}} \mathcal
  T^{1/3} \int_\R | \hat f ( \lambda + {\mathcal T}^{1/3}
  (\xi-\lambda) )|^2 |\xi|^{-\nu} d\xi\,.
  \]
  So let $y - \lambda = \mathcal T^{1/3} (\xi - \lambda)$, which means
  \[%\begin{equation}\label{equa:U_T control}
  \xi^{-\nu} = \frac{\mathcal T^{\nu/3}}{(y + (\mathcal T^{1/3} - 1)\lambda)^{\nu}}\,. 
  \]%\end{equation} 
  Observe $\xi^{-\nu}$ satisfies
  \[
  \left(\frac{\mathcal T^{1/3}}{\lambda \mathcal T^{1/3} +
      1/2}\right)^{\nu} \leq \xi^{-\nu} \leq \left(\frac{\mathcal
      T^{1/3}}{\lambda \mathcal T^{1/3} - 1/2}\right)^{\nu}\,.
  \]
  These bounds are made worse by setting $\mathcal T = 1$, so that we
  get
  \[
  \frac{(\nu - 1)!}{(2\pi (\lambda + 1/ 2))^{\nu}} \| \hat f
  \|_{L^2(\R)} \leq \|U_{\mathcal T} f\|_{\mathcal H_{\mu}} \leq
  \frac{(\nu - 1)!}{(2\pi (\lambda - 1/ 2))^{\nu}} \| \hat f
  \|_{L^2(\R)}\,.
  \]

  We get the same upper and lower bounds for $\| f\|_0$, so there is a
  constant $C > 0$ such that
  \[
  \left(\frac{1 - 1/(2\lambda)}{1 + 1/(2\lambda)}\right)^{\nu} \leq
  \frac{\|U_{\mathcal T} f \|_0}{\| f \|_0} \leq \left(\frac{1 +
      1/(2\lambda)}{1 - 1/(2\lambda)}\right)^{\nu}\,.
  \]
  Now because $\lambda \geq \nu + 1$, we get a constant $C > 0$ such
  that
  \[
  \frac{1}{C} \leq \frac{\|U_{\mathcal T} f \|_0}{\| f \|_0} \leq C\,.
  \]
\end{proof}

\section{}\label{appe:C}

\begin{proof}[Proof of Lemma~\ref{lemma:fund_box}]

  Let $\mathbb H := \{w = x + i y \in \C : \im(w) > 0\}$ be the
  Poincar\'e upper-half plane endowed with the Riemannian hyperbolic
  metric
  \[
  ds = \frac{\sqrt{dx^2 + dy^2}}{y}\,,
  \]
  and let $M$ be the unit tangent bundle of a surface $S=\Gamma
  \backslash \mathbb H$.

  When $M$ is compact, the horocycle flow has no periodic orbits and
  all orbits are transverse to the $X$-$V$ leaves.  By compactness,
  there is a constant $C_\Gamma' > 0$ such that for any $x \in M$
  $\alpha_x$ is injective on the domain
  \[
  [-C_\Gamma' , C_\Gamma'] \times [-1, 1] \times [-C_\Gamma' ,
  C_\Gamma' ]\,.
  \]
  Because $e^{-d_M(x)} < 1$, the lemma is proven if $M$ is compact.

  Now assume $M$ non-compact.  Let $\{C_i\}$ be the collection of
  disjoint cusps of the surface $S$ bounded by a cuspidal horocycles
  of length~$\ell_\Gamma<1$. By a cusp of $M$ we mean the tangent unit
  bundle $\tilde C_i\subset M$ of a cusp $C_i$.

  By compactness, there exists a constant $K_\Gamma > 0$ such that if
  $d_M(x) > K_\Gamma$ then the ball of center $x$ and radius $4$ is
  contained in some cusp $\tilde C_i$.

  For $d_M(x) \leq K_\Gamma$, the above argument gives a constant
  $C_\Gamma^{(2)} > 0$ such that $\alpha_x$ is injective on the domain
  \[ [-C_\Gamma^{(2)} e^{-d_M(x)}, C_\Gamma^{(2)}e^{-d_M(x)}] \times
  [-1, 1] \times [-C_\Gamma^{(2)}e^{-d_M(x)} , C_\Gamma^{(2)}
  e^{-d_M(x)}]\,.
  \]

  Let $C$ be any cusp. By conjugating the lattice if necessary, we may
  assume that the cusp $C$ has a fundamental domain $D=\{ z\in \mathbb
  H : |\Re z|\le 1/2, \im z > \ell_\Gamma^{-1}\}$ and that a parabolic
  subgroup $\Gamma$ stabilizing the cusp $C$ is the subgroup
  \[
  \left\{ \gamma_n:=\left(
      \begin{array}{cc}
        1 & n \\ 
        0 & 1 
      \end{array}
    \right): n\in \Z\right\} \subset \Gamma \,.
  \]
  The tangent unit bundle $\tilde D$ of $D$ is a fundamental domain of
  the cusp $\tilde C\subset M$.

  We consider the usual identification of $\SL(2, \R)/{\pm} I$ with
  the tangent unit bundle $T^1\mathbb H$ the mapping $ g \in \SL(2,
  \R)/\pm I \mapsto (g \cdot i , dg_i (i)) $. Under this
  identification, the domain $\tilde D$ is identified to
  \[
  \tilde D= \left \{ \pm \left(\begin{smallmatrix} a&b\\c&d
      \end{smallmatrix}\right): c^2+d^2 < \ell_{\Gamma} , \left|
      \tfrac{bd + a c}{c^2 + d^2} \right| < \tfrac 1 2 \right\}
  \]
  In fact
  \[
  \pm \left(
    \begin{array}{cc}
      a & b \\ 
      c & d 
    \end{array}
  \right) \cdot i = \frac{bd + a c}{c^2 + d^2} + i \frac{1}{c^2 +
    d^2}\,.
  \]

  For simplicity, for all $g = \pm \left(\begin{smallmatrix} a&b\\c&d
    \end{smallmatrix}\right) \in T^1\mathbb H\approx \SL_2(\R)/{\pm}
  I$, we define $\im(g) := \im(g \cdot i)$, i.e.\ $\im(g) =
  (c^2+d^2)^{-1}$. We also remark that, by the triangle inequality,
  there exists a constant~$c_{\Gamma}>0$ such that, if $\bar x \in
  \tilde D$ is a representative of $x\in \tilde C$, then
  \begin{equation}
    \label{eq:Twist-horo-6-13:3}
    \log \im (\bar x) -c_{\Gamma}\le d_M(x)\le \log \im (\bar x) +c_{\Gamma} \,.
  \end{equation}

  Our choice of the constant $K_\Gamma$ was motivated by the following
  observation: if $ \bar x\in \tilde D$ and $ \bar x_1 \in T^1\mathbb
  H\approx \SL_2(\R)$ are two representatives of a point $x\in \tilde
  C$ which are at a distance less than~$4$ from each other, then if
  $d_M(x)> K_\Gamma$ there exists $n\in \Z$ such that $\bar x_1=
  \gamma_n \bar x $ (in fact $|n|\le 4\, \im \bar x $).

  \smallskip Set, for conciseness,
  \[
  A(t)=\exp (tX/2),\qquad H(t) = \exp (tU),\qquad \bar H(t) =x\exp
  (tV)\,.
  \]
  Let $I_{\Delta}= [-\Delta,\Delta]\times [-1,1]\times
  [-\Delta,\Delta]$. By continuity, there exist $\Delta_0$ such that
  for all $\Delta< \Delta_0$ the map $ (t,y,z) \in I_{\Delta} \mapsto
  H(t)A(y)\bar{H}(z)\in \SL_2(\R)$ is a diffeomorphism onto its image
  satisfying, for all $(t,y,z) \in I_{\Delta}$,
  \[\operatorname{dist}( \operatorname{Id}_{\SL_2(\R)},
  H(t)A(y)\bar{H}(z)) < 2. \] Furthermore, by taking partial
  derivatives of the function
  \[
  F\colon ((t,y,z), (t',y',z'))\in I_{\Delta}\times I_{\Delta} \mapsto
  H(t)A(y)\bar{H}(z)\bar H(z')^{-1}A(y')^{-1}{H}(t')^{-1}
  \]
  at the point $((0,y,0), (0,y',0))$, we find that
  \begin{equation}
    \label{eq:Twist-horo-6-13:2}
    F((t,y,z), (t',y',z')) =
    \begin{pmatrix}
      e^{\frac{y-y'}{2}} &te^{-\frac{y-y'}{2}} -t'   e^{\frac{y-y'}{2}}   \\
      e^{-\frac{y+y'}{2}} (z-z')&e^{-\frac{y-y'}{2}}
    \end{pmatrix} + O(\Delta^2)\,.
  \end{equation}
  By choosing a smaller $\Delta_0$, if necessary, we may assume that
  all the terms $O(\Delta^2)$ appearing in the above identity are
  bounded by $\Delta/2$.

  Let $x\in \tilde C$ satisfy $d_M(x)> K_\Gamma$ and let
  \[
  \bar x=\left(
    \begin{array}{cc}
      a & b \\ 
      c & d 
    \end{array}
  \right) \in \SL(2, \R)
  \]
  be a representative for $x$ belonging to the domain $\tilde D$. We
  shall show that if $\Delta= \min (\Delta_0, \im(\bar x)^{-1}/14)$
  the function
  \[
  \alpha_x \colon (t,y,z) \in I_{\Delta} \mapsto x H(t)A(y)\bar{H}(z)
  \]
  is injective.

  Suppose, by contradiction that this is not the case. Then there
  exists distinct triplets $(t,y,z)\in I_{\Delta}$ and $(t',y',z')\in
  I_{\Delta}$ such that $\alpha_x (t,y,z) = \alpha_x (t',y',z')$.  It
  follows that the elements $\bar x_1= \bar x H(t)A(y)\bar{H}(z)$ and
  $\bar x_1'=\bar x H(t')A(y')\bar{H}(z')$ are distinct
  representatives of the same point $x_1\in M$. Since $\operatorname
  {dist}(\bar x, \bar x_1) <2$ and $\operatorname {dist}(\bar x, \bar
  x_1') <2$ and $\bar x_1\neq \bar x_1'$, by a previous observation,
  there exists $n\in \Z\setminus \{0\}$ such that 
  %\begin{equation}
  %\label{equa:gamma_nx2}
  $\bar x_1=
  \gamma_n\bar x_1',$
  %\end{equation} 
  that is such that
  \[\bar x H(t)A(y)\bar{H}(z)= \gamma_n
  \bar x H(t')A(y')\bar{H}(z')\,.\] Since
  \begin{equation}
  \label{equa:conjugate-matrix3}
  \bar x^{-1}\gamma_n \bar x=
  \begin{pmatrix}
    1 +n cd & n d^2\\ -n c^2 & 1 -ncd
  \end{pmatrix}
  \end{equation}
  the previous identity may be rewritten as
  \begin{equation*}
  \begin{pmatrix}
    1 +n cd & n d^2\\ -n c^2 & 1 -ncd
  \end{pmatrix} = F((t,y,z), (t',y',z'))\,.
  \end{equation*}
  From this identity and the identity \eqref{eq:Twist-horo-6-13:2} we
  obtain
  \begin{align*}
    \label{eq:Twist-horo-6-13:1}
    n d^2 &=te^{-\frac{y-y'}{2}} -t'   e^{\frac{y-y'}{2}}  +O(\Delta^2)\\
    -n c^2 &=e^{-\frac{y+y'}{2}} (z-z')+O(\Delta^2)
  \end{align*}
  and conclude that
  \[
  |n|\, \im(\bar x)^{-1} = |n ( c^2 + d^2 )| \le (4 e + 2)\Delta< 14
  \Delta \le \im(\bar x)^{-1}\,.
  \]
  We proved that $n=0$, reaching a contradiction. The proof of the
  Lemma is concluded by observing that by the
  inequalities~\eqref{eq:Twist-horo-6-13:3} the term $\im(\bar
  x)^{-1}$ is equivalent, up to an absolute constant depending only by
  the lattice, to the term $\exp (-d_M(x))$.
\end{proof}

\begin{proof}[Proof of Lemma~\ref{lemma:returns_loc}]
We use the same notation as in the proof of 
Lemma~\ref{lemma:fund_box}.  
Let $x \in \tilde C$, where $C$ is any cusp of the surface $S$. 
%By assumption, 
%\[
%x H(t_1) = x H(t_0) \bar H(z)\,,
%\]
%where $|z| \in \frac{\mathcal T^{-2/3}}{c_\Gamma(x, \mathcal T T)} (e^{-(\beta + 1)}, e^{-\beta}]$.  
Let $\tilde D$ be the fundamental domain of $\tilde C$ described above, 
and let $\bar x \in \tilde D$ be a representative of $x$.  

Then by definition of a $(\beta, \mathcal T, T)$-return, 
the points $\bar x H(t_1)$ and $\bar x H(t_0) \bar H(z)$ 
are two representatives for the same point $x H(t_1)$.  
Then the observation mentioned below \eqref{eq:Twist-horo-6-13:3} 
gives an integer $n$ such that 
\[
\bar x H(t_1) = \gamma_n \bar x H(t_0) \bar H(z)\,.
\] 
Notice that $n \in \Z \setminus \{0\}$, because $z \not= 0$.  

Now by \eqref{equa:conjugate-matrix3} and by multiplying on the 
right by ($H(t_0) H(z))^{-1}$, we get 
\[
  \begin{pmatrix}
    1 +n cd & n d^2\\ -n c^2 & 1 -ncd
  \end{pmatrix}
  = 
    \begin{pmatrix}
   1 - t_1 z& t_1(1+t_0z)-t_0\\ -z & 1+t_0z
  \end{pmatrix}\,.
\]
The diagonal terms give that $t_1 z =- n c d = t_0 z$.  
Because $z \neq 0$, it follows that  
\[
t_0 = t_1\,.
\]
The bottom-left entry now implies
that the point $x \exp(t_0 U)$ is a periodic 
point, with period $z = n c^2$, 
for the unstable horocycle flow $\{\bar h_t\}$.  
Lemma~\ref{lemma:returns_loc} follows from this.  
\end{proof}


\begin{thebibliography}{20}

\bibitem{Bou} J.~Bourgain, \emph{ On the maximal ergodic theorem for
    certain subsets of the integers}, Israel J. Math., \textbf{61} (1),
  1988, 39--72.

\bibitem{BuFo} A.~Bufetov \& G.~Forni, \emph{Limit Theorems for
    Horocycle Flows}, Ann. Sci. ENS, {\bf 47} (5), 2014, 851--903.
  (arXiv:1104.4502v1).
 
\bibitem{Bur} M.~Burger, \emph{Horocycle flow on geometrically finite
    surfaces}, Duke Math. J. \textbf{61}, 1990, 779--803.

\bibitem{Cartan} H.~Cartan, \emph{Th\'eorie \'el\'ementaire des
    fonctions analytiques d'une ou plusieurs variables complexes},
  Hermann, Paris, 1964.

\bibitem{Da} S.~G.~Dani.  \emph{Invariant measures and minimal sets of
    horospherical flows}.  Invent. Math. \textbf{64} (2), 1981,
  357--385.

\bibitem{FF1} L. Flaminio \& G. Forni, \emph{Invariant Distributions
    and Time Averages for Horocycle Flows}, Duke Math. J.,
  \textbf {119} (3), 2003,  465--526.

\bibitem{FF2} L. Flaminio \& G. Forni, \emph{Equidistribution of
    nilflows and applications to theta sums}.  Erg. Th. 
  Dynam. Sys., \textbf{26}, 2006, 409--433.

\bibitem{FF3} L. Flaminio \& G. Forni, \emph{On effective
    equidistribution for higher step nilflows}. Preprint:
  arXiv:1407.3640.

\bibitem{Fu} H.~Furstenberg, \emph{The unique ergodicity of the
    horocycle flow}, in Recent Advances in Topological Dynamics (New
  Haven, Conn., 1972), LNM \textbf {318}, Springer,
  Berlin, 1973, 95--115.

\bibitem{GF} I.~M.~Gelfand \& S.~V.~Fomin, \emph{Unitary
    representations of Lie groups and geodesic flows on surfaces of
    constant negative curvature} (in Russian), Dokl. Akad. Nauk SSSR
  \textbf{76}, 1951, 771--774.

\bibitem{GN} I.~M.~Gelfand \& M.~Neumark, \emph{Unitary
    representations of the Lorentz group},
  Acad. Sci. USSR. J. Phys. \textbf{10}, 1946, 93--94.

\bibitem{Go} A.~Good, \emph{Cusp forms and eigenfunctions of the
    Laplacian}, Math. Ann. \textbf{255}, 1981, 523--438.

\bibitem{Hj} D.~A.~Hejhal, \emph{On the uniform equidistribution of
    long closed horocycles}.  Loo-Keng Hua: a great mathematician of
  the twentieth century. Asian J. Math.  \textbf {4} (4), 2000,
  839--853.

\bibitem{Ma} G.~A.~Margulis. \emph{Problems and conjectures in
    rigidity theory}, Mathematics: frontiers and perspectives, 2000,
  161-174, Amer. Math. Soc., Providence, RI.

\bibitem{Pu} {L. Puk$\acute{\text{a}}$nszky}, \emph{The Plancherel
    formula for the universal covering group of SL(R, 2)}.  Math Ann.,
  \textbf{156}, 1964, 96--143.

\bibitem{Ra} M.~Ratner, \emph{The rate of mixing for geodesic and
    horocycle flows}, Erg. Th. Dynam. Sys. \textbf{7}, 1987,
  267--288.

\bibitem{Sa} P.~Sarnak, \emph{Asymptotic behavior of periodic orbits
    of the horocycle flow and Eisenstein series}, Comm. Pure
  Appl. Math. \textbf{34}, 1981, 719--739.

\bibitem{SU} P.~Sarnak, A.~Ubis, \emph{The horocycle flow at prime
    times}, Journal de math\'ematiques pures et appliqu\'ees \textbf{103} (2),
    2015,  575--618 (arXiv:1110.0777v4).

\bibitem{Sh} N.~A.~Shah, \emph{Limit distributions of polynomial
    trajectories on homogeneous spaces}, Duke Math. J. \textbf{75} (3),
  1994, 711--732.

\bibitem{St1} A.~Str\"ombergsson, \emph{On the uniform
    equidistribution of long closed horocycles}. Duke Math. J. \text{123} (3),  
    2004, 507--547.
 
\bibitem{St2} A.~Str\"ombergsson, \emph{On the deviation of ergodic
    averages for horocycle flows}, J. Modern Dynamics \textbf{7},
  2013, 291--328.

\bibitem{Su} D.~Sullivan, \emph{Discrete conformal groups and
    measurable dynamics}, Bull. Amer. Math. Soc. (N.S.) \textbf{6} (1),
  1982, 57--73.

\bibitem{T} J.~Tanis, \emph{The Cohomological Equation and Invariant
    Distributions for Horocycle Maps}.  Ergodic Theory and Dynamical
  systems, \textbf{12}, 2012, 1-42, 10.1017/etds.2012.125

\bibitem{TV} J.~Tanis \& P.~Vishe, \emph{Uniform bounds for period
    integrals and sparse equidistribution}. Preprint: arXiv:1501.05228

\bibitem{V} A.~Venkatesh.  \emph{Sparse equidistribution problems,
    period bounds and subconvexity}.  Ann. of Math. \textbf{172}, 2010,  989--1094.

\bibitem{Za} D.~Zagier, \emph{Eisenstein series and the Riemann zeta
    function}, in Automorphic Forms, Representation Theory and
  Arithmetic (Bombay, 1979), Tata Inst. Fund. Res. Studies in Math.
  \textbf{10}, Tata Inst. Fund. Res., Bombay, 1981, 275--301.
\end{thebibliography}
\end{document}